\documentclass[12pt]{article}
\usepackage{e-jcCHANGE}

\usepackage{amsthm,amsmath,amssymb}

\usepackage{graphicx}

\usepackage[colorlinks=true,citecolor=black,linkcolor=black,urlcolor=blue]{hyperref}

\numberwithin{equation}{section}
\theoremstyle{plain}
\newtheorem{theorem}{Theorem}[section]
\newtheorem{lemma}[theorem]{Lemma}
\newtheorem{corollary}[theorem]{Corollary}
\newtheorem{proposition}[theorem]{Proposition}
\newtheorem{fact}[theorem]{Fact}

\newtheorem{comment}[theorem]{Comment}

\theoremstyle{definition}

\theoremstyle{remark}
\newtheorem{remark}[theorem]{Remark}

\title{\bf Asymptotic distribution of the numbers of vertices and arcs of the giant strong component in sparse random digraphs}

\author{Boris Pittel\thanks{The authors gratefully acknowledge support from NSF grant \# DMS-1101237.}\\
\small Department of Mathematics\\[-0.8ex]
\small The Ohio State University\\[-0.8ex] 
\small Columbus, Ohio, U.S.A.\\
\small\tt bgp@math.osu.edu\\
\and
Daniel Poole\footnotemark[1] \\
\small Department of Mathematics\\[-0.8ex]
\small The Ohio State University\\[-0.8ex]
\small Columbus, Ohio, U.S.A.\\
\small\tt poole@math.osu.edu
}

\begin{document}

\maketitle


\begin{abstract}
Two models of a random digraph on $n$ vertices, $D(n,\text{Prob}(\text{arc})=p)$ and
$D(n,\newline\text{number of arcs }=m)$ are studied. 
In 1990,  Karp for $D(n,p)$ and independently T. \L uczak for $D(n,m=cn)$  proved that for $c>1$, with probability tending to $1$, there is an unique strong component of size of order $n$. Karp showed, in fact, that the giant
 component has likely size asymptotic to $n\theta^2$, where $\theta=\theta(c)$ is the unique positive root of $1-\theta=e^{-c \theta}$. In this paper we prove that, for both random digraphs,
 the joint distribution of the number of vertices and number of arcs in the giant strong component is asymptotically Gaussian with the same mean vector $n\boldsymbol{\mu}(c)$, $\boldsymbol{\mu}(c):=(\theta^2, c\theta^2)$ and two distinct $2\times 2$ covariance matrices,
$n\mathbf{B}(c)$ and $n\bigl[\mathbf{B}(c)+c (\boldsymbol{\mu}'(c))^T (\boldsymbol{\mu}'(c)))
\bigr]$. To
this end, we introduce and analyze a randomized deletion process which determines  the directed $(1,1)$-core, the maximal digraph with minimum in-degree and out-degree at least 1. This $(1,1)$-core contains all non-trivial strong components.  However, we show that the likely numbers of  peripheral vertices and arcs in the $(1,1)$-core, those outside the largest strong component,  are 
of polylog order,  thus dwarfed by anticipated fluctuations, on the scale of $n^{1/2}$,  of the giant
component parameters. By approximating the likely realization of the deletion algorithm with a deterministic trajectory, we obtain our main result via exponential supermartingales and Fourier-based techniques.

\end{abstract}

\section{Introduction}\label{sec: intro}

\subsection{Definitions and Main Results}\label{sec: refs and main results}

In this paper, we will be studying the random digraphs $D(n,m)$ and $D(n,p)$. $D(n,m)$ is uniformly distributed on all digraphs with vertex set $[n]=\{1,2,\ldots, n\}$ and $m=m(n)$ arcs. The (Bernoulli) digraph $D(n,p)$ is a random digraph on $[n]$, where each of the $n(n-1)$ possible arcs is present with probability $p=p(n)$ independently of all other arcs. As customary, we say that for a given $m$ ($p$ resp.) some graph property holds for $D(n,m)$ ($D(n,p)$ resp.) {\it asymptotically almost surely}, denoted a.a.s.,\ if the probability that $D(n,m)$ ($D(n,p)$ resp.) has this property tends to 1 as $n \to \infty.$ A digraph is {\it strongly connected} if for any pair of vertices $v$ and $w$, there is a path from $v$ to $w$ and a path from $w$ to $v$. A {\it strong component} of a digraph is a maximal strongly connected subgraph. 

The phase transition in $D(n,p)$ and in $D(n,m)$ was established by Karp~\cite{karp} and T. \L uczak~\cite{luczak}, respectively. For instance, if $\lim np= c<1$, then a.a.s.\ the number of vertices in the largest strong component is bounded in probability, and if $\lim np= c>1$, then a.a.s.\ there is a strong giant component on $(\theta(c)^2+o(1))n$ vertices, where $\theta$ is the unique root in $(0,1)$ of $1-\theta = e^{-c \theta}.$ In the latter case, all other strong components have size bounded in probability. Later, T. \L uczak and Seierstad~\cite{luczak seierstad} investigated the size of the strong components for the cases $np=1 \pm \epsilon$, where $\epsilon \to 0$, but  $n^3 \epsilon \to \infty.$ In particular, for $np=1+\epsilon$, they demonstrated that a.a.s.\ the size of the giant component is $(\theta(np)^2+o(\epsilon^2))n$. Our main result is the asymptotic joint distribution of the numbers of vertices and arcs in the strong giant component,  both in $D(n,m=c_nn)$ and $D(n,p=c_n/n)$. Throughout the paper $c_n$ satisfies the condition $\lim c_n=c\in (1,\infty)$.

\begin{theorem}\label{thm 1}
{\bf (i)} Let $V_1 ,A_1$ denote the vertex set and arc set of the largest strong component of  $D(n,m)$. Suppose $m=c_n n$. Denote $\theta_n=\theta(c_n)$. Then there exists a continuous
$2 \times 2$ matrix $\mathbf{B}=\mathbf{B}(c)$ such that 
\begin{equation*}
\left( \frac{|V_1| - \theta_n^2 n}{n^{1/2}}, \frac{|A_1| - c_n \theta_n^2 n}{n^{1/2}} \right) \overset{d}{\implies} \mathcal{N}(\mathbf{0}, \mathbf{B}),
\end{equation*}
where $\mathcal{N}(\mathbf{0}, \mathbf{B})$ is the $2$-dimensional Gaussian distribution with mean vector $\mathbf{0}$ and covariance matrix $\mathbf{B}$.

{\bf (ii)} Let $V_1^p ,A_1^p$ denote the vertex set and arc set of the largest strong component of $D(n,p)$. Suppose $p=c_n/n$. Denote $\boldsymbol{\mu}(c)=(\theta^2(c), c \theta^2(c) )$. Then for the $2 \times 2$ matrix $\mathcal{B}(c)=\mathbf{B}(c)+ c (\boldsymbol{\mu}'(c))^T (\boldsymbol{\mu}'(c))$, we have that 
\begin{equation*}
\left( \frac{|V_1^p| - \theta_n^2 n}{n^{1/2}}, \frac{|A_1^p| - c_n \theta_n^2 n}{n^{1/2}} \right) \overset{d}{\implies} \mathcal{N}(\mathbf{0}, \mathcal{B}),
\end{equation*}
where $\mathcal{N}(\mathbf{0}, \mathcal{B})$ is the $2$-dimensional Gaussian distribution with mean vector $\mathbf{0}$ and covariance matrix $\mathcal{B}$.
\end{theorem}

To prove Theorem \ref{thm 1}, we develop and analyze a randomized deletion algorithm for
$D(n,m=c_n n)$. In
steps, we successively delete {\it semi-isolated} vertices, those  with either in-degree zero or out-degree zero. The terminal digraph delivered by the algorithm is the $(1,1)$-core, the maximal subgraph with minimum in-degree and out-degree at least 1. The core contains the strong giant component, and conjecturally, a.a.s.,\ the rest of the core  has size 
negligible relative to random fluctuations, of order $n^{1/2}$, of the core parameters. Guided by this intuition, we zero in  on the asymptotic distribution of the number of vertices and number of arcs in the $(1,1)$-core, rather than the strong giant component itself. Assuming the limit distribution is Gaussian, we determine parameters of this Gaussian distribution via approximating the actual realization of the deletion process by a deterministic system of partial differential equations. Once these parameters are determined, we prove that the Fourier transform of the actual $2$-dimensional vector of the number of vertices and arcs in the $(1,1)$-core of
$D(n,m=c_n n)$ does indeed converge, pointwise, to the Fourier transform of $\mathcal{N}(\mathbf{0}, \mathbf{B})$, implying part \textbf{(i)} of Theorem \ref{thm 1}. The part \textbf{(ii)} follows then immediately 
by using the approach from Pittel~\cite{pittel 90}. To finish the proof of Theorem \ref{thm 1}, we justify our above conjecture showing that in fact the difference between the $(1,1)$-core and the strong giant component has size of polylog order (Theorem \ref{thm 2}). 

In the next subsections, we relate our study to the known distributional results
for the random {\it undirected\/} graphs. We also provide a preliminary insight into the salient points of our argument for the directed graphs, in the hope that it will serve as a helpful guide through the detailed proofs that follow.

\subsection{Similarities to, and distinctions from the undirected case}

Let $V_1(G)$ denote the vertex set  of the largest component of undirected graph $G$.  Back in $1970$, Stepanov~\cite{step 87} proved that if $p=c/n,$ $c>1$, then $|V_1(G(n,p))|$, properly centered and scaled, is asymptotically normal. Twenty years later, Pittel~\cite{pittel 90} proved, for $G(n,m=cn/2)$, $c>1$, a functional limit theorem for the counts of  of tree components of all various sizes, and used this theorem to prove asymptotic normality of $|V_1|$ for $G(n,m=cn/2)$. Stepanov's result for $G(n,p=c/n)$ followed then without
much effort. Subsequently, Pittel and Wormald~\cite{pittel wormald} 
found an asymptotic formula for the count of sparse connected $2$-cores with given numbers of vertices and edges. They used this formula to prove, for $G(n,m)$,  the $3$-dimensional {\it local\/} Gaussian limit
theorem, whence the integral limit theorem, for the three leading parameters: the number of vertices and the number of edges in the $2$-core of the giant component, and the total size of trees rooted at the core vertices. The counterpart of this $3$-dimensional limit theorem for
$G(n,p)$ followed via the method in \cite{pittel 90}.  Ding, Kim, Lubetzky and Peres~\cite{DKLP}
obtained sharp asymptotic results for the ``young'' giant component, i.e.\  for $p=(1+\epsilon)/n$, ($\epsilon\to 0$, $\epsilon^3n\to\infty$). Extending to hypergraphs, Karo\'nski and 
T. \L uczak~\cite{karluc},  Behrisch, Coja-Oghlan and Kang~\cite{order of hyper}, \cite{size and order of hyper},  Bollob\'as and 
Riordan~\cite{BolRiorhyp} proved versions of  the central limit theorem for the joint distribution of the number of vertices and number of hyperedges in the largest component of the random $d$-uniform hypergraph models $H_d(n,p)$ and $H_d(n,m),$ for $p=c (d-1)!/n^{d-1}$ and $m=cn/d$, where $c>1/(d-1)$. Recently, Seierstad~\cite{seierstad} found that for a family of random graph processes, the order of the giant component is asymptotically normal, provided certain general conditions are met. 
This account is by no means complete.  There have been
obtained fine asymptotic results regarding such parameters of the giant component
as its diameter and the mixing time of the walk on the giant component close to its inception; see, for instance, ~\cite{DKLP}, and further references therein.

Recently, there has been interest in re-deriving results about the size of the giant component by analyzing ``exploration" search processes which determine all the components in a graph. By analyzing a depth-first search version in $G(n, p)$, Krivelevich and Sudakov~\cite{krivelevich sudakov} found a simple proof of the phase transition. In fact, they found that for $p=c/n, c>1$, a.a.s.\ $G(n, p)$ contains a path of linear length. Nachmias and Peres~\cite{nachmias peres} analyzed a similar exploration process to rederive Bollob\'{a}s'~\cite{bollobas epsilon} and \L uczak's~\cite{luczak epsilon} concentration results about the size of young giant component. Barraez, Boucheron and De La Vega~\cite{barraez} and Bollob\'{a}s and Riordan~\cite{BolRiorgraph} proved that the size of the giant component in $G(n, p=c/n),$ $c>1$, is asymptotically normal using exploration processes.

In light of this progress, lack  of  distributional results for the strong giant component in $D(n,m)$ and $D(n,p)$ seems rather striking. In fact, just counting strongly connected digraphs had
been an open challenge. Even though Bender, Canfield and McKay~\cite{BCM}
were already able in 
$1992$ to determine the asymptotic count of connected undirected graphs with given
numbers of vertices and edges, a counterpart of their remarkable formula for the strongly
connected digraphs was obtained only very recently, see
P\'{e}rez-Gim\'{e}nez and Wormald~\cite{perez wormald}, Pittel~\cite{pittel counting strong}.

Perhaps one of the reasons  for this disparity is that determining the strong components of a digraph is algorithmically more difficult than finding the components of a graph. For a digraph, the component notion morphs into two, harder-to-handle, dual notions of a sink-set and a source-set, the subsets of vertices with no arc going outside, and no arcs coming from outside, respectively.
A digraph is strongly connected precisely when the full vertex set is the only source-set and the only sink-set. As a consequence, finding the strong component containing some generic vertex $v$ depends on ``global" information, i.e.\ this determination requires information about the digraph parts possibly quite distant from $v$. 
If we try to determine the strong component containing $v$ by either a depth-first or breadth-first search, we apparently would have to search for vertices of $2$ distinct types, the vertices that can reach $v,$ along a directed path, and the vertices that can be reached from $v$. However
the size of the
intersection of two sets in the random digraph, in either $D(n,p)$ or $D(n,m)$,  is hard to access
since the two search processes are {\it interdependent\/}. We clearly need 
to find a middle-ground search process which would provably deliver a close {\it approximation\/} to the giant component, without us having to deal with this nasty interdependence. 

We consider the following deletion algorithm on a digraph $D$. First, we delete all isolated vertices, those with both in-degree zero and out-degree zero, from $D$ obtaining $D(0)$. 
Recursively, if $t\ge 0$ is such that $D(t)$ does not have any semi-isolated vertices, then 
the deletion process stops, and  we define $D(s)\equiv D(t)$, $s>t$.  If $D(t)$ does have semi-isolated vertices, then
\begin{itemize}
\item First, we delete a semi-isolated vertex, chosen uniformly at random among all semi-isolated vertices, along with its incident arcs from $D(t)$. 
\item Second, we delete all newly isolated vertices, and set $D(t+1)$ equal to the remaining
subdigraph of $D(t)$.
\end{itemize}
Let $\bar{\tau}$ be the first moment that $D(t)$ does not have any semi-isolated vertex; so  $D(t)=D(\bar{\tau})$ for all $t\geq \bar{\tau}$.  The terminal $D(\bar{\tau})$ is both the $(1,1)$-core of $D$ and the $(1,1)$-core of all the digraphs $D(t).$  If $D$ has non-trivial strong components, the largest strong component is contained within the $(1,1)$-core. More precisely, the $(1,1)$-core is comprised of {\em all} non-trivial strong components along with directed paths between these components. Our key result is that for $D(n,m=c_n n)$, $c:=\lim c_n\in (1,\infty)$,  a.a.s.\ there are not many vertices and arcs which are in the $(1,1)$-core but not in the largest strong component. 

\subsection{Switching to the core} 

The following theorem allows us to switch from finding the number of vertices and arcs in the strong giant component to finding those numbers of vertices and arcs in the $(1,1)$-core. 

\begin{theorem}\label{thm 2}
Let $V_{1,1},\, A_{1,1}$ denote the vertex set and the arc set of the $(1,1)$-core of $D(n,m=c_n n)$. Then a.a.s.\ 
\begin{align*}
0 &\leq |V_{1,1}| - |V_1| \leq 2(\ln n)^8, \\  0 & \leq |A_{1,1}|-|A_1| \leq 4 (\ln n )^{10}.
\end{align*}
\end{theorem}

Theorem \ref{thm 2} calls to mind an observation that most sparse digraphs, with
minimum in-out degree $1$ at least,  provably contain a massive strong component.
This was a key ingredient in derivation of asymptotic counts of strongly connected digraphs in
\cite{perez wormald} and in \cite{pittel counting strong}.  Cooper and Frieze~\cite{cooper frieze} used a similar property for a random digraph with a {\it given\/} degree sequence. The proof of Theorem \ref{thm 2} extends to a exploration process, the {\it full} depth-first search, the following observation due to Karp~\cite{karp}: the size of the descendant set of a generic vertex in $D(n, p)$ has the same distribution as the size of the component containing a generic vertex in $G(n, p)$. Karp's observation has been gainfully used before, see Biskup, Chayes and Smith~\cite{biskup}. We will prove
Theorem \ref{thm 2} in the last Section \ref{sec: rest} since the argument does not require
the properties of the deletion process. Next comes

\subsection{Finding the core}
Next comes
\begin{theorem}\label{thm 3}
There is a continuous $2 \times 2$, positive-definite, matrix $\mathbf{B}=\mathbf{B}(c)$ such that
\begin{equation*}
\left( \frac{|V_{1,1}| - \theta^2_n n}{n^{1/2}}, \frac{|A_{1,1}| - c_n \theta^2_n n}{n^{1/2}} \right) \overset{d}{\implies} \mathcal{N}(\mathbf{0}, \mathbf{B}),
\end{equation*}
where $\mathcal{N}(\mathbf{0}, \mathbf{B})$ is the 2 dimensional Gaussian distribution with mean $\mathbf{0}$ and covariance matrix $\mathbf{B}$.
\end{theorem}
\noindent
In particular, both  $|V_{1,1}|$ and $|A_{1,1}|$ undergo random fluctuations of order $n^{1/2}$
around $\theta^2_n n$ and $c_n \theta^2_n n$ respectively, with $n^{1/2}\gg (\ln n)^{11}$,
the likely bound of the error-approximations of $|V_1|$ by $|V_{1,1}|$, and of $|A_1|$ by $|A_{1,1}|$
in Theorem \ref{thm 2}. Thus Theorem \ref{thm 3} combined with Theorem \ref{thm 2} imply Theorem \ref{thm 1}{\bf (i)}.

The $(1,1)$-core is a natural counterpart  of the $k$-core in {\it undirected} graphs, see Bollob\'{a}s~\cite{bollobas 84}.  The {\it $k$-core} of a graph is the maximal subgraph with minimum degree at least $k$. Pittel, Spencer and Wormald~\cite{pittel spencer wormald} introduced a randomized deletion algorithm which terminates with the $k$-core and used it to identify the phase transition window of width $n^{1/2 +\epsilon}$ around
an explicit  threshold value $c_kn$ of number of edges necessary for a.a.s.\  existence of  a $k$-core ($k \geq 3$), as well as to establish the likely concentration of the $k$-core size
within $n^{1/2+\epsilon}$ distance from its expected value.  Later, Janson and M. Luczak~\cite{janson luczak} undertook a distributional analysis of  this deletion algorithm and proved that,
for sparse $G(n,m)$, the size of the $k$-core is asymptotically normal with standard deviation
of order $n^{1/2}$. They also demonstrated that the random moment when the Erd\H os-R\'enyi  graph process $\{G(n,m)\}$ develops a $k$-core is asymptotically normal, again
with standard deviation of order $n^{1/2}$. 

To prove Theorem \ref{thm 3}, we analyze the likely behavior of the deletion process for
finding the $(1,1)$-core in the directed graph $D(n,m)$. In part, our approach has certain semblance with
investigation of randomized deletion processes for the $k$-core problem  
in \cite{pittel spencer wormald}, and for the Karp-Sipser greedy matching algorithm  in Aronson, Frieze and Pittel~\cite{aronson}.
 
By construction, the deletion process is obviously Markovian, but prohibitively hard to analyze due to the enormous number of states. Fortunately it is possible to aggregate these states into equivalence classes preserving the Markovian nature of the process. Namely,  we introduce the process $\{\mathbf{s}(t)\}$, 
\begin{equation}\label{eqn: s(t) def}
\mathbf{s}(t)=(\nu(t), \nu_i(t), \nu_o(t), \mu(t));
\end{equation} 
here $\nu(t)$ is the number of vertices, $\nu_i(t)$ is the number of vertices with zero in-degree, $\nu_o(t)$ is the number of vertices with zero out-degree, and $\mu(t)$ is the number of arcs of $D(t)$. Clearly our  task is to determine the asymptotic distribution of $\nu(t)$ and $\mu(t)$ at 
\begin{equation*}
\bar{\tau}=\min\{t: \nu_i(t)=0,\,\nu_{o}(t)=0\}, 
\end{equation*}
the first moment $t$ when there are no semi-isolated 
vertices left. In Section \ref{sec: greedy deletion}, we show that  $\{\mathbf{s}(t)\}$ indeed remains Markovian. The price for lumping together various digraphs with the same foursome $\mathbf{s}$ 
is that we have to determine decidedly more involved transition probabilities. For instance, it is necessary to evaluate $g(\mathbf{s})$, the number of digraphs with a generic parameter $\mathbf{s}$.  In Section \ref{sec: asymptotic trans prob}, following Pittel~\cite{pittel counting strong}, we asymptotically evaluate $g(\mathbf{s})$ for a wide range of $\mathbf{s}$, via an argument based on McKay's asymptotic formula~\cite{mckay1} for the number of $(0,1)$-matrices with given row and column sums.

In Section \ref{sec: char fun}, we introduce the joint characteristic function of $\nu(\bar{\tau}), \mu(\bar{\tau})$, 
$$
\varphi_{\mathbf{s}}(\mathbf{u})=E\bigl[\exp\bigl(i\bold u^T\bigl(\nu(\bar{\tau}), \mu(\bar{\tau})\bigr)\bigr)\bigr],\quad \bold u\in \Bbb R^2,
$$
 for the deletion process that starts from a generic initial state $\mathbf{s}$. 
Due to $\{\mathbf{s}(t)\}$ being Markovian, $\varphi_{\mathbf{s}}(\mathbf{u})$ satisfies an equation 
\begin{equation}\label{recur,phi}
\varphi_{\mathbf{s}}(\mathbf{u})=E_{\mathbf{s}}
\bigl[\varphi_{\mathbf{s}^\prime}(\mathbf{u})\bigr], 
\end{equation}
$\mathbf{s}^\prime$ standing for the
random {\it next\/} state. It had been shown (e.g.\ Pittel~\cite{normal convergence}, with
Laplace, rather than Fourier transform), that an equation of this kind  can be used to establish asymptotic normality in the cases when the
mean and the variance of the random parameter in question are (almost) linear in $n$, even
when no representation of this parameter as a sum of weakly dependent, uniformly
negligible, terms is forthcoming. Expecting that the mean and the covariance of
 $(\nu(\bar{\tau}), \mu(\bar{\tau}))$ are indeed linear in $n$, we hope to approximate $\varphi_{\mathbf{s}}(\mathbf{u})$ by a Gaussian characteristic function, $G_n(\mathbf{s}/n, \mathbf{u})$, with a  mean $n \mathbf{f}(\mathbf{s}/n)=n[f_1(\mathbf{s}/n), f_2(\mathbf{s}/n)]^T$ and and covariance matrix $n \boldsymbol{\psi}(\mathbf{s}/n)=n\{\psi_{j,k}(\mathbf{s}/n)\}$, i.e.\ being dependent on the initial state $\mathbf{s}$. Explicitly we set 
\begin{equation*}
G_n(\mathbf{s}/n, \mathbf{u}) = \exp \left( i\,n \sum_{j=1}^2 u_j f_j(\mathbf{s}/n) - \frac{n}{2} \sum_{j,k=1}^2 u_j u_k \psi_{j,k}(\mathbf{s}/n) \right).
\end{equation*}
We want to show that $\bigl|\varphi_{\mathbf{s}}(\mathbf{u})-G_n(\mathbf{s}/n,\mathbf{u})\bigr|\to 0$, at least for ``good'' $\mathbf{s}/n$, those among  the likely values of $\mathbf{s}(0)$ arising from $D(n,m)$. Assuming smoothness of  $\mathbf{f}( \cdot)$ and $\boldsymbol{\psi}(\cdot)$ for good values of the argument, we wish to determine $\mathbf{f}( \cdot)$ and $\boldsymbol{\psi}(\cdot)$ out of the condition that  $G_n$ {\it nearly\/} satisfies the recurrence relation, i.e.\ within the
 additive term $o(n^{-1})$. This condition leads us to the system of first
order PDE for $f_i$ and $\psi_{j,k}$. The method of characteristics reduces the PDE to
a system of ODEs, whose solution is anticipated a.a.s.\  to be close to the random
$\{\mathbf{s}(t)\}$. In Section \ref{sec: solving diff eqn}, we solve this system of equations, in explicit form for $\mathbf{f}$,  and in integral form for $\boldsymbol{\psi}$. We do so by identifying
two explicit integrals, i.e.\ functions of the current state, that remain constant along
the characteristics of the ODE. 

In Section \ref{sec: large deviation}, we introduce a pair of exponential super-martingales
constructed from those two integrals of the ODE system and use them to show  that until 
 the end of the process, $\{\mathbf{s}(t)\}$ stays  close to the deterministic trajectory,
provided that $\mathbf{s}=\mathbf{s}(0)$ is ``slightly better'' (even closer to expected value of $\mathbf{s}(0)$ in $D(n, m)$)
than just good. Therefore, for those initial
$\mathbf{s}$, $G_n(\mathbf{s}(t)/n,\mathbf{u})$ nearly satisfies \eqref{recur,phi} for all $t\le \bar{\tau}$.
So, conditioned on a better-than-good initial state $\mathbf{s}(0)$,   $(\nu(\bar{\tau}),\mu(\bar{\tau}))$ is asymptotically Gaussian, with mean $n \bold f(\mathbf{s}(0)/n)$ and covariance matrix
$n\boldsymbol{\psi}(\mathbf{s}(0)/n)$.

However, our ultimate goal is to determine the asymptotic distribution of $(\nu(\bar{\tau}),\mu(\bar{\tau}))$ for the {\it random\/} $\mathbf{s}(0)$ in $D(n, m=c_n n)$. In Section \ref{sec: asymp dist of s}, we determine the asymptotic distribution of this random $\mathbf{s}(0)$, and use it together
with the limiting distribution of $(\nu(\bar{\tau}),\mu(\bar{\tau}))$ conditioned on a generic $\mathbf{s}(0)$
to prove asymptotic normality of the terminal pair $(\nu(\bar{\tau}), \mu(\bar{\tau}))$ for $D(n,m)$ as the starting state
of the deletion process. We will find that the random fluctuations of $\mathbf{s}(0)$ have no 
influence on the limiting means, but  have a discernible effect on the limiting covariance matrix.

\subsection{Description of the mean and covariance parameters as $c\downarrow1$}

Although the entries of $\mathbf{B}(c)$ are in integral form, we can say much more about these entries for $c$ close to 1, which we detail in Section \ref{sec: special case}. The formulas
are particularly simple for the pair  $(|V_1|, \text{Exc}_1)$, where $\text{Exc}_1:=|A_1|-|V_1|$ is the {\it excess} of the largest strong component. For $D(n, m=cn)$ and $D(n, p=c/n)$, this pair is asymptotically Gaussian with mean $n (\theta^2, (c-1)\theta^2)$ with covariance matrices $n \tilde{\mathbf{B}}$ and $n\tilde{\mathcal{B}}$, which are determined from $\mathbf{B}$ and $\mathcal{B}$ respectively. For $c-1=\epsilon \downarrow 0,$ we have that $\theta=2\epsilon+O(\epsilon^2)$ and both $\tilde{\mathbf{B}}$ and $\tilde{\mathcal{B}}$ are  
\begin{equation*}
\left( \begin{array}{cc} 40 \epsilon + O(\epsilon^2) & 60 \epsilon^2 + O(\epsilon^3) \\ 60 \epsilon^2+O(\epsilon^3) & \frac{272}{3}\epsilon^3+O(\epsilon^4) \end{array} \right).
\end{equation*}
Qualitatively this is similar to the covariance matrix of $(\text{size of }2\text{-core}, \text{excess of }2\text{-core})$ in $G(n, m=cn/2)$ and $G(n, p=c/n),$ $c>1$, see Pittel and 
Wormald~\cite{pittel wormald}. Note though that, unlike our present set-up, the formulas in \cite{pittel wormald}
were established under a weaker condition, $\epsilon=\epsilon(n) \gg n^{-1/3}$. We conjecture that the
pair $(\nu(\bar{\tau}),\mu(\bar{\tau}))$ is also asymptotically Gaussian for 
$\epsilon\gg n^{-1/3}$.

\begin{comment}
Throughout this paper, all unspecified limits are always with respect to $n \to \infty$. 
\end{comment}

\section{Deletion Process}\label{sec: greedy deletion}

By construction, the process $\{D(t)\}$ is clearly Markovian. In this section, our goal is to show that the simpler process $\{\mathbf{s}(t)\},$ defined in \eqref{eqn: s(t) def}, is Markovian as well.  The proof below uses as a template the reduction argument in Aronson, Frieze and Pittel~\cite{aronson} for
the Karp-Sipser~\cite{karpsipser} greedy matching algorithm.

We start with a few definitions. Let $D=(V,E)$ be a digraph. A vertex $w$ is a {\it descendant} of $v$ if either $w=v$ or
there is a directed path from $v$ to $w$. We call $w$ is a {\it direct} descendant of $v$ if $(v,w)\in
E$.  Dually, we say that $w$ is an {\it ancestor} of $v$ if either $w=v$ or there is a directed path from $w$ to $v$, with $w$ being a direct ancestor if $(w,v)\in E$. Let $\mathcal{S}:= (V,\mathcal{O}_{i},\mathcal{O}_{o},\mu)$ denote the foursome  composed of the vertex set of $V$, the set
of vertices of in-degree zero, the set of vertices of out-degree zero, and the {\it number\/} of arcs. 

Let us show first that $\{\mathcal{S}(t)\}$ is itself a Markov chain. 
While at the first step we delete all isolated vertices, at every other step we begin with a digraph without isolated vertices and deliver its sub-digraph without isolated vertices. 

{\bf Substep 1.} Choose a semi-isolated vertex (uniformly at random) and delete this vertex along with all incident arcs, obtaining an intermediate digraph $D^*$ with parameter $\mathcal{S}^*=(V^*, \mathcal{O}_i^*, \mathcal{O}_o^*, \mu^*)$,  along with $I^*: = \mathcal{O}^*_{i} \cap\mathcal{O}^*_{o}$ being the set of newly-born isolated vertices in $D^*$.

{\bf Substep 2.} Delete the vertices in $I^*$ from $D^*$ obtaining a digraph, $D',$ without isolated vertices, with parameter $\mathcal{S}'=(V', \mathcal{O}_i', \mathcal{O}_o', \mu')$.

The probability of a specific semi-isolated vertex being chosen is thus $1/(|\mathcal{O}_i|+|\mathcal{O}_o|)$.

We say that $\mathcal{S}'$ {\it can follow} from $\mathcal{S}$ if a digraph with parameter $\mathcal{S}'$ can be obtained from some digraph with parameter $\mathcal{S}$ after one step of the deletion algorithm. 

\begin{proposition}
For $\mathcal{S}'$ to be able to follow $\mathcal{S}$ there must exist  vertex sets $A \subset \mathcal{O}_i,$ $B \subset \mathcal{O}_o,$ and $R_i, R_o \subset V \setminus (\mathcal{O}_i \cup \mathcal{O}_o)$ such that
\begin{align*}
V' = V \setminus (A \cup B), \quad \mathcal{O}_i' = (\mathcal{O}_i \setminus A) \cup R_i, \quad \mathcal{O}_o' = (\mathcal{O}_o \setminus B) \cup R_o, \quad \mu' = \mu-k, 
\end{align*}
where (i) at least one of $A$ and $B$ have cardinality $1$, (ii) at least one of $R_i$ and $R_o$ are empty, and (iii) $k \geq \max \{ |A|+|R_o|, |B|+|R_i| \}.$ Furthermore if $\mathcal{S}'$ can follow $\mathcal{S}$, the sets $A$ and $B$ are uniquely determined by $(\mathcal{S}, \mathcal{S}')$. 
\end{proposition}

\begin{proof}
For certainty, suppose we delete a vertex, $v$, with {\it in}-degree zero in substep 1.  Each vertex other than $v$ has the same out-degree in $D^*$ as it does in $D$, so the vertices with out-degree zero stay the same, i.e.\ $\mathcal{O}_{o}= \mathcal{O}_{o}^*$. However, vertices from $\mathcal{O}_{o}$, whose only direct ancestor in the original digraph $D$ is the deleted vertex, $v,$ now also have zero out-degree, and hence are isolated in the intermediate digraph $D^*$. In fact, all isolated vertices of $D^*$ are born this way. Non-semi-isolated vertices of $D$ join $\mathcal{O}_{i}^*$ if their only direct ancestor in $D$ is $v$. In the second substep, we delete the isolated vertices, $I^*$, from the digraph to obtain $D'$. In particular $V' = V^* \setminus I^*, \mathcal{O}_{i}' = \mathcal{O}_{i}^* \setminus I^*, \mathcal{O}_{o}'=\mathcal{O}_{o}^* \setminus I^*, \mu'=\mu^*$. 

At the end, a vertex $v$ is deleted from $\mathcal{O}_{i}$, along with vertices, $B$, from $\mathcal{O}_{o}$ whose only direct ancestor was $v$, so that $\mathcal{O}_{o}'=\mathcal{O}_{o}\setminus B$. Also, the set $R_i$ of vertices which in $D$  have $v$ as their only direct ancestor, now have in-degree zero; so $\mathcal{O}_i' = \left(\mathcal{O}_i\setminus\{v\}\right)\cup R_i$. The number of arcs must decrease by at least $\max\{1, |B| + |R_i|\}$;  since $D$ had no isolated vertices,  $v$ had at least one direct ancestor, and so at least one arc is deleted in Substep 1.  In particular, we have $V'=V\setminus (\{v\}\cup B)$, $\mathcal{O}'_i=(\mathcal{O}_i\setminus\{v\})\cup R_i$, $\mathcal{O}'_o=\mathcal{O}_o\setminus B$, and $\mu-\mu' \geq |B|+|R|$. There is a similar description if 
in Substep 1 we delete a vertex of zero {\it out}-degree.
\end{proof}

To proceed, let $\mathcal{D}_{\mathcal{S}}$ denote the set of all digraphs with parameters $\mathcal{S}$. For $D' \in \mathcal{D}_{\mathcal{S}'},$ let $N_{\mathcal{S},\mathcal{S}'}^{(in)}(D')$\label{def: NSS'} denote the number of digraphs $D \in \mathcal{D}_{\mathcal{S}}$ such that $D'$ can be obtained from $D$ via one step of this deletion algorithm by initially deleting a vertex with in-degree zero. Similarly, let $N_{\mathcal{S},\mathcal{S}'}^{(out)}(D')$ denote the number of digraphs, $D$, such that $D'$ can be obtained after deleting a vertex with zero out-degree. 

\begin{lemma} Both $N_{\mathcal{S},\mathcal{S}'}^{(in)}(D')$ and $N_{\mathcal{S}, \mathcal{S}'}^{(out)}(D')$ depend only on $\mathcal{S}$ and $\mathcal{S}'$, i.e.\ there exist 
$N_{\mathcal{S},\mathcal{S}'}^{(in)}$ and $N_{\mathcal{S}, \mathcal{S}'}^{(out)}$ such that
\begin{equation*}
N_{\mathcal{S},\mathcal{S}'}^{(in)}(D') \equiv N_{\mathcal{S},\mathcal{S}'}^{(in)},\quad
N_{\mathcal{S}, \mathcal{S}'}^{(out)}(D')\equiv N_{\mathcal{S}, \mathcal{S}'}^{(out)}, \qquad\,
\forall\, D' \in \mathcal{D}_{\mathcal{S}'}. 
\end{equation*} 
\end{lemma}
\begin{proof}
Consider, for instance, $N_{\mathcal{S}, \mathcal{S}'}^{(in)}(D')$. First of all,
$N_{\mathcal{S},\mathcal{S}'}^{(in)}(D') \equiv 0$, unless $\mathcal{S}$ and $\mathcal{S}'$ are such that $|A|=1, |T|=0$ and $\mu-\mu' \geq \max\{1, |B|+|R_i|\}$. In the latter case let $v$ be
the single vertex in $A$. For any starting digraph $D$, $v$ is necessarily the vertex deleted in 
Substep $1$. Furthermore,  in $D$, each vertex in $B\cup R_i$ has $v$ as its only direct
ancestor;  so $v$ must have $|B|=|\mathcal{O}_o|-|\mathcal{O}'_o|$ arcs ending at vertices of $B$ and $|R_i|=|\mathcal{O}'_i|-|\mathcal{O}_i|+1$ arcs ending at vertices of $R_i$. Other arcs incident to $v$ must end at a vertex in $V\setminus (\mathcal{O}_i\cup R_i \cup B).$ Hence any such $D$ can be reconstructed by distributing these additional arcs among the vertices of $V\setminus (\mathcal{O}_i \cup R_i \cup B).$ In particular,
\begin{equation*}
N_{\mathcal{S},\mathcal{S}'}^{(in)}(D^\prime) =
\begin{cases}
 { |V|  - |\mathcal{O}_{i}|-|B| - |R_i|  \choose \mu - \mu' - |B| - |R_i| }, & \text{if } |A|=1,|R_o|=0, \mu-\mu' \geq \max\{1, |B|+|R_i|\}, \\
0, & \text{otherwise},
\end{cases}
\end{equation*}
implying that  $N_{\mathcal{S},\mathcal{S}'}(D^\prime)$ is indeed the same for all $D' \in \mathcal{D}_{\mathcal{S}'}$. Likewise, one can show that
\begin{equation*}
N_{\mathcal{S},\mathcal{S}'}^{(out)}(D^\prime) =
\begin{cases}
{ |V| - |\mathcal{O}_{o}|-|A| - |R_o|  \choose \mu - \mu' - |A| - |R_o| }, & \text{if } |B|=0,|R_i|=0, \mu-\mu' \geq \max \{1, |A|+|R_o|\}, \\ 0, & \text{otherwise}.
\end{cases}
\end{equation*}
for all $D' \in \mathcal{D}_{\mathcal{S}'}$.
\end{proof}
Next

\begin{lemma}
{\bf (i)} Suppose that, conditioned on $\mathcal{S}(0)$,  $D(0)$ is uniform on $\mathcal{D}_{\mathcal{S}(0)}.$ Then, for each $t\ge 0$,  $D(t)$ conditioned on $\mathcal{S}(0), \ldots, \mathcal{S}(t)$ is distributed uniformly on $\mathcal{D}_{\mathcal{S}(t)}$.

{\bf (ii)} Consequently, $\{\mathcal{S}(t)\}$ is Markovian.
\end{lemma}
\begin{proof}
{\bf (i)} We prove this lemma by induction on $t \geq 0$. The base case is assumed to be true by hypothesis. Suppose that, for some $t\ge 0$,   conditioned on $\mathcal{S}(0), \ldots, \mathcal{S}(t)$, $D(t)$  is uniformly distributed on $\mathcal{D}_{\mathcal{S}(t)}$. Let $\mathcal{S}=\mathcal{S}(t)=(V, \mathcal{O}_i, \mathcal{O}_o, \mu)$ and $\mathcal{S}'=\mathcal{S}(t+1)=(V', \mathcal{O}_i', \mathcal{O}_o', \mu')$. If $D' \in \mathcal{D}_{\mathcal{S}'}$, then
\begin{equation*}
P\big( D(t+1)=D' | \mathcal{S}(0), \ldots, \mathcal{S}(t)\big) = \sum_{D \in \mathcal{D}_{\mathcal{S}}} P\big(D(t+1)=D', D(t)=D | \mathcal{S}(0), \ldots, \mathcal{S}(t)\big),
\end{equation*}
and we can break up this latter probability as follows
\begin{multline*}
P\big(D(t+1)=D', D(t)=D | \mathcal{S}(0), \ldots, \mathcal{S}(t)\big) =\\ P\big(D(t+1)=D'|D(t)=D\big) P\big(D(t)=D|\mathcal{S}(0), \ldots, \mathcal{S}(t)\big). 
\end{multline*}
Using the inductive hypothesis, i.e.\ $P\big(D(t)=D|\mathcal{S}(0),\ldots,\mathcal{S}(t)\big) =|\mathcal{D}_{\mathcal{S}}|^{-1}$ $(\forall\,D \in \mathcal{D}_{\mathcal{S}})$, we have that
\begin{align*}
P\big(D(t+1)=D' | \mathcal{S}(0), \ldots, \mathcal{S}(t)\big) &= \sum_{D \in \mathcal{D}_{\mathcal{S}}} P\big(D(t+1)=D' |D(t)=D\big) \frac{1}{|\mathcal{D}_{\mathcal{S}}|}. 
\end{align*}
To finish the proof of the lemma, it suffices to show that
\begin{align}\label{eqn: 521}
\sum_{D \in \mathcal{D}_{\mathcal{S}}} P\left( D(t+1)=D' | D(t)=D \right) &= \frac{N_{\mathcal{S},\mathcal{S}'}^{(in)}}{|\mathcal{O}_i|+|\mathcal{O}_o|}+\frac{N_{\mathcal{S},\mathcal{S}'}^{(out)}}{|\mathcal{O}_i|+|\mathcal{O}_o|}, 
\end{align}
because if so, the probability that $D(t+1)=D'$ (conditioned on $\mathcal{S}(0), \ldots, \mathcal{S}(t)$) depends only upon $\mathcal{S}$ and $\mathcal{S}'$ and not on choice of $D' \in \mathcal{D}_{\mathcal{S}'}$.

Now let us prove \eqref{eqn: 521}. First, we break up the event $\{D(t+1)=D'\}$ into two events depending on whether we delete a vertex with in-degree zero or out-degree zero in the first substep. Let $C$ denote the event that we delete a vertex of in-degree zero. By symmetry, it suffices to show that
\begin{equation*}
\sum_{D \in \mathcal{D}_{\mathcal{S}}} P\left( \{D(t+1)=D'\} \cap C|D(t)=D\right) = \frac{N_{\mathcal{S},\mathcal{S}'}^{(in)}}{|\mathcal{O}_i|+|\mathcal{O}_o|}.
\end{equation*}
By  $N_{\mathcal{S},\mathcal{S}'}^{(in)}(D')\equiv N_{\mathcal{S},\mathcal{S}'}^{(in)}$, we know that exactly $N_{\mathcal{S},\mathcal{S}'}^{(in)}$ of these summands are non-zero. Furthermore, if $D$ is such that the probability is non-zero, then because we know (from $\mathcal{S}$ and $\mathcal{S}'$, as well as $C$) exactly which semi-isolated vertex is deleted first, this probability is precisely the probability of choosing this vertex in the first substep, which is $(|\mathcal{O}_i|+|\mathcal{O}_o|)^{-1}$. Part {\bf (i)} is proven.\\

{\bf (ii)} We compute
\begin{align*}
P&(\mathcal{S}(t+1)=\mathcal{S}'|\,\mathcal{S}(0),\dots,\mathcal{S}(t))=\,\sum_{D^\prime\in \mathcal{ D}_{\mathcal{S}'}}
P(D(t+1)=D'\,|\,\mathcal{S}(0),\dots,\mathcal{S}(t))\\
=&\,\sum_{D'\in \mathcal{D}_{\mathcal{S}'}}\sum_{D\in \mathcal{D}_{\mathcal{S}(t)}}
 P(D(t+1)=D'\,|\,D(t)=D)\cdot  P(D(t)=D\,|\,\mathcal{S}(0),\dots,\mathcal{S}(t)).
\end{align*}
By part {\bf (i)} of the lemma, conditioned on $\mathcal{S}(0), \ldots, \mathcal{S}(t)$,
$D(t)$  is uniformly distributed, so
\begin{equation}\label{eqn: feb3 1}
P(\mathcal{S}(t+1)=\mathcal{S}' | \mathcal{S}(0), \ldots, \mathcal{S}(t) ) = \,\sum_{D'\in \mathcal{D}_{\mathcal{S}'} \atop D \in \mathcal{D}_{\mathcal{S}(t)}}  P(D(t+1)=D'\,|\,D(t)=D)\cdot |\mathcal{D}_{\mathcal{S}(t)}|^{-1},
\end{equation}
which, by \eqref{eqn: 521}, depends on $\mathcal{S}(t)$ and $\mathcal{S}'$ only.
\end{proof}

In fact, we can go much farther in lumping states into simpler equivalence classes. Namely, let's introduce a generic state $\mathbf{s}$ of the deletion process as $\mathbf{s}=(\nu, \nu_i, \nu_o, \mu)$, formed by the cardinalities of the three set components of $\mathcal{S}$: $\nu=|V|$, $\nu_i=|\mathcal{O}_i|$, and $\nu_o=|\mathcal{O}_o|$. 

Now we can prove that the reduced deletion process $\{\mathbf{s}(t)\},$ formed by the cardinalities of the three set components of $\mathcal{S}(t)$ (i.e.\ $\mathbf{s}(t)  = (|V(t)|, |\mathcal{O}_i(t)|, |\mathcal{O}_o(t)|, \mu(t))$), is Markovian as well. To this end, we will need to introduce $g(\mathbf{s})=g(\nu, \nu_i, \nu_o, \mu),$ the number of digraphs with vertex set $[\nu]$ and $\mu$ arcs such that $\{1, \ldots, \nu_i\}$ are the vertices of in-degree zero and $\{\nu_i+1, \ldots, \nu_i+\nu_o\}$ are the vertices of out-degree zero. Note that the number of digraphs, $|\mathcal{D}_{\mathcal{S}}|$, with parameter $\mathcal{S}$ depends only upon the cardinalities of the entries of $\mathcal{S}$ (i.e.\ $\mathbf{s}$).

\begin{lemma} $\{\mathbf{s}(t)\}$ is a Markov process.
\end{lemma}
\begin{proof}
Let $|\mathcal{S}|$ denote $(|V|, |\mathcal{O}_i|, |\mathcal{O}_o|, \mu)$. Combining \eqref{eqn: 521} and \eqref{eqn: feb3 1}, we have that
\begin{align*}
P(\mathcal{S}(t+1) = \mathcal{S}'| \mathcal{S}(0), \ldots, \mathcal{S}(t) ) &= \sum_{D' \in \mathcal{D}_{\mathcal{S}'} }\left( \frac{N^{(in)}_{\mathcal{S}, \mathcal{S}'} }{|\mathcal{O}_i| + |\mathcal{O}_o|} + \frac{N^{(out)}_{\mathcal{S}, \mathcal{S}'}}{|\mathcal{O}_i|+|\mathcal{O}_o|} \right) \frac{1}{g(|\mathcal{S}|)} \\&= \left( \frac{N^{(in)}_{\mathcal{S}, \mathcal{S}'} }{|\mathcal{O}_i| + |\mathcal{O}_o|} + \frac{N^{(out)}_{\mathcal{S}, \mathcal{S}'}}{|\mathcal{O}_i|+|\mathcal{O}_o|} \right) \frac{g(|\mathcal{S}'|)}{g(|\mathcal{S}|)}.
\end{align*}
Hence, if $\mathbf{s}=|\mathcal{S}|$, then 
\begin{align*}
P(\mathbf{s}^\prime|\,\mathcal{S}(t))&:=P\big(\mathbf{s}(t+1)=\mathbf{s}^\prime |\mathcal{S}(t)\big) = \sum_{\mathcal{S}':\mathbf{s}(\mathcal{S}')=\mathbf{s}'} \frac{N_{\mathcal{S},\mathcal{S}'}^{(in)}}{\nu_i+\nu_o} \frac{g(\mathbf{s}')}{g(\mathbf{s})}+\sum_{\mathcal{S}':\mathbf{s}(\mathcal{S}')=\mathbf{s}'} \frac{N_{\mathcal{S},\mathcal{S}'}^{(out)}}{\nu_i+\nu_o} \frac{g(\mathbf{s}')}{g(\mathbf{s})}\\
&=P_i(\mathbf{s}^\prime|\,\mathcal{S}(t))+P_o(\mathbf{s}^\prime|\,\mathcal{S}(t)),
\end{align*}
where
\begin{equation*}
P_i(\mathbf{s}^\prime|\,\mathcal{S}(t)) :=\sum_{\mathcal{S}':\mathbf{s}(\mathcal{S}')=\mathbf{s}'} \frac{N_{\mathcal{S},\mathcal{S}'}^{(in)}}{\nu_i+\nu_o} \frac{g(\mathbf{s}')}{g(\mathbf{s})}, \quad P_o(\mathbf{s}^\prime|\,\mathcal{S}(t)) :=\sum_{\mathcal{S}':\mathbf{s}(\mathcal{S}')=\mathbf{s}'} \frac{N_{\mathcal{S},\mathcal{S}'}^{(out)}}{\nu_i+\nu_o} \frac{g(\mathbf{s}')}{g(\mathbf{s})}.
\end{equation*}
Consider $P_i(\mathbf{s}^\prime|\,\mathcal{S}(t))$. Recall that $N_{\mathcal{S},\mathcal{S}'}^{(in)}=0$ unless 
\begin{equation*}
\nu-\nu'=1+b, \ \ \nu_i'-\nu_i=r_i-1,\ \ \nu_o'=\nu_o-b,\ \ \mu' = \mu-k,
\end{equation*}
for some $b,r \geq 0$ and $k \geq \max\{1, b+r_i\};$ in particular, note that $|A|=a, |B|=b, |R_i|=r_i, |R_o|=r_o.$ We can solve for $b$ and $r_i$ here in terms of $\mathbf{s}$ and $\mathbf{s}'$. For any $\mathcal{S}$ and $\mathcal{S}'$, so that $N_{\mathcal{S},\mathcal{S}'}^{(in)} \neq 0$, 
\begin{equation*}
N_{\mathcal{S},\mathcal{S}'}^{(in)} = {\nu-\nu_i-b-r_i \choose k-b-r_i}.
\end{equation*}
We wish to find the number of $\mathcal{S}'$ with $|\mathcal{S}'|=\mathbf{s}'$ so that $N_{\mathcal{S},\mathcal{S}'}^{(in)} \neq 0$. Any such $\mathcal{S}'$ can be found by choosing exactly $1$ vertex from $\mathcal{O}_{i}$ to be deleted in Substep 1, choosing $b$ vertices to be deleted out of $\mathcal{O}_{o}$ to obtain $\mathcal{O}_{o}'$ (i.e.\ choosing the set $B$), and choosing $r_i$ vertices out of $V\setminus \left(\mathcal{O}_{i}\cup\mathcal{O}_{o}\right)$ that are added to obtain $\mathcal{O}_{i}'$ (i.e.\ choosing the set $R_i$). This yields
\begin{equation*}
|\{ \mathcal{S}' : \mathbf{s}(\mathcal{S}')=\mathbf{s}', N_{\mathcal{S}, \mathcal{S}'}^{(in)}\neq 0 \}| = {\nu_i \choose 1} {\nu_o \choose b}{\nu-\nu_i-\nu_o \choose r_i}.
\end{equation*}
So, we find that
\begin{equation*}
P_i(\mathbf{s}^\prime|\,\mathcal{S}(t)) = \frac{1}{\nu_i+\nu_o} \frac{g(\mathbf{s}')}{g(\mathbf{s})} {\nu-\nu_i-b-r_i \choose k-b-r_i}{\nu_i \choose 1}{\nu_o \choose b}{\nu-\nu_i-\nu_o \choose r_i}, 
\end{equation*}
and likewise, 
\begin{equation*}\label{Po(s'|S)}
P_o(\mathbf{s}^\prime|\,\mathcal{S}(t)) = \frac{1}{\nu_i+\nu_o} \frac{g(\mathbf{s}')}{g(\mathbf{s})} {\nu-\nu_i-a-r_o \choose k-a-r_o}{\nu_o \choose 1}{\nu_i \choose a}{\nu-\nu_i-\nu_o \choose r_o}.
\end{equation*}
We see that both $P_i(\mathbf{s}^\prime|\,\mathcal{S}(t))$ and $P_o(\mathbf{s}^\prime|\,\mathcal{S}(t))$ depend on $\mathcal{S}(t)$ only through $\mathbf{s}(t)$. Therefore $\{\mathbf{s}(t)\}$ is a Markov process, with the transition probability given by
\begin{equation}\label{Ps'|s}
\begin{aligned}
P(\mathbf{s}^\prime|\,\mathbf{s})&=P_i(\mathbf{s}^\prime|\,\mathbf{s})+P_o(\mathbf{s}^\prime|\,\mathbf{s}),\\
P_i(\mathbf{s}^\prime|\,\mathbf{s})&:=\frac{1}{\nu_i+\nu_o} \frac{g(\mathbf{s}')}{g(\mathbf{s})} {\nu-\nu_i-b-r_i \choose k-b-r_i}{\nu_i \choose 1}{\nu_o \choose b}{\nu-\nu_i-\nu_o \choose r_i},\\
P_o(\mathbf{s}^\prime|\,\mathbf{s})&:= \frac{1}{\nu_i+\nu_o} \frac{g(\mathbf{s}')}{g(\mathbf{s})} {\nu-\nu_i-a-r_o \choose k-a-r_o}{\nu_o \choose 1}{\nu_i \choose a}{\nu-\nu_i-\nu_o \choose r_o}.
\end{aligned}
\end{equation}
The fivesome $(a,b,r_i,r_o,k)$ is uniquely determined by the {\it four} components of $
\Delta\mathbf{s}:=(\mathbf{s}^\prime-\mathbf{s})$ for both $P_i(\mathbf{s}^\prime|\,\mathbf{s})$ and $P_o(\mathbf{s}^\prime|\,\mathbf{s})$,  since $a=1$, $r_o=0$ for $i$-transition from $\mathbf{s}$ to $\mathbf{s}'$, and so 
\begin{equation*}
\Delta \mathbf{s}^T=(-1-b,r_i-1,-b,-k ),
\end{equation*}
and, for $o$-transition, we have $b=1$, $r_i=0$, so that  
\begin{equation*}
\Delta \mathbf{s}^T=(-1-a, -a, r_o-1,- k ).
\end{equation*}
With this correspondence in mind, we will denote $P_{i,o}(\mathbf{s}'|\,\mathbf{s})=
P_{i,o}(\Delta\mathbf{s}|\,\mathbf{s})$, and $P(\mathbf{s}^\prime|\,\mathbf{s})=
P_i(\Delta\mathbf{s}|\,\mathbf{s})+P_o(\Delta\mathbf{s}|\,\mathbf{s})$.
\end{proof}

\section{Asymptotic Transition Probability}\label{sec: asymptotic trans prob}

The transition probability formulas \eqref{Ps'|s} contain the ratio $g(\mathbf{s}
^\prime)/g(\mathbf{s})$ of the counts of digraphs with parameters $\mathbf{s}^\prime$ and
$\mathbf{s}$. No exact formula for those counts is available, and so our next step is
to evaluate these counts sharply for $\mathbf{s}^\prime$ and $\mathbf{s}$ likely to
be encountered in the deletion process. To obtain a usable approximation of the transition 
probabilities, we will also have to approximate the binomials
in those formulas, but this is considerably easier.  

\subsection{Counting digraphs with constrained degree sequences}\label{sec: estimating g(s)}

To estimate $g(\mathbf{s})$, we use an argument resembling that in Pittel~\cite{pittel counting strong} for estimating and upper-bounding the total number of sparse digraphs with given
numbers of vertices, arcs, vertices of {\it out\/}-degree zero, and no vertices of {\it in\/}-degree 
zero.  We will need the following theorem counting the number of directed graphs with a specific in/out-degree sequence which is an important case of McKay's  asymptotic formula \cite{mckay1}, \cite{mckay2} for the number of $(0,1)$-matrices with specified row and column sums.
(For very recent progress see also Greenhill and McKay~\cite{Green1}, \cite{Green2}.)
\begin{theorem}\label{thm: mckay formula}
Let $\boldsymbol{\delta}:=(\delta_1, \delta_2, \ldots, \delta_{\nu}) \geq \mathbf{0}$ and $\boldsymbol{\Delta}:=(\Delta_1, \ldots, \Delta_{\nu})\geq \mathbf{0}$ be such that
\begin{equation*}
\sum_i \delta_i = \mu, \quad \sum_i \Delta_i = \mu,
\end{equation*}
where $\mu \geq \nu.$ Introduce $g(\boldsymbol{\delta}, \boldsymbol{\Delta})$, the number of digraphs on $[\nu]$ with $\mu$ arcs with in-degree sequence $\boldsymbol{\delta}$ and out-degree sequence $\boldsymbol{\Delta}$. If $D:=\max_i \delta_i + \max_i \Delta_i=O(\mu^{1/4}),$ then
\begin{equation*}
g(\boldsymbol{\delta}, \boldsymbol{\Delta}) = \frac{ \mu!}{\prod_{i} \delta_i! \Delta_i!} H(\boldsymbol{\delta}, \boldsymbol{\Delta}),
\end{equation*}
where the ``fudge factor", always 1 at most, is given by 
\begin{equation*}
H(\boldsymbol{\delta}, \boldsymbol{\Delta})=\exp \left( - \frac{1}{\mu} \sum_i \delta_i \Delta_i - \frac{1}{2 \mu^2} \sum_i (\delta_i)_2 \sum_j (\Delta_j)_2 + O\left( D^4/\mu\right)\right).
\end{equation*}
\end{theorem}
McKay proved this theorem using a random matching scheme. First, one starts with two copies, $[\mu]_1$ and $[\mu]_2$, of $[\mu]$ along with partitions $[\mu]_1=\cup_i  I_i, [\mu]_2 =\cup_i O_i$, where $|I_i|=\delta_i$ and $|O_i|=\Delta_i$. Each of the $\mu!$ matchings from $[\mu]_1$ to $[\mu]_2$ gives rise to a directed multigraph (where multiple arcs between pairs of vertices and loops are allowed) where we collapse each $I_i$ and $O_i$ to a single vertex, $v_i$, and keep all present arcs.

By construction, this directed multigraph has in-degree sequence $\boldsymbol{\delta}$ and out-degree sequence $\boldsymbol{\Delta}$. Each digraph, without loops or multiple arcs, corresponds to exactly $\prod_i (\delta_i! \Delta_i!)$ matchings. So $H(\boldsymbol{\delta}, \boldsymbol{\Delta})$ is precisely the probability that a matching chosen uniformly at random among all possible matchings gives rise to a digraph without loops and multiple arcs. 
Thus we necessarily have $H(\boldsymbol{\delta}, \boldsymbol{\Delta}) \leq 1$, yielding
\begin{equation*}
g(\boldsymbol{\delta}, \boldsymbol{\Delta}) \leq \frac{\mu!}{\prod_i (\delta_i! \Delta_i!)}.
\end{equation*}

\begin{lemma}\label{bad g(s) estimate} For all $\mathbf{s}=(\nu, \nu_i, \nu_o, \mu)$, 
\begin{equation}\label{eqn 917 1}
g(\mathbf{s}) \leq \mu!\, \frac{\left(e^x-1\right)^{\nu-\nu_i}}{x^{\mu}} \frac{\left(e^y-1\right)^{\nu-\nu_o}}{y^{\mu}},\quad\forall\, x,y>0.
\end{equation}
\end{lemma}
\begin{proof}
First, by the definition of $g(\nu,\nu_i,\nu_o,\mu)$,
\begin{equation*}
g(\nu,\nu_i,\nu_o,\mu)= \sum_{(\boldsymbol{\delta},\boldsymbol{\Delta})\in \boldsymbol{D}}g(\boldsymbol{\delta},\boldsymbol{\Delta}),
\end{equation*}
where $\boldsymbol{D}$ is the set of all pairs of admissible in-degree sequences and out-degree sequences for the
$4$-tuple $(\nu, \nu_i, \nu_o, \mu)$. Formally, $\boldsymbol{D}$ is defined by constraints
\begin{equation}\label{sumdelta,Delta}
\sum_{j=1}^{\nu} \delta_j=\sum_{j=1}^{\nu} \Delta_j = \mu,
\end{equation}
and 
\begin{equation}\label{delta;Delta}
\begin{aligned}
\delta_j &= 0, \text{ for } j \in \{1,\ldots, \nu_i\}, \\ \delta_j & > 0 \text{ for } j > \nu_i, \\ 
\Delta_j&=0 \text{ for } j \in \{\nu_i+1, \ldots, \nu_i+\nu_o\},\\
 \Delta_j &>0 \text{ for } j \notin \{\nu_i+1, \ldots, \nu_i+\nu_o\}.
\end{aligned}
\end{equation}
 The constraint \eqref{sumdelta,Delta} calls for using the bivariate generating function
of the counts of pairs $(\boldsymbol{\delta},\boldsymbol{\Delta})$ by the values of $\sum_j\delta_j$
and $\sum_j\Delta_j$. From McKay's Formula (Theorem \ref{thm: mckay formula}), we see that 
\begin{equation}\label{counting g eqn}
\begin{aligned}
g(\mathbf{s}) &= \mu!\, [x^{\mu} y^{\mu}] \sum_{(\boldsymbol{\delta}, \boldsymbol{\Delta}) \in \boldsymbol{D}} H(\boldsymbol{\delta}, \boldsymbol{\Delta}) \prod_j \frac{x^{\delta_j}y^{\Delta_j}}{\delta_j! \Delta_j!}\\
 & \leq \mu! \, [x^{\mu}y^{\mu}] \sum_{(\boldsymbol{\delta},\boldsymbol{\Delta})
 \text{ meets }\eqref{delta;Delta}} \prod_j \frac{x^{\delta_j}}{\delta_j!} \cdot \frac{y^{\Delta_j}}{\Delta_j!}.
 \end{aligned}
\end{equation}
Since the constraints \eqref{delta;Delta} are imposed only on individual $\delta_{j^\prime}$, $\Delta_{j^{\prime\prime}}$,  the last sum can be factored into the product of simple series:
\begin{align*}
g(\mathbf{s}) &\leq \mu!\, [x^{\mu}y^{\mu}] \prod_{j=1}^{\nu_i} \sum_{\Delta_j \geq 1} \frac{y^{\Delta_j}}{\Delta_j!} \cdot \prod_{j=\nu_i+1}^{\nu_i+\nu_o} \sum_{\delta_j \geq 1} \frac{x^{\delta_j}}{\delta_j!} \cdot \prod_{j=\nu_i+\nu_o+1}^{\nu} \sum_{\delta_j, \Delta_j \geq 1} \frac{x^{\delta_j} y^{\Delta_j}}{\delta_j! \Delta_j!} \\&= \mu!\, [x^{\mu}y^{\mu}]\left(\left( e^y-1 \right)^{\nu_i} \cdot \left( e^x-1\right)^{\nu_o} \cdot \left( (e^x-1)(e^y-1) \right)^{\nu-\nu_i-\nu_o}\right)\\
&\leq \mu! \, x^{-\mu}y^{-\mu}\left(\left( e^y-1 \right)^{\nu_i} \cdot \left( e^x-1\right)^{\nu_o} \cdot \left( (e^x-1)(e^y-1) \right)^{\nu-\nu_i-\nu_o}\right),
\end{align*}
for all $x,y>0$.
\end{proof}

Naturally, we wish to determine the values of $x$ and $y$ which minimize the RHS in \eqref{eqn 917 1}.

\begin{lemma}\label{zi,zo}
Suppose $\mu>\nu-\nu_i, \,\nu-\nu_o>0$. If 
\begin{equation*}
\phi_i(x):=\frac{(e^x-1)^{\nu-\nu_i}}{x^\mu}, \hspace{1 cm} \phi_o(y):=\frac{(e^y-1)^{\nu-\nu_o}}{y^\mu},
\end{equation*}
then $\phi_i(x)\phi_o(y)$ is minimized at $x=z_i, y=z_o$, where $z_i, z_o$ are the unique positive roots of
\begin{equation*}
\frac{z_i e^{z_i}}{e^{z_i}-1} = \frac{\mu}{\nu-\nu_i}, \hspace{1 cm} \frac{z_o e^{z_o}}{e^{z_o}-1} = \frac{\mu}{\nu-\nu_o}.
\end{equation*}
\end{lemma}

\begin{remark}
We shall see later that the non-zero in-degree (the non-zero out-degree, resp.) of a generic vertex in a digraph chosen uniformly at random among all $g(\mathbf{s})$ digraphs is close in distribution to a Poisson $Z(z_i)$ ($Z(z_o)$, resp.) conditioned on $Z(z_i) \geq 1$ ($Z(z_o) \geq 1$, resp.). In fact, these variables $z_i$ and $z_o$ are important for the
asymptotics of $g(\mathbf{s})$. Such seemingly ``hidden'' parameters turned out to
be ubiquitous in situations requiring asymptotic enumeration of graphs with constrained degree sequences; see, \cite{aronson},  \cite{pittel wormald}, \cite{pittel spencer wormald}.
\end{remark}

\begin{proof}
Since $\phi_i$ and $\phi_o$ are positive functions, it suffices to minimize $\phi_i(x)$ and $\phi_o(y)$ separately.  Consider $\phi_i(x)$. We find that
\begin{align*}
\frac{d}{dx}\ln \phi_i(x) = (\nu-\nu_i) \frac{e^x}{e^x-1} -  \frac{\mu}{x} = \frac{\nu-\nu_i}{x} \left( \frac{x e^x}{e^x-1} - \frac{\mu}{\nu-\nu_i}\right).
\end{align*}
The function $\ell(x):=\frac{x e^x}{e^x-1}$ is strictly increasing on $(0, \infty),$ where $\ell(0+)=1<\frac{\mu}{\nu-\nu_i}$ and $\ell(\infty)=\infty$. Hence, $\phi_i(x)$ does attain its minimum at the unique $z_i=
z_i(\mathbf{s})$, satisfying $\ell(z_i)=\frac{\mu}{\nu-\nu_i}$. Likewise $\phi_o(y)$ attains its minimum at the
unique $z_o=z_o(\mathbf{s})$, satisfying $\ell(z_o)=\frac{\mu}{\nu-\nu_o}$.
\end{proof}

Now we can sharply approximate $g(\mathbf{s})$ for a good range of $\mathbf{s}$. 

\begin{theorem}\label{thm g(s) asymptotics}
Let $\mathbf{s}=(\nu,\nu_i,\nu_o,\mu)$ be such that $\nu\to\infty$, $\mu-\nu\to\infty$, $\mu=O(\nu)$ and $\nu-\nu_i-\nu_o =\Theta(\nu)$.
Introduce $Z^i, Z^o,$ two independent truncated Poissons, with parameters $z_i=z_i(\mathbf{s})$ and $z_o=z_o(\mathbf{s})$ from Lemma \ref{zi,zo}, i.e.\
\begin{equation*}
P(Z^i=j)=\frac{z_i^j/j!}{e^{z_i}-1}, \quad P(Z^o=j)=\frac{z_o^j/j!}{e^{z_o}-1}, \quad (j \geq 1).
\end{equation*}
Then
\begin{equation}\label{g(s)sim}
\begin{aligned}
g(\mathbf{s}) &= \left(1 + O\left( \frac{ (\ln \nu)^6}{\nu}\right)\right) \mu! \,\frac{ (e^{z_i}-1)^{\nu-\nu_i}(e^{z_o}-1)^{\nu-\nu_o}}{z_i^{\mu} z_o^{\mu}}\\ &\times \frac{e^{-\eta}}{2 \pi \sqrt{(\nu-\nu_i)Var[Z^i] (\nu-\nu_o)Var[Z^o]}},
\end{aligned}
\end{equation}
where
\begin{equation*}
\eta=\frac{\mu(\nu-\nu_i-\nu_o)}{(\nu-\nu_i)(\nu-\nu_o)}+\frac{z_i z_o}{2}.
\end{equation*}
\end{theorem}

This theorem and its proof  are similar to Theorem 2.2 in \cite{pittel counting strong}. 
\begin{proof}
First  rewrite \eqref{counting g eqn}  as
\begin{equation}\label{thesum}
\begin{aligned}
g(\nu, \nu_i, \nu_o, \mu) &= \mu! \,\frac{(e^{z_i}-1)^{\nu-\nu_i} (e^{z_o}-1)^{\nu-\nu_o}}{ (z_i z_o)^{\mu}} \\ &\times \sum_{(\boldsymbol{\delta},\boldsymbol{\Delta})\in \boldsymbol{D}} H(\boldsymbol{\delta}, \boldsymbol{\Delta})  [x^\mu y^\mu]
\prod_{j : \delta_j>0} \frac{(z_i x)^{\delta_j}/\delta_j!}{e^{z_i}-1}\cdot \prod_{j : \Delta_j>0} \frac{(z_o y)^{\Delta_j}/\Delta_j!}{e^{z_o}-1}.
\end{aligned}
\end{equation}
Further we notice that $\frac{(z_i)^{\delta_j}/\delta_j!}{e^{z_i}-1}$ is precisely the probability that $Z^i$ equals $\delta_j$; $\frac{(z_0)^{\delta_j}/\delta_j!}{e^{z_o}-1}$ is interpreted
similarly. The products over $j$ in \eqref{thesum} force us to introduce the sequences of independent copies of $Z^i$ and $Z^o$. Define $Z^i_{\nu_i+1}, Z^i_{\nu_i+2}, \ldots, Z^i_{\nu}$ as independent copies of $Z^i$, and $Z^o_1, \ldots, Z^o_{\nu_i}, Z^o_{\nu_i+\nu_o+1}, \ldots, Z^o_{\nu}$ as independent copies of $Z^o$, and introduce
\begin{align*}
\bold{Z}^i&=(\nu_i \text{ zeroes}, Z^i_{\nu_i+1}, \ldots, Z^i_{\nu}),\\
\bold{Z}^o&=(Z^o_1, \ldots, Z^o_{\nu_i}, \nu_o \text{ zeroes}, Z^o_{\nu_i+\nu_o+1}, \ldots, Z^o_{\nu}). 
\end{align*}
For the coordinates that are zero, we define $Z^i_{j^\prime}$ and $Z^o_{j^{\prime\prime}}$ to be zero.
Because all $Z^i_{j^\prime}$, $Z^o_{j^{\prime\prime}}$ are independent, the factor by
$H(\boldsymbol{\delta}, \boldsymbol{\Delta})$ in the second line of \eqref{thesum} is
$P(\bold{Z}^i=\boldsymbol{\delta},\bold{Z}^o=\boldsymbol{\Delta})$. Therefore 
the expression in this line is the expected value of $H(\boldsymbol{Z}^i, \boldsymbol{Z}^o)$,
{\it conditioned\/} on the event $\{\|\boldsymbol{Z}^i\| = \mu, \|\boldsymbol{Z}^o\|=\mu\}$;
($\|\{x_j\}\|\overset{def}=\sum_j |x_j|$, the 1-norm of $\{x_j\}$). To get a handle on this conditional expectation, we
notice first  that
\begin{align*}
E\left[H(\boldsymbol{Z}^i, \boldsymbol{Z}^o) 1_{\{\|\boldsymbol{Z}^i\|=\|\boldsymbol{Z}^o\|=\mu\}} \right] &= E\left[H(\boldsymbol{Z}^i, \boldsymbol{Z}^o) \Big| \|\boldsymbol{Z}^i\| = \| \boldsymbol{Z}^o \| = \mu\right] \\ &\times P(\|\boldsymbol{Z}^i\|=\mu)\,P(\|\boldsymbol{Z}^o\|=\mu).
\end{align*}
Furthermore, by a local limit theorem, (see \cite{aronson}, Appendix), one can show that, under conditions
of Theorem \ref{thm g(s) asymptotics},
\begin{equation*}
P(\|\boldsymbol{Z}^i\|=\mu)=\frac{1+O(1/\nu)}{\sqrt{2 \pi (\nu-\nu_i)Var[Z^i]}}, \quad P(\|\boldsymbol{Z}^o\|=\mu)=\frac{1+O(1/\nu)}{\sqrt{2 \pi (\nu-\nu_o)Var[Z^o]}}.
\end{equation*}
And just  like~\cite{pittel counting strong}, (see also Pittel and Wormald~\cite{pittel wormald asym}), one can show that this conditional expectation of $H(\boldsymbol{Z}^i, \boldsymbol{Z}^o)$ is within a multliplicative factor  $(1+O( (\ln \nu)^6/\nu))$ from 
\begin{equation*}
\exp \left\{ - \frac{1}{\mu} E\left[\, \sum_j Z^i_j Z^o_j \right] - \frac{1}{2} \left(\frac{1}{\mu} E\left[ \sum_j \left( Z^i_j\right)_2\right]\right)\left( \frac{1}{\mu} E\left[\sum_j \left( Z^o_j \right)_2\right] \right) \right\},
\end{equation*}
where $(Z)_2:=Z(Z-1)$. Since $Z^i_j$ and $Z^o_j$ are independent, we have $E[Z^i_j Z^o_j] = E[Z^i_j] E[Z^o_j]$. For the first $\nu_i+\nu_o$ coordinates, either $E[Z^i_j]=0$ or $E[Z^o_j]=0$. For the last $\nu-\nu_i-\nu_o$ entries, 
\begin{equation*}
E[Z^i_j] = E[Z^i] = \frac{z_i}{1-e^{-z_i}} = \frac{\mu}{\nu-\nu_i}
\end{equation*}
and 
\begin{equation*}
E[Z^o_j] = E[Z^o]= \frac{\mu}{\nu-\nu_o}.
\end{equation*}
Hence
\begin{equation*}
\frac{1}{\mu}E\left[\sum_j Z^i_j Z^o_j \right] = \frac{\mu(\nu-\nu_i-\nu_o)}{(\nu-\nu_i)(\nu-\nu_o)}.
\end{equation*}
Furthermore, if $Z^i_j$ is non-zero, then 
\begin{equation*}
E[(Z^i_j)_2] = E[(Z^i)_2] = z_i^2/(1-e^{-z_i}) = z_i \mu/(\nu-\nu_i).
\end{equation*}
Since exactly $\nu-\nu_i$ of the $Z^i_j$ are non-zero, we have that
\begin{equation*}
\frac{1}{\mu} E\left[ \sum_j (Z^i_j)_2 \right] = \frac{\nu-\nu_i}{\mu} \cdot \frac{z_i \mu}{\nu-\nu_i}=z_i.
\end{equation*}
Likewise, we have
\begin{equation*}
\frac{1}{\mu} E\left[ \sum_j (Z^o_j)_2 \right] = z_o
\end{equation*}
which completes the proof of the theorem.
 \end{proof}
Since the asymptotic count of digraphs with parameter $\mathbf{s}$ is, in essence,
an approximate estimate of the number of most numerous digraphs, the following 
interpretation of $Z^i$ and $Z^o$ is hardly surprising. Pick an admissible digraph uniformly
at random among all $g(\mathbf{s})$ such digraphs. Then, under conditions of Theorem  
\ref{thm g(s) asymptotics}, the out-degree (in-degree) of a fixed vertex in the set $\{\nu_i+1,\dots, n\}$ ($\{1,\dots,\nu_i\}\cup\{\nu_i+\nu_o+1,\dots,n\}$ resp.) is asymptotic, in distribution, to $Z^o$ ($Z^i$ resp.). Moreover, for a fixed set of vertices, their in/out-degrees are asymptotically
independent.

\subsection{Approximating transition probabilities}\label{sec: approx trans prob}

Now that we have found an asymptotic formula for $g(\mathbf{s})$, we can find asymptotic formulas for our transition probabilities defined in \eqref{Ps'|s}. To do so, we will impose stricter conditions on $\mathbf{s}$. First define two functions,
\begin{equation}\label{definition of F_1}
F_1(\mathbf{s}) := \frac{\mu(\nu-\nu_i-\nu_o)}{(\nu-\nu_i)(\nu-\nu_o)}, \quad F_2(\mathbf{s}):= \frac{z_i(\mathbf{s}) z_o(\mathbf{s})}{\mu/n}.
\end{equation}
Implicit in this definition is the constraint on $\mathbf{s}$:
\begin{equation}\label{impl}
\mu,\,\nu-\nu_i,\,\nu-\nu_o>0, \quad \nu \leq n, \quad \mu \leq c_n n, \quad \quad\text{and}\quad \frac{\mu}{\nu-\nu_i},\,\frac{\mu}{\nu-\nu_o}>1.
\end{equation}
We will eventually prove that, for the initial states $\mathbf{s}_0$ in the
likely range arising from $D(n, m=c_n n)$, the values of $F_1$ and $F_2$ a.a.s.\  are  almost constant throughout  the deletion process. That's what makes $F_i(\mathbf{s})$ instrumental in our analysis of the deletion algorithm.

\begin{remark}
Notice that both functions appear in the exponent $\eta=\eta(\mathbf{s})$ in Theorem
\ref{thm g(s) asymptotics}. Their appearance is coincidental, as far as we can tell. 
\end{remark}

Since $c_n \to c \in (1, \infty)$, we may and will assume that $c_n$ is bounded away from $1$.
Suppose that $\epsilon=\epsilon(n)>0$ is such that $c_n - \epsilon \ge 1 + \delta$,
for some fixed $\delta>0$. Define
\begin{equation}\label{eqn: def of s eps}
\mathbf{S}_{\epsilon} := \big \{ \mathbf{s}: \mathbf{s}\text{ meets }\eqref{impl};\,F_1(\mathbf{s}),\,
F_2(\mathbf{s}) \in (c_n - \epsilon, c_n + \epsilon);\, \nu_i+\nu_o>0\big\}.
\end{equation}
\begin{fact}\label{remark: 2.24.1}
Uniformly over $\mathbf{s} \in \mathbf{S}_{\epsilon}$, 
\begin{itemize}
\item {\bf (i)\/} $\nu,\, \nu-\nu_i, \,\nu-\nu_o,\, \nu-\nu_i-\nu_o = \Theta(n)$;
\item {\bf (ii)\/} $\mu - \nu = \Theta(n)$;
\item {\bf (iii)\/} $z_i(\mathbf{s})$ and $z_o(\mathbf{s})$ (defined in Lemma \ref{zi,zo}) are bounded away from $0$ and 
$\infty$.
\end{itemize}
In particular, the conditions of Theorem \ref{thm g(s) asymptotics} are met uniformly for
$\mathbf{s}\in\mathbf{S}_{\epsilon}$. 
\end{fact}
\begin{proof} Since $F_1(\mathbf{s})>0$, we have $\nu-\nu_i-\nu_o>0$. That $z_i(\mathbf{s})$ and $z_o(\mathbf{s})$ are bounded away from $0$ follows from 
\begin{equation*}
\frac{\mu}{\nu-\nu_i},\,\frac{\mu}{\nu-\nu_o} \ge F_1(\mathbf{s})\ge c_n-\epsilon > 1.
\end{equation*}
Now, because $z_i(\mathbf{s})$ and $z_o(\mathbf{s})$ are bounded away from 0 and $z_i(\mathbf{s}) z_o(\mathbf{s}) / (\mu/n) = F_2(\mathbf{s}) \leq c_n + \epsilon$, we have that $\mu/n$ is bounded away from zero as well. In addition, since $\mu =O(n)$, $F_2(\mathbf{s}) \leq c_n + \epsilon$ and $z_i(\mathbf{s})$ is bounded away from zero, we conclude that $z_o(\mathbf{s})$ is bounded away from $\infty$; likewise $z_i(\mathbf{s})$ is bounded away from $\infty$ as well.  And this implies the part {\bf (i)\/}. Finally, the part
{\bf (ii)\/} follows then from 
\begin{equation*}
\mu=F_1(\mathbf{s})\frac{(\nu-\nu_i)(\nu-\nu_o)}{\nu-\nu_i-\nu_o}
=F_1(\mathbf{s})\left[\nu+\frac{\nu_i\nu_o}{\nu-\nu_i-\nu_o}\right],
\end{equation*}
and the condition $F_1(\mathbf{s})\ge c_n-\epsilon$.  
\end{proof}

With this preliminaries done, we focus on the factor $g(\mathbf{s}')/g(\mathbf{s})$ in the formulas \eqref{Ps'|s}
for the transition probabilities $P_i(\Delta  \mathbf{s}|\mathbf{s})$ and $P_o(\Delta \mathbf{s}|\mathbf{s})$.  In the next statement and in the
rest of the paper we will use a notation $A\lesssim  B$, meaning $A=O(B)$, when expression for $B$ is too bulky. 

\begin{lemma}\label{g(s) ratio asymptotics} Let  $\mathbf{s}' = \mathbf{s} +\Delta \mathbf{s}$ and $\Delta \mathbf{s}= (-a-b, r_i-a, r_o-b, -k)$.
Uniformly over $\mathbf{s} \in \mathbf{S}_{\epsilon}$, 
$\mathbf{(i)}$ 
\begin{equation*}
\frac{g(\mathbf{s}')}{g(\mathbf{s})} \lesssim \frac{\nu}{(\mu)_k} \frac{(z_i z_o)^k}{(e^{z_i}-1)^{b+r_i}(e^{z_o}-1)^{a+r_o}};
\end{equation*}
$\mathbf{(ii)}$ if, in addition, $a=1, \,r_o=0$ and $\max\{1,b+r_i\} \leq k \leq \ln n$, then
\begin{equation*}
\frac{g(\mathbf{s}')}{g(\mathbf{s})}= \left( 1 + O\left( \frac{(\ln n)^6}{n}\right) \right) \left( \frac{z_i z_o}{\mu} \right)^k \frac{1}{(e^{z_o}-1)(e^{z_i}-1)^{b+r_i}};
\end{equation*}
likewise, if  $b=1,\, r_i=0$ and $\max\{1, a+r_o\} \leq k \leq \ln n$, then
\begin{equation*}
\frac{g(\mathbf{s}')}{g(\mathbf{s})}= \left( 1 + O\left( \frac{(\ln n)^6}{n}\right) \right) \left( \frac{z_i z_o}{\mu} \right)^k  \frac{1}{(e^{z_i}-1)(e^{z_o}-1)^{a+r_o}}.
\end{equation*}
\end{lemma}

\begin{proof} $\mathbf{(i)}$ First, by Theorem \ref{thm g(s) asymptotics}, we have that 
\begin{equation*}
g(\mathbf{s})=\left(1+O\left(\frac{(\ln n)^6}{n}\right)\right)  \frac{ \mu! h_i(\mathbf{s}, z_i) h_o(\mathbf{s}, z_o) e^{-\eta(\mathbf{s})}}{2 \pi \sqrt{(\nu-\nu_i)Var[Z^i] (\nu-\nu_o)Var[Z^o]}},
\end{equation*}
where 
\begin{equation*}
h_i(\mathbf{s}, z_i)= \frac{(e^{z_i}-1)^{\nu-\nu_i}}{z_i^{\mu}}, \hspace{1 cm} h_o(\mathbf{s}, z_o)= \frac{(e^{z_o}-1)^{\nu-\nu_o}}{z_o^{\mu}},
\end{equation*}
and
\begin{equation*}
\eta(\mathbf{s})=F_1(\mathbf{s})+F_2(\mathbf{s})\,\frac{\mu}{2\nu}.
\end{equation*}
Uniformly over $\mathbf{s} \in \mathbf{S}_{\epsilon}$, 
$\eta(\mathbf{s})$, $z_i(\mathbf{s})$ and $z_o(\mathbf{s})$ are bounded away from $\infty$. So the variances of $Z^i$ and $Z^o$ are bounded as well. Hence
\begin{equation*}
\frac{1}{g(\mathbf{s})} \lesssim  \frac{\nu}{\mu!} \frac{z_i^{\mu}}{(e^{z_i}-1)^{\nu-\nu_i}} \frac{z_o^{\mu}}{(e^{z_o}-1)^{\nu-\nu_o}}.
\end{equation*}
From Lemma \ref{bad g(s) estimate} using $x=z_i(\mathbf{s})$ and $y=z_o(\mathbf{s})$, we also have that 
\begin{equation*}
g(\mathbf{s}') \leq (\mu-k)! \frac{(e^{z_i}-1)^{\nu-\nu_i-b-r_i}}{z_i^{\mu-k}} \frac{(e^{z_o}-1)^{\nu-\nu_o-a-r_o}}{z_o^{\mu-k}},
\end{equation*}
where, crucially,  $z_i=z_i(\mathbf{s})$ and $z_o=z_o(\mathbf{s})$ rather than $z_i(\mathbf{s}')$ and $z_o(\mathbf{s}')$. Consequently
\begin{equation*}
\frac{g(\mathbf{s}')}{g(\mathbf{s})} \lesssim  \frac{\nu}{(\mu)_k} \frac{(z_i z_o)^k}{(e^{z_i}-1)^{b+r_i}(e^{z_o}-1)^{a+r_o}}.  
\end{equation*}

$\mathbf{(ii)}$ Consider, for instance, the case $a=1, \,r_o=0$ and $\max\{1,b+r_i\} \leq k \leq \ln n$.
Here $\mathbf{s}'$ also meets the conditions of Theorem \ref{thm g(s) asymptotics}, and so
we have the asymptotic formula \eqref{g(s)sim} for $g(\mathbf{s}^\prime)$. To estimate 
sharply $g(\mathbf{s}^\prime)/g(\mathbf{s})$, we need to look closely at the difference 
between $z_{i,o}(\mathbf{s})$ and $z_{i,o}(\mathbf{s}^\prime)$, and between 
$h_{i,o}\bigl(\mathbf{s},z_{i,o}(\mathbf{s})\bigr)$ and $h_{i,o}\bigl(\mathbf{s}^\prime,z_{i,o}(\mathbf{s}^\prime)\bigr)$.

First of all, 
\begin{equation*}
\nu'=\nu+O(\ln n), \quad \nu'_i=\nu_i+O(\ln n),\quad \nu_o'=\nu_o+O(\ln n),\quad \mu'=\mu+O(\ln n).
\end{equation*} 
Then, denoting $z_{i,o}^\prime=z_{i,o}(\mathbf{s}^\prime)$, we have that 
\begin{equation*}
\ell(z_i')=\frac{\mu'}{\nu'-\nu_i'} = \frac{\mu-k}{\nu-\nu_i-b-r_i} = \frac{\mu}{\nu-\nu_i}+O( (\ln n)/n)=\ell(z_i)+O( (\ln n)/n),
\end{equation*}
where as before $\ell(z) = \frac{ze^z}{e^z-1}.$ Since $\ell'(z)\in [1/2,1]$  for $z\geq0$, it follows
then that $z'_i=z_i+O( (\ln n)/n)$, and similarly, $z'_o=z_o+O( (\ln n)/n)$. Therefore,
denoting $(Z^{i,o})'=Z^{i,o}(\mathbf{s}^\prime)$, $\eta^\prime=\eta(\mathbf{s}^\prime)$, 
\begin{equation*}
Var[(Z^i)'] = Var[Z^i] (1+O( (\ln n)/n)), \quad Var[(Z^o)'] =Var[Z^o](1+O( (\ln n)/n)),
\end{equation*}
and
\begin{equation*}
\eta' = \eta+O( (\ln n)/n).
\end{equation*}
Consequently,
\begin{equation*}
\frac{g(\mathbf{s}')}{g(\mathbf{s})} = \left( 1 + O \left( (\ln n)^6/n\right)\right) \frac{1}{\mu^k} \frac{h_i(\mathbf{s}', z'_i) }{ h_i(\mathbf{s}, z_i) } \cdot \frac{h_o(\mathbf{s}', z'_o)}{ h_o(\mathbf{s}, z_o)}.
\end{equation*}
Notice that
\begin{equation*}
\frac{h_i(\mathbf{s}', z_i')}{h_i(\mathbf{s}, z_i)} = \frac{ (e^{z_i'}-1)^{\nu'-\nu'_i}/(z_i')^{\mu'}}{ (e^{z_i}-1)^{\nu'-\nu'_i}/(z_i)^{\mu'}} \frac{z_i^k}{(e^{z_i}-1)^{b+r_i}}=\frac{h_i(\mathbf{s}', z_i')}{h_i(\mathbf{s}', z_i)} \frac{z_i^k}{(e^{z_i}-1)^{b+r_i}}
\end{equation*}
and
\begin{equation*}
\frac{h_o(\mathbf{s}', z_o')}{h_o(\mathbf{s}, z_o)} = \frac{ (e^{z_o'}-1)^{\nu'-\nu'_o}/(z_o')^{\mu'}}{ (e^{z_o}-1)^{\nu'-\nu'_o}/(z_o)^{\mu'}} \frac{z_o^k}{e^{z_o}-1}=\frac{h_o(\mathbf{s}', z_o')}{h_o(\mathbf{s}', z_o)} \frac{z_o^k}{e^{z_o}-1}.
\end{equation*}
To complete the proof, it suffices to show that, uniformly over $\mathbf{s}$ and $\mathbf{s}'$ in question, 
\begin{equation}\label{f'/f}
\frac{h_i(\mathbf{s}', z_i')}{h_i(\mathbf{s}', z_i)} = \left(1+O\left( (\ln n)^2/n\right)\right), \quad \frac{h_o(\mathbf{s}', z_o')}{h_o(\mathbf{s}', z_o)} = \left(1+O\left( (\ln n)^2/n\right)\right).
\end{equation}
To this end, we expand the exponent of $h_i(\mathbf{s}', z)$ about $z=z_i$: 
\begin{align*}
h_i(\mathbf{s}', z_i') &= \exp \left( (\nu'-\nu_i') \ln \left(e^{z_i'}-1\right) - \mu' \ln z_i' \right) \\&= \exp \left( (\nu'-\nu'_i) \ln \left( e^{z_i}-1\right) - \mu' \ln z_i\right) \\& \times \exp\left( \left( (\nu'-\nu_i') \frac{e^{z_i}}{e^{z_i}-1} - \mu'/z_i\right) (z_i'-z_i) + O\left( \nu \cdot (z_i'-z_i)^2\right)\right).
\end{align*}
Since $z_i'-z_i=O( (\ln n)/n)$, the big-Oh term is $O( (\ln n)^2/n)$. Also, by definition of $z_i$,
\begin{align*}
(\nu'-\nu'_i) \frac{e^{z_i}}{e^{z_i}-1} - \frac{\mu'}{z_i} &= (\nu-\nu_i)\frac{e^{z_i}}{e^{z_i}-1} - \frac{\mu}{z_i} + O( \ln n) \\
&= \frac{\nu-\nu_i}{z_i} \left( \ell(z_i) - \frac{\mu}{\nu-\nu_i} \right) + O(\ln n)=O(\ln n).
\end{align*}
So the first equation in \eqref{f'/f} follows. The second equation is proved similarly.
\end{proof}

Besides $g(\mathbf{s}')/g(\mathbf{s})$, the factors  in formulas \eqref{Ps'|s} and  for $P_i$ and $P_o$ are binomials.  In the case of $P_i$, they are:
\begin{equation*}
{\nu-\nu_i-b-r_i \choose k-b-r_i}, {\nu-\nu_i-\nu_o \choose r_i}, \text{ and } {\nu_o \choose b}.
\end{equation*}
For $\mathbf{s}$ and $\Delta \mathbf{s}$ that meet the conditions of Lemma \ref{g(s) ratio asymptotics} {\bf (ii)}, we have that \begin{align}\label{eqn: 10.21.1}
{\nu-\nu_i-b-r_i \choose k-b-r_i} &= \frac{(\nu-\nu_i)^{k-b-r_i}}{(k-b-r_i)!} \prod_{i=0}^{k-b-r_i-1} \frac{\nu-\nu_i-b-r_i-i}{\nu-\nu_i}\nonumber \\ &= \frac{(\nu-\nu_i)^{k-b-r_i}}{(k-b-r_i)!} \left(1+O\left( (\ln n)^2/n\right)\right),
\end{align}
and similarly,
\begin{equation}\label{eqn: 10.21.2}
{\nu-\nu_i-\nu_o \choose r_i}=\frac{(\nu-\nu_i-\nu_o)^{r_i}}{r_i!} \left(1+O\left( (\ln n)^2/n \right)\right).
\end{equation}
However, we leave ${\nu_o \choose b}$ as it stands, since close to the end of the deletion process, $\nu_o$, the number of vertices with out-degree zero, can not be expected to be much larger than $b$, the number of out-degree zero vertices deleted in one step.

Motivated by these asymptotic expressions and the sharp asymptotics of the ratio $g(\mathbf{s}')/g(\mathbf{s})$ in Lemma \ref{g(s) ratio asymptotics}, we define the following auxiliary transition ``probabilities": 
\begin{equation}\label{eqn def: qi}
\begin{aligned}
q_i(\Delta\mathbf{s}|\mathbf{s}) &:= \frac{\nu_i}{\nu_i+\nu_o}\frac{1}{e^{z_o}-1} {\nu_o \choose b} \left( \frac{1}{e^{z_i}-1} \frac{z_i z_o}{\mu}\right)^b \frac{1}{r_i!} \left( \frac{\nu-\nu_i-\nu_o}{e^{z_i}-1} \frac{z_i z_o}{\mu}\right)^{r_i} \\  &\times \frac{1}{(k-b-r_i)!} \left( (\nu-\nu_i) \frac{z_i z_o}{\mu}\right)^{k-b-r_i}, 
\end{aligned}
\end{equation}
if $\Delta \mathbf{s}$ is such that $a=1, r_o=0$ and $k \geq \max \{1, b+r_i\}$, and $q_i=0$ for any other $\Delta \mathbf{s}$.  Likewise
\begin{align*}
q_o(\Delta \mathbf{s}|\mathbf{s}) &:= \frac{\nu_o}{\nu_i+\nu_o} \frac{1}{e^{z_i}-1} {\nu_i \choose a}\left( \frac{1}{e^{z_o}-1} \frac{z_i z_o}{\mu}\right)^a \frac{1}{r_o!} \left( \frac{\nu-\nu_i-\nu_o}{e^{z_o}-1} \frac{z_i z_o}{\mu}\right)^{r_o} \\ &\times \frac{1}{(k-a-r_o)!} \left( (\nu-\nu_o) \frac{z_i z_o}{\mu}\right)^{k-a-r_o},
\end{align*}
if $\Delta \mathbf{s}$ is such that $b=1, r_i=0$ and $k\geq \max \{1, a+r_o\}$, and $q_o=0$ for any other $\Delta \mathbf{s}$. Although it is not immediately apparent, $q_i+q_o$ is a substochastic distribution, meaning that $\sum_{\Delta \mathbf{s}}\bigl[q_i(\Delta \mathbf{s}| \mathbf{s})+q_o(\Delta \mathbf{s}| \mathbf{s})\bigr] <1$, with the probability deficit being exponentially small. Sure enough, 
$q_{i,o}(\Delta \mathbf{s}|\mathbf{s})$ closely approximates $P_{i,o}(\Delta \mathbf{s}|\mathbf{s})$
for $\mathbf{s}$ and $\Delta \mathbf{s}$ that matter.

\begin{lemma}\label{lemma: pi to qi}
$\mathbf{(i)}$ Uniformly over $\mathbf{s} \in \mathbf{S}_{\epsilon}$, defined in \eqref{eqn: def of s eps}, and $\Delta \mathbf{s}$ with $k \leq \ln n$, 
\begin{align*}
q_i(\Delta \mathbf{s}|\mathbf{s}) &= P_i(\Delta \mathbf{s}|\mathbf{s}) \left(1+O\left( (\ln n)^6/n\right)\right), \\ q_o(\Delta \mathbf{s}|\mathbf{s}) &= P_o(\Delta \mathbf{s}|\mathbf{s})\left(1+O\left( (\ln n)^6/n\right)\right).
\end{align*} 

$\mathbf{(ii)}$ Uniformly over $\mathbf{s} \in \mathbf{S}_{\epsilon}$, 
\begin{equation*}
\sum_{\Delta \mathbf{s} : k \geq \ln n} q_i(\Delta \mathbf{s}| \mathbf{s}), \sum_{\Delta \mathbf{s} : k \geq \ln n} P_i(\Delta \mathbf{s} | \mathbf{s} ) \leq \exp \left(-\frac{2}{3}(\ln n) (\ln \ln n) \right)
\end{equation*}
and
\begin{equation*}
\sum_{\Delta \mathbf{s} : k \geq \ln n} q_o(\Delta \mathbf{s}| \mathbf{s}), \sum_{\Delta \mathbf{s} : k \geq \ln n} P_o(\Delta \mathbf{s} | \mathbf{s} ) \leq \exp \left(-\frac{2}{3}(\ln n )(\ln \ln n) \right).
\end{equation*}
\end{lemma}

\begin{proof}
By symmetry, we need only consider $q_i$ and $P_i$ for both parts of the lemma.

$\mathbf{(i)}$ This equality is immediate from part $\mathbf{(ii)}$ of Lemma \ref{g(s) ratio asymptotics} and the binomial expression approximations \eqref{eqn: 10.21.1} and \eqref{eqn: 10.21.2}.

$\mathbf{(ii)}$  First, by \eqref{Ps'|s}, we have that
\begin{equation}\label{threeb}
P_i(\Delta \mathbf{s}| \mathbf{s}) \leq \frac{g(\mathbf{s}')}{g(\mathbf{s})} {\nu-\nu_i-b-r_i \choose k-b-r_i} {\nu_o \choose b}{\nu-\nu_i-\nu_o\choose r_i}  \leq \frac{g(\mathbf{s}')}{g(\mathbf{s})}\,
2^k { \nu-\nu_i \choose k}.
\end{equation}
Indeed, the product of the second and third  binomials in \eqref{threeb} is at most $\binom{\nu-\nu_i}{b+r_i}$,
the number of ways to sample $b+r_i$ balls out of the urn with $\nu_o$ blue balls and 
$\nu-\nu_i-\nu_o$ red balls. And $\binom{\nu-\nu_i}{b+r_i}$ times the first binomial in \eqref{threeb} is at most $2^k\binom{\nu-\nu_i}{k}$, the total number of ways to sample $k$ balls out of the urn with $\nu-\nu_i$ colorless balls {\it and} color each of sampled balls either white or green.

So,  using the bound for $g(\mathbf{s}')/g(\mathbf{s})$ from part $\mathbf{(i)}$ of Lemma \ref{g(s) ratio asymptotics}, we have that
\begin{align*}
P_i(\Delta \mathbf{s}| \mathbf{s}) & \leq \frac{\nu}{(\mu)_k} \frac{(2z_i z_o)^k}{(e^{z_i}-1)^{b+r_i} (e^{z_o}-1)} \frac{(\nu)_k}{k!} \leq \frac{\nu}{k!}  \frac{(2z_i z_o)^k}{(e^{z_i}-1)^{b+r_i}(e^{z_o}-1)}. 
\end{align*}
Since $\mathbf{s} \in \mathbf{S}_{\epsilon}$,  $z_i$ and $z_o$ are bounded above by some fixed $A>0$,  see Fact \ref{remark: 2.24.1}. Using $(e^x-1)^{-1} < x^{-1}$, we have then that 
\begin{equation*}
P_i(\Delta \mathbf{s}|\mathbf{s})  \leq \frac{\nu}{k!}\,2^k  (z_i)^{k-b-r_i} (z_o)^{k-1} \leq \frac{\nu (2A^2)^k}{k!},
\end{equation*}
uniformly for $b,r_i$ with $b+r_i\le k$. Consequently
\begin{equation*}
\sum_{b \ge 0} \sum_{r_i\ge 0} \sum_{k\ge \max\{b+r_i, \ln \nu\}} P_i(\Delta \mathbf{s}| \mathbf{s}) =O\left( \nu^2\, \frac{(2A^2)^{\ln \nu}}{\lfloor\ln \nu\rfloor!} \right).
\end{equation*}
Using Stirling's formula on $\lfloor \ln \nu \rfloor!$ yields
\begin{align*}
\sum_{\Delta \mathbf{s}: k \geq \ln \nu} P_i(\Delta \mathbf{s}|\mathbf{s})& \lesssim   \nu^3 \frac{(2e A^2)^{\ln \nu}}{(\ln \nu)^{\ln \nu}} = \exp \left(-(\ln \nu) (\ln \ln \nu) + O(\ln \nu)\right)\\
&=\exp\left(-(\ln n)(\ln\ln n) +O\bigl(\ln n\bigr)\right),
\end{align*}
uniformly for $\mathbf{s}\in \mathbf{S}_{\epsilon}$. The other three inequalities can be proved in a similar fashion.
\end{proof}

\subsection{Estimating expectations}\label{sec: estimating exps}

Lemma \ref{lemma: pi to qi} allows us to sharply approximate the (conditional) expected 
values of all five parameters $a, b, r_i, r_o, k$ and all $15$ pairwise products of these parameters
by the ``expected'' values with respect to the approximation $q(\Delta\mathbf{s}|\,\mathbf{s}):=q_i(\Delta\mathbf{s}|\mathbf{s})+q_o(\Delta\mathbf{s}|\mathbf{s})$.
These approximations will be instrumental for the key proofs. For brevity, given any function $f$
of parameters $a,b,r,t,k$,  we denote the corresponding expected values
$E_{\mathbf{s}}[f]$ and $E_{\mathbf{s}}^q[f]$.
To be sure, $E_{\mathbf{s}}^q[f]$ is not exactly the expected value of $f$ since $q(\Delta\mathbf{s}|\mathbf{s})$ is substochastic. 
\begin{lemma}\label{change} If $f$ is a linear or quadratic function of
the components of $\Delta\mathbf{s}$, 
\begin{equation*}
E_{\mathbf{s}}[f] = E_{\mathbf{s}}^q[f] + O( (\ln n)^8/n),
\end{equation*}
uniformly for $\mathbf{s} \in \mathbf{S}_{\epsilon}$, 
\end{lemma}
\begin{proof} Observe that $|f|=O((\ln n)^2)$ if $k\le \ln n$, and $|f|=O(n^2)$ always. By  Lemma \ref{lemma: pi to qi} {\bf (ii)\/}, the contributions to $E_{\mathbf{s}}[f]$ and $E_{\mathbf{s}}^q[f]$ coming from  $k\geq \ln n$ are on order at most 
\begin{equation*}
n^3 \exp \left( -\frac{2}{3} \ln n (\ln \ln n) \right) \ll \frac{(\ln n)^8}{n}.
\end{equation*}
Furthermore, by Lemma \ref{lemma: pi to qi} {\bf (i)\/}, and the first condition on $f$,
\begin{multline*}
\left|\sum_{\Delta \mathbf{s}: k \leq \ln n} f(a,b,r_i,r_o,k) P_{i,o}(\Delta \mathbf{s}|\mathbf{s}) - \sum_{\Delta \mathbf{s} : k \leq \ln n} f(a,b,r_i,r_o,k)q_{i,o}(\Delta \mathbf{s}|\mathbf{s})\right|\\
 \lesssim   \sum_{\Delta \mathbf{s}: k \leq \ln n} (\ln n)^2 \big|q_{i,o}(\Delta \mathbf{s}|\mathbf{s})-P_{i,o}(\Delta \mathbf{s}|\mathbf{s})\big|\\
 \lesssim  \frac{(\ln n)^8}{n}  \sum_{\Delta \mathbf{s}}  P_{i,o}( \Delta \mathbf{s}|\mathbf{s}) \leq \frac{ (\ln n)^8}{n}.
 \end{multline*}
\end{proof}
Importantly, for each $E_{\mathbf{s}}^q[f]$ we can find an explicit $\mathcal{E}_{\mathbf{s}}[f]$
such that 
\begin{equation*}
E_{\mathbf{s}}^q[f]=\mathcal{E}_{\mathbf{s}}[f] +O(1/n),
\end{equation*}
uniformly for $\mathbf{s} \in \mathbf{S}_{\epsilon}$.
Here is how. Introduce the trivariate (probability) generating function of the parameters $b,r_i,k$ with respect to $q_i$,
\begin{align*}
F(x,y,w):&= \sum_{b\ge 0} \sum_{r_i\ge 0} \sum_{k \geq \max\{1,b+r_i\}} x^b y^{r_i} w^k q_i(\Delta \mathbf{s}|\mathbf{s}). 
\end{align*}
Even though the definition of $q_i$ in  \eqref{eqn def: qi} is a bit forbidding, a simple algebraic
computation---we encourage the reader to do it--- shows that
\begin{align*}
F(x,y,w)&=\frac{\nu_i}{\nu_i+\nu_o}\frac{1}{e^{z_o}-1} \left  \{ \left( 1 + \frac{z_i z_o xw}{\mu(e^{z_i}-1)} \right)^{\nu_o}  \right. \\&  \left. \times \exp \left( \frac{(\nu-\nu_i-\nu_o)z_i z_o y w}{(e^{z_i}-1)\mu} + \frac{(\nu-\nu_i)z_i z_o w}{\mu} \right)-1 \right \}.
\end{align*}
Armed with this formula, we obviously can compute the $q_i$-expected values of the functions $f$ in
question by evaluating $F$ and its partial derivatives at $x=y=w=1$.  To simplify the resulting $q_i$-expectations, we  will use 
\begin{equation}\label{eqn: 2.26.1}
\left( 1 + \frac{z_i z_o}{\mu(e^{z_i}-1)} \right)^{\nu_o} = \exp \left( \frac{ \nu_o z_i z_o}{\mu(e^{z_i}-1)} \right) + O(n^{-1})
\end{equation}
uniformly over $\mathbf{s} \in \mathbf{S}_{\epsilon}$.  For ease of notation, each of the following sums are over $\Delta \mathbf{s}$ where $a=1, r_o=0,$ $b \geq 0, r_i \geq 0$ and $k\geq \max\{1, b+r_i\}$. Necessarily, $q_i=0$ unless $r_o=0$; so
$\sum_{\Delta \mathbf{s}}  r_o q_i(\Delta\mathbf{s}|\mathbf{s})=0$. Further,
$q_i=0$ unless $a=1$, in which case
\begin{align*}
\sum_{\Delta \mathbf{s}} a q_i (\Delta \mathbf{s}| \mathbf{s}) &= \sum_{\Delta \mathbf{s}} q_i( \Delta \mathbf{s}| \mathbf{s})=F(1,1,1) \\&= \frac{\nu_i}{\nu_i+\nu_o} \frac{1}{e^{z_o}-1} \left \{ \exp \left( \frac{\nu_o z_i z_o}{(e^{z_i}-1) \mu}+ \frac{(\nu-\nu_i-\nu_o)z_i z_o}{(e^{z_i}-1)\mu}\right. \right. \\& \left. \left. + \frac{(\nu-\nu_i)z_i z_o}{\mu} \right) - 1 \right \} + O(n^{-1}),
\end{align*}
uniformly over $\mathbf{s} \in \mathbf{S}_{\epsilon}$. We note that
 \begin{equation*}
\frac{\nu_o z_i z_o}{(e^{z_i}-1) \mu} + \frac{(\nu-\nu_i-\nu_o)z_i z_o}{(e^{z_i}-1)\mu} + \frac{(\nu-\nu_i)z_i z_o}{\mu} = \frac{(\nu-\nu_i)z_i z_o}{\mu} \left( \frac{1}{e^{z_i}-1}+1 \right) = z_o,
\end{equation*}
since $z_i e^{z_i}/(e^{z_i}-1)=\ell(z_i)=\mu/(\nu-\nu_i).$ Hence
\begin{align*}
\sum_{\Delta \mathbf{s}} a q_i(\Delta \mathbf{s}|\mathbf{s}) = \frac{\nu_i}{\nu_i+\nu_o} + O(n^{-1}).
\end{align*}
To estimate the $q_i$-averages of the remaining $b,r_i,k$, we evaluate the
corresponding  partial derivatives of $F$ at $(1,1,1)$. For example,
\begin{align*}
\sum_{\Delta \mathbf{s}} b q_i(\Delta \mathbf{s}| \mathbf{s}) &= 
\sum_{\Delta \mathbf{s}} x^b y^{r_i} w^k q_{i}(\Delta \mathbf{s}|\mathbf{s}) \Big|_{x=y=w=1} = F_x(1,1,1) \\&= \frac{\nu_i}{\nu_i+\nu_o} \frac{1}{e^{z_o}-1} \left \{ \nu_o \left( 1 + \frac{z_i z_o}{(e^{z_i}-1)\mu}\right)^{\nu_o-1} \frac{z_i z_o}{(e^{z_i}-1)\mu} \right. \\& \left. \times \exp \left( \frac{(\nu-\nu_i-\nu_o)z_i z_o}{(e^{z_i}-1)\mu} + \frac{(\nu-\nu_i)z_i z_o}{\mu} \right) \right \}.
\end{align*}
Using the asymptotic expression \eqref{eqn: 2.26.1} as well as the definition of $z_i$ and $z_o$, we can rewrite the above RHS as
\begin{equation*}
\sum_{\Delta \mathbf{s}} b q_i (\Delta \mathbf{s} | \mathbf{s}) = \frac{\nu_i \nu_o \mu e^{-z_i}}{(\nu_i+\nu_o)(\nu-\nu_i)(\nu-\nu_o)} + O(n^{-1}).
\end{equation*} 
Similarly
\begin{equation*}
\sum_{\Delta \mathbf{s}} r_i q_i(\Delta \mathbf{s}|\mathbf{s}) = F_y(1,1,1) = \frac{\nu_i \mu (\nu-\nu_i-\nu_o) e^{-z_i}}{(\nu_i+\nu_o)(\nu-\nu_o)(\nu-\nu_i)} + O(n^{-1}),
\end{equation*}
and
\begin{equation*}
\sum_{\Delta \mathbf{s}} k q_i(\Delta \mathbf{s}|\mathbf{s}) = F_w(1,1,1) = \frac{\nu_i \mu}{(\nu_i+\nu_o)(\nu-\nu_o)}+O(n^{-1}).
\end{equation*}
By symmetry between ``in'' and ``out'', the $q_o$-average of $a$ (resp. $b$) is found by switching $i$ and $o$ in the formula for the $q_i$-average of $b$ (resp. $a$). Similarly, we determine the sums over $q_o$ of the other variables. Thus, the $O(n^{-1})$ error terms aside,
the leading terms for $q$-expectations of the transition parameters  
\begin{equation}\label{Eqabrtk=calEabrtk} 
\begin{aligned}
&E_{\mathbf{s}}^q[a]=\mathcal{E}_{\mathbf{s}}[a] +O(1/n);\quad \mathcal{E}_{\mathbf{s}}[a]:=  \frac{\nu_{i}}{\nu_{i}+\nu_{o}} + \frac{\nu_{i}\nu_{o}\, \mu \, e^{-z_o}}{(\nu_{i}+\nu_{o})(\nu-\nu_i)(\nu-\nu_o)}, \\ 
&E_{\mathbf{s}}^q[b]=\mathcal{E}_{\mathbf{s}}[b]+O(1/n);\quad \mathcal{E}_{\mathbf{s}}[b]:=  \frac{\nu_{o}}{\nu_{i}+\nu_{o}} + \frac{\nu_{i}\nu_{o} \, \mu \, e^{-z_i}}{(\nu_{i}+\nu_{o})(\nu-\nu_i)(\nu-\nu_o)}, \\  
&E_{\mathbf{s}}^q[r_i]=\mathcal{E}_{\mathbf{s}}[r_i]+O(1/n);\quad \mathcal{E}_{\mathbf{s}}[r_i]:= \frac{\nu_{i} \, \mu \, (\nu-\nu_i-\nu_o) e^{-z_i}}{(\nu_{i}+\nu_{o})(\nu-\nu_i)(\nu-\nu_o)}, \\ 
&E_{\mathbf{s}}^q[r_o]=\mathcal{E}_{\mathbf{s}}[r_o]+O(1/n);\quad \mathcal{E}_{\mathbf{s}}[r_o]:= \frac{\nu_{o} \, \mu \, (\nu-\nu_i-\nu_o) e^{-z_o}}{(\nu_{i}+\nu_{o})(\nu-\nu_i)(\nu-\nu_o)}, \\  
&E_{\mathbf{s}}^q[k]=\mathcal{E}_{\mathbf{s}}[k]+O(1/n);\quad\mathcal{E}_{\mathbf{s}}[k] := \frac{\mu}{\nu_{i}+\nu_{o}} \left( \frac{\nu_{o}}{\nu-\nu_{i}} + \frac{\nu_{i}}{\nu-\nu_{o}} \right),
\end{aligned}
\end{equation}
uniformly for $\mathbf{s}\in \mathbf{S}_{\epsilon}$ ($\mathbf{S}_\epsilon$ defined in \eqref{eqn: def of s eps}). Tellingly and importantly, all these functions are zero-degree homogeneous functions of $\mathbf{s}$. Combining Lemma \ref{change} and \eqref{Eqabrtk=calEabrtk}, we have that uniformly over $\mathbf{s} \in \mathbf{S}_\epsilon$ and any $f=a,b,r_i, r_o, k$, 
\begin{equation}\label{Esf=calEf}
E_{\mathbf{s}}[f] = \mathcal{E}_{\mathbf{s}}[f] + O((\ln n)^8/n).
\end{equation}
We naturally extend the approximate expectation to linear combinations of the terms $a, b, r_i, r_o, k$. For instance, $\mathcal{E}_{\mathbf{s}}[\lambda_1 a + \lambda_2 b]:=\lambda_1 \mathcal{E}_{\mathbf{s}}[a]+\lambda_2 \mathcal{E}_{\mathbf{s}}[b]$ for any constants $\lambda_1$ and $\lambda_2$. Consequently, \eqref{Esf=calEf} holds for $f$ being a linear combination of $a,b,r_i,r_o,k$.

The same technique works for the expected values of the $15$ pairwise products $a^2, ab,\\
\dots, k^2$, via evaluating the second order partial derivatives of $F$ at $(1,1,1)$. 
For instance, we find that $q_i$-averages of $f=a^2,b^2$ are given by  
$E_{\mathbf{s},i}[f]=\mathcal{E}_{\mathbf{s},i}[f]+O(1/n)$, where 
 \begin{align*}
\mathcal{E}_{\mathbf{s},i}[a^2]&=\frac{\nu_i}{\nu_i+\nu_o},\\
\mathcal{E}_{\mathbf{s},i}[b^2]&=
\frac{\nu_i}{\nu_i+\nu_o} \frac{e^{z_o}}{e^{z_o}-1} \left( \left(\frac{\nu_o}{e^{z_i}-1} \frac{z_i z_o}{\mu}\right)^2 + \frac{\nu_o}{e^{z_i}-1} \frac{z_i z_o}{\mu} \right).
\end{align*}
Using symmetry, $\mathcal{E}_{\mathbf{s},o}[a^2] $ ($\mathcal{E}_{\mathbf{s},o}[b^2]$ resp.)
 is obtained from $\mathcal{E}_{\mathbf{s},i}[b^2]$ ( $\mathcal{E}_{\mathbf{s},i}[a^2]$ resp.) 
 by switching each instance of $i$ with $o,$ and $o$ with $i$. Hence
\begin{align*}
\mathcal{E}_{\mathbf{s}}[a^2]:=&\mathcal{E}_{\mathbf{s},i}[a^2]+\mathcal{E}_{\mathbf{s},o}[a^2]\nonumber \\
&=\frac{\nu_i}{\nu_i+\nu_o}+ \frac{\nu_o}{\nu_i+\nu_o} \frac{e^{z_i}}{e^{z_i}-1} \left( \left(\frac{\nu_i}{e^{z_o}-1} \frac{z_i z_o}{\mu}\right)^2 + \frac{\nu_i}{e^{z_o}-1} \frac{z_i z_o}{\mu} \right), 
\end{align*}
with $\mathcal{E}_{\mathbf{s}}[b^2]$ obtained from the above expression by switching each $i$ to $o$ and each $o$ to $i$. See Appendix \ref{app: app exp} for the full list of 15 approximate expected values. Further, each of these approximate expectations are zero-degree homogeneous functions of $\mathbf{s}$.  
Needless to say, the asymptotic approximation \eqref{Esf=calEf}
continues to hold, uniformly for $\mathbf{s}\in \mathbf{S}_{\epsilon}$, for all these 
$f$ as well. Namely, we have that
\begin{lemma}\label{lemma: 5.8.15.1}
Uniformly over $\mathbf{s} \in \mathbf{S}_\epsilon$ and any function $f$ that is a linear combination of the linear terms $a,b,r_i,r_o,k$ or their pairwise products $a^2, ab, \ldots, k^2$,
\begin{equation*}
E_{\mathbf{s}}[f]=\mathcal{E}_{\mathbf{s}}[f]+O((\ln n)^8/n).
\end{equation*}
\end{lemma}

Here is a brief summary of what  we have done so far. First, we described the Markovian
deletion process which terminates at the $(1,1)$-core. Next we showed that we can lump together the digraph states with the same foursome $\{\mathbf{s}(t)\}$ without sacrificing the
Markovian nature of the process. Then, we found a tractable asymptotic approximation $q(\Delta\mathbf{s}|
\mathbf{s})$ of the transition probabilities $P(\Delta\mathbf{s}|\mathbf{s})$ for the states $\mathbf{s}$ in the likely range of
the random $\mathbf{s}(0)$ in $D(n, m=c_n n)$. Lastly we found sharp explicit approximations for the 
conditional expectations of the parameters $a,b,r_i,r_o,k$ associated with the random transition from $\mathbf{s}$ to the next state $\mathbf{s}'$, under condition $\mathbf{s}\in \mathbf{S}_{\epsilon}$. We will use these
approximations shortly.

\section{Characteristic Function of the Core Parameters}\label{sec: char fun}

The deletion process $\{D(t)\}$ stops at  $\bar\tau := \min\{ t : \nu_i(t)=\nu_o(t)=0\}$,
the first time there are no semi-isolated vertices, i.e.\  $D(t)\equiv D(\bar\tau)$ for $t\ge \bar\tau$.
In terms of the reduced process $\{\mathbf{s}(t)\}$, $\mathbf{s}(t)\equiv \mathbf{s}(\bar\tau)$
for $t\ge \bar\tau$. Let $\bar{\nu}$ and $\bar{\mu}$ denote the terminal number of vertices and arcs, i.e.\ $\bar{\nu}=\nu(\bar{\tau}), \bar{\mu}=\mu(\bar{\tau})$. Our ultimate goal is to show 
asymptotic Gaussianity  of  the pair $(\bar{\nu}, \bar{\mu})$ for the reduced process that starts at
the random $\mathbf{s}(0)$ in $D(n, m=c_n n)$.  

To this end, we show first that the pair $(\bar{\nu}, \bar{\mu})$ is asymptotically Gaussian
for a deterministic $\mathbf{s}(0)=\mathbf{s}$ belonging to a likely range of initial states $\mathbf{s}(0)$ in $D(n, m=c_n n)$.  Let $\varphi_{\mathbf{s}}(\mathbf{u})$, $\mathbf{u}=(u_1, u_2)^T\in \Bbb R^2$,  denote the joint characteristic function of $(\bar{\nu}, \bar{\mu})$ for a generic, deterministic, $\mathbf{s}(0)=\mathbf{s}$. Formally, $\varphi_{\mathbf{s}}(\mathbf{u}):= E_{\mathbf{s}}[ e^{i u_1 \bar{\nu}+i u_2 \bar{\mu}}]$,
where as before, $E_{\mathbf{s}}[\cdot]$ denotes the expectation according to probability distribution of the process, $\{\mathbf{s}(t)\}$, starting at $\mathbf{s}(0)=\mathbf{s}$.
Since the  process is time-homogeneous, $\varphi_{\mathbf{s}}$ satisfies 
\begin{equation}\label{eqn: recur}
\varphi_{\mathbf{s}}(\mathbf{u})= \sum_{\mathbf{s}'} \varphi_{\mathbf{s}'}(\mathbf{u}) P\left( \mathbf{s}^\prime|\mathbf{s}\right),
\end{equation}
with $P\left( \mathbf{s}^\prime|\mathbf{s}\right)$ denoting the probability of one step-transition from
$\mathbf{s}$ to $\mathbf{s}^\prime$.  We would need to show existence of (smooth) functions
$f_j$, $\psi_{j,k}$ ($j,k=1,2$), of the scaled $\mathbf{s}/n$, such that the $2\times 2$ symmetric matrix $\boldsymbol{\psi}=\{\psi_{j,j}\}_{1\le j,k\le 2}$ is positive-definite, and  
\begin{equation}\label{Gn=}
G_n(\mathbf{s}/n,\mathbf{u}):=\exp \left( i n \mathbf{u}^T \mathbf{f}(\mathbf{s}/n)-\frac{n}{2} \mathbf{u}^T \boldsymbol{\psi}(\mathbf{s}/n) \mathbf{u}\right), \quad \bold f^T:=(f_1,f_2),
\end{equation}
sharply approximates $\varphi_{\mathbf{s}}(\mathbf{u})$ for $\|\bold u\|=O(n^{-1/2})$ (this order of 
$\bold u$ is dictated by our goal of proving asymptotic normality of the numbers of vertices and arcs in the (1,1)-core centered by their respective means and scaled by $n^{1/2}$, the asymptotic order of their standard deviations). 
Necessarily, such $f_j$ and $\psi_{j,k}$ would have to  satisfy the boundary conditions
\begin{equation}\label{eqn: boundary conditions1}
f_1(\alpha,0,0,\gamma)=\alpha,\ \ \ f_2(\alpha, 0, 0, \gamma)=\gamma,\ \ \ \psi_{j,k}(\alpha,0,0,\gamma)=0.
\end{equation}
Indeed, if $\nu_i$ and $\nu_o$ are both zero, our process is already stopped at time $t=0$,
and so $\bar{\nu}=\nu$ and $\bar{\mu}=\mu$. Ideally, we would hope to determine $\bold f$ and
$\boldsymbol{\psi}$ out of the condition that $G_n(\mathbf{s}/n,\mathbf{u})$ ``almost''
satisfies the equation \eqref{eqn: recur}, with ``almost'' to be specified shortly. 

This method of 
proving asymptotic normality  with simultaneous determination of the attendant parameters had been used earlier, Pittel~\cite{normal convergence},  \cite{urn model}, and Pittel and Weishaar~\cite{random bipartite neighbor}. 

However, our analysis would inevitably rely on the asymptotic approximation of the transition
probability $P(\Delta\mathbf{s}|\mathbf{s})$ by the sub-stochastic $q(\Delta\mathbf{s}|\mathbf{s})$, established {\it only\/} for $\mathbf{s}\in \mathbf{S}_{\epsilon}$, defined in
\eqref{eqn: def of s eps}; see Lemma \ref{g(s) ratio asymptotics}, Lemma \ref{lemma: pi to qi}. So effectively we are forced to stop the process possibly earlier, at time
\begin{equation}\label{defhattau}
\hat{\tau}=\left\{\begin{aligned}
&\min\{t< \bar\tau:\,\mathbf{s}(t) \notin \mathbf{S}_{\epsilon}\},\\
&\bar\tau,\quad \text{if no such }t\text{ exists}.\end{aligned}\right.
\end{equation}
and to consider instead $(\hat\nu,\hat\mu):=(\nu(\hat\tau),\mu(\hat\tau))$. We will show, however,
that, for $\mathbf{s}(0)$ from a large enough subset of $\mathbf{S}_{\epsilon}$, a.a.s.\  $\hat{\tau}=\bar\tau$. So  asymptotic normality of $(\hat\nu,\hat\mu)$ will imply asymptotic normality
of $(\bar\nu,\bar\mu)$.

We will use $\hat{P}$ and $\hat{E}_{\mathbf{s}}[\circ]$ for the transition probabilities and the expectations, starting at generic state $\mathbf{s}$, for this (still Markovian) modification of the deletion process that freezes at time
$\hat{\tau}$. The
corresponding $\hat{\varphi}_{\mathbf{s}}(\mathbf{u}):=\hat E_{\mathbf{s}}[ e^{i u_1 \hat{\nu}+i u_2 \hat{\mu}}]$
satisfies 
\begin{equation}\label{eqn: p star recur}
\hat{\varphi}_{\mathbf{s}}(\mathbf{u}) = \sum_{\mathbf{s}'} \hat{\varphi}_{\mathbf{s}'}(\mathbf{u}) \hat{P}\left( \mathbf{s}(1)=\mathbf{s}'| \mathbf{s}(0)=\mathbf{s}\right).
\end{equation}
Thus we set up to determine smooth functions $f_j$ and $\psi_{j,k}$, for $1 \leq j,k\leq 2$, such that for $\|\mathbf{u}\| = O(n^{-1/2})$,
and $\mathbf{s}\in\mathbf{S}_{\epsilon}$, the function $G_n(\mathbf{s}/n, \mathbf{u})$ defined
in \eqref{Gn=} almost satisfies \eqref{eqn: p star recur}, namely
\begin{equation}\label{eqn: G approx recur}
G_n(\mathbf{s}/n, \mathbf{u})-\sum_{\mathbf{s}'} G_n(\mathbf{s}'/n, \mathbf{u}) \hat{P}(\mathbf{s}(1)=\mathbf{s}'|\mathbf{s}(0)=\mathbf{s}) = o(1/n).
\end{equation} 
We need this additive error term to be that small as at one point we will have to sum up
those terms over the duration of the process. 

To this end, we plug $G_n(\cdot, \mathbf{u})$ into the sum in \eqref{eqn: G approx recur} and simplify the resulting expression via Taylor-expanding the generic exponents around $\mathbf{s}
$. 
That's where we need and hope for sufficient smoothness---two continuous derivatives would do---of
the functions $\bold f$ and $\boldsymbol{\psi}$. Requiring the resulting LHS of \eqref{eqn: G approx recur} be $o(1/n)$, we will arrive at a system of first-order PDF for $\bold f$ and
$\boldsymbol{\psi}$ and solve it using the ODE for the characteristics of the PDE, thus
establishing existence of the required functions.

Here are the details. Recall that $\mathbf{s}'$ can follow from $\mathbf{s}$ only if 
\begin{equation*}
\Delta \mathbf{s}= \mathbf{s}^\prime-\mathbf{s}=(-a-b, r_i-a, r_o-b, -k)
\end{equation*}
is such that $k\geq \max\{a+r_o, b+r_i\}$ and either $a=1,r_o=0$ or $b=1, r_i=0$. Suppose that
$k\le \ln n$. Assume that the second order derivatives of $\bold f$ and $\boldsymbol{\psi}$
have absolute values bounded on the line segment  $[\mathbf{s},\mathbf{s}^\prime]$. Then, denoting the gradient of $f_j$ at $\mathbf{s}$ by $\nabla f_j$, we have 
\begin{align*}
n u_j f_j (\mathbf{s}'/n) &= n u_j f_j(\mathbf{s}/n)+n u_j (\Delta \mathbf{s}/n)^T  \nabla f_j+O(n |u_j| (\ln n)^2n^{-2}) \\ &= n u_j f_j(\mathbf{s}/n)+u_j \Delta \mathbf{s}^T \nabla f_j + O( (\ln n)^2 n^{-3/2})\\
&=n u_j f_j(\mathbf{s}/n) +O((\ln n) n^{-1/2}),
\end{align*}
as $\|\bold u\|=O(n^{-1/2})$. Likewise
\begin{align*}
n u_j u_k \psi_{j,k}(\mathbf{s}'/n) =& n u_j u_k \psi_{j,k}(\mathbf{s}/n)+u_j u_k \Delta \mathbf{s}^T \nabla \psi_{j,k}+O( (\ln n)^2n^{-2})\\
=&n u_j u_k \psi_{j,k}(\mathbf{s}/n)+O((\ln n)n^{-1}).
\end{align*}
Using multiplicativity of $e^z$, and $e^y=1+y +y^2/2 +O(|y|^3)$, we have then: for $\| \bold u \| = O(n^{-1/2})$, 
\begin{multline*}
\frac{G_n(\mathbf{s}'/n,\mathbf{u})}{G_n(\mathbf{s}/n,\mathbf{u})} = \exp \left( i \sum_j u_j \Delta \mathbf{s}^T \nabla f_j - \frac{1}{2} \sum_{j,k} u_j u_k \Delta \mathbf{s}^T \nabla \psi_{j,k} + O\left(\frac{(\ln n)^2}{n^{3/2}}\right)\right) \\ = 1+ i \sum_j u_j \Delta \mathbf{s}^T \nabla f_j - \frac{1}{2} \sum_{j,k} u_j u_k \left[ \left( \Delta \mathbf{s}^T \nabla f_j\right)\left( \Delta \mathbf{s}^T \nabla f_k\right)+\Delta \mathbf{s}^T \nabla \psi_{j,k} \right] +O\left(\frac{(\ln n)^2}{n^{3/2}}\right).
\end{multline*} 
Here the second and the third term are of order $n^{-1/2}$ and $n^{-1}$ respectively.
So, for $G_n$ to satisfy \eqref{eqn: G approx recur}  for all
$\bold u$ in question, it would suffice that the sums of the second and third terms over $\Delta \mathbf{s}$ are $o(n^{-1/2})$ and $o(1)$, respectively. In other words, 
\begin{equation}\label{eqn: first diff eq}
\left(\sum_{\Delta\mathbf{s}} \Delta \mathbf{s}^T \hat{P}(\Delta\mathbf{s}|\mathbf{s})\right) \nabla f_j = \hat{E}_{\mathbf{s}}[\Delta \mathbf{s}^T] \nabla f_j = o(n^{-1/2}),
\end{equation}
and 
\begin{equation}\label{eqn: second diff eq}
\hat{E}_{\mathbf{s}}[\Delta \mathbf{s}^T] \nabla \psi_{j,k} + \hat{E}_{\mathbf{s}}[\left(\Delta \mathbf{s}^T \nabla f_j\right)\left(\Delta \mathbf{s}^T \nabla f_k\right)] =o(1).
\end{equation}
The reason we haven't replaced those $o(n^{-1/2})$ by $0$ is that we still have to approximate 
the expected values in these two equations using Lemma \ref{change}. At any rate, these are 
PDE-type equations  for $\mathbf{f}$ and $\boldsymbol{\psi}$ that have to be solved under 
 the boundary conditions \eqref{eqn: boundary conditions1}. 
 

\section{PDE}\label{sec: solving diff eqn}

\subsection{PDE for the mean parameters}\label{sec: approx PDE fj}

For the PDE-type equation \eqref{eqn: first diff eq}, we need to estimate 
\begin{equation*}
\hat{E}_{\mathbf{s}}[\Delta \mathbf{s}^T] = (\hat{E}_{\mathbf{s}}[-a-b],\, \hat{E}_{\mathbf{s}}[r_i-a],\, \hat{E}_{\mathbf{s}}[r_o-b],\, \hat{E}_{\mathbf{s}}[-k]).
\end{equation*}
By the definition of $\hat P$,  $\hat{E}_{\mathbf{s}}[\Delta \mathbf{s}^T]=\bold 0$,
if $\mathbf{s} \notin \mathbf{S}_{\epsilon}$, and $\hat{E}_{\mathbf{s}}[\Delta \mathbf{s}^T] =
E_{\mathbf{s}}[\Delta \mathbf{s}^T]$ for $\mathbf{s}\in \mathbf{S}_{\epsilon}$.  
By Lemma \ref{lemma: 5.8.15.1}, we have
\begin{equation}\label{eqn: 10.19.5}
E_{\mathbf{s}}[\Delta \mathbf{s}^T] =\mathcal{E}_{\mathbf{s}}[\Delta \mathbf{s}^T] + O( (\ln n)^8/n),
\end{equation}
uniformly for $\mathbf{s}\in \mathbf{S}_{\epsilon}$.  The four coefficients by the partial
derivatives of $f_j$ are computed, by linearity, via \eqref{Eqabrtk=calEabrtk}. 

In order to establish the existence of  $f_j$ satisfying \eqref{eqn: first diff eq}, it suffices to prove existence of $f_j$ that satisfy the {\it homogeneous} partial differential equation: 
\begin{equation}\label{the diffy q}
\begin{aligned}
&\textbf{ PDE:} 
\\ 
&\mathcal{E}_{\mathbf{s}}[-(a+b)] (f_j)_\alpha + \mathcal{E}_{\mathbf{s}}[r_i-a] (f_j)_{\beta_i} + \mathcal{E}_{\mathbf{s}}[r_o-b] (f_j)_{\beta_o} + \mathcal{E}_{\mathbf{s}}[-k] (f_j)_\gamma = 0,
\end{aligned}
\end{equation}
$\mathbf{s}:=n(\alpha,\beta_i,\beta_o,\gamma)$, under the boundary conditions: $f_1(\alpha,0,0,\gamma)=\alpha$ and $f_2(\alpha,0,0,\gamma)=\gamma.$ 
This PDE equation is zero-degree homogeneous, since the coefficients depend on $(\alpha,\beta_i,\beta_o,
\gamma)$ only.

\subsection{The Characteristics for the PDE}\label{sec: char for fj}

The characteristics of the first-order PDE \eqref{the diffy q} are the trajectories of  the following system of ODEs:
\begin{align*}
\frac{d \alpha}{dt} = \mathcal{E}_{\mathbf{s}}[-a-b],\ \ \frac{d \beta_i}{dt} = \mathcal{E}_{\mathbf{s}}[r_i-a],\ \ \frac{d \beta_o}{dt} = \mathcal{E}_{\mathbf{s}}[r_o-b],\ \ \frac{d \gamma}{dt} = \mathcal{E}_{\mathbf{s}}[-k].
\end{align*}
Using  \eqref{Eqabrtk=calEabrtk}, after minor simplifications we obtain the ODE in an explicit form:
\begin{equation}\label{eqns: characteristics}
\begin{aligned}
\frac{d \alpha}{d t} &= - 1 - \frac{\beta_{i} \beta_{o}\, \gamma \, (e^{-z_i}+e^{-z_o})}{(\beta_{i}+\beta_{o})(\alpha-\beta_i)(\alpha-\beta_o)}, \\ 
\frac{d \beta_{i}}{d t} &=  \frac{\beta_{i}}{\beta_{i}+\beta_{o}} \left( \frac{\gamma(\alpha-\beta_i-\beta_o) \, e^{-z_i}}{(\alpha-\beta_i)(\alpha-\beta_o)}  - 1 - \frac{\beta_o \, \gamma \, e^{-z_o}}{(\alpha-\beta_i)(\alpha-\beta_o)} \right), \\ 
\frac{d \beta_{o}}{d t} &=  \frac{\beta_{o}}{\beta_{i}+\beta_{o}}  \left( \frac{\gamma(\alpha-\beta_i-\beta_o) \, e^{-z_o}}{(\alpha-\beta_i)(\alpha-\beta_o)}  - 1 - \frac{\beta_i \, \gamma \, e^{-z_i}}{(\alpha-\beta_i)(\alpha-\beta_o)} \right),  \\ 
\frac{d\gamma}{d t} &= - \frac{\gamma}{\beta_{i}+\beta_{o}} \left( \frac{\beta_{o}}{\alpha-\beta_{i}} + \frac{\beta_{i}}{\alpha-\beta_{o}} \right). 
\end{aligned}
\end{equation}
Here $z_i=z_i(\alpha, \beta_i, \beta_o, \gamma)$ and $z_o=z_o(\alpha, \beta_i, \beta_o, \gamma)$ are defined by 
\begin{equation}\label{zi,zo=}
\ell(z_i):=\frac{z_i e^{z_i}}{e^{z_i}-1} = \frac{\gamma}{\alpha-\beta_i}, \hspace{1 cm} 
\ell(z_o):=\frac{z_o e^{z_o}}{e^{z_o}-1} = \frac{\gamma}{\alpha-\beta_o}.
\end{equation}
Note that this fits with our previous definition of $z_i(\mathbf{s})$ and $z_o(\mathbf{s})$ (Lemma \ref{zi,zo}) since the RHS of \eqref{zi,zo=} is zero-degree homogeneous.  The differential equations above are certainly well defined for $\mathbf{w}=(\alpha,\beta_i,\beta_o,\gamma)\ge \bold 0$ satisfying 
\begin{equation}\label{well}
\beta_i,\beta_o >0;\,\beta_i+\beta_o<\alpha;\quad \frac{\gamma}{\alpha-\beta_i},\,\frac{\gamma}{\alpha-\beta_o} >1.
\end{equation}
This seemingly formidable system has two remarkable integrals, the explicit state-functions  that remain constant along every characteristic. 

In vector notation, this system becomes  
\begin{equation}\label{Hw}
\frac{d \mathbf{w}}{dt} = \mathcal{E}_{\mathbf{s}}[\Delta \mathbf{s}] =: \mathbf{H}(\mathbf{w}).
\end{equation}

Suppose $\mathbf{w}_0$ is such that $\mathbf{H}(\mathbf{w})$ has bounded partial derivatives in
a ball centered at $\mathbf{w}_0$. Then the system \eqref{Hw} has a unique solution 
$\mathbf{w}(t)$, $\mathbf{w}(0)=\mathbf{w}_0$, defined for $t\in [0,t(\mathbf{w}_0)]$, for some
$t(\mathbf{w}_0)>0$ ($t(\mathbf{w}_0)$ could be infinite).

\begin{proposition}\label{first constant} Let $I_2(\mathbf{w})=z_iz_o/\gamma$. For 
$t\le t(\mathbf{w}_0)$, we have $I_2(\mathbf{w}(t))\equiv I_1(\mathbf{w}(0))$.
\end{proposition}
\begin{proof}
We need to compute $dz_i/dt$ and $dz_o/dt$. By \eqref{zi,zo=},
\begin{equation*}
\frac{dz_i}{dt}= \frac{1}{k^\prime(z_i)}\cdot\frac{d}{dt}\left(\frac{\gamma}{\alpha-\beta_i}\right)
=\frac{1}{k^\prime(z_i)}\cdot\frac{\gamma^\prime(\alpha-\beta_i)-\gamma(\alpha^\prime-
\beta_i^\prime)}{(\alpha-\beta_i)^2}.
\end{equation*}
Plugging in the expressions for $\alpha^\prime$, $\beta_i^\prime$ and $\gamma^\prime$ from
\eqref{eqns: characteristics} and simplifying, we arrive at a surprisingly simple formula,
\begin{equation}\label{eqn: 1 zi d zi}
\frac{d z_i}{dt} =  - \frac{\beta_i z_i}{(\beta_i+\beta_o)(\alpha-\beta_o)},
\end{equation}
Likewise
\begin{equation}\label{eqn: 1 zo d zo}
\frac{d z_o}{dt} = -\frac{\beta_o z_o}{(\beta_i+\beta_o)(\alpha-\beta_i)}.
\end{equation}
Using the formula for $d \gamma/dt$ from \eqref{eqns: characteristics}, we see that
\begin{equation*}
\frac{1}{z_i}\frac{dz_i}{dt} +\frac{1}{z_o}\frac{dz_o}{dt} - \frac{1}{\gamma}\frac{d \gamma}{dt}=0,
\end{equation*}
or simply
\begin{equation*}
\frac{d}{dt} \left[ \ln z_i + \ln z_o - \ln \gamma \right]=0\Longrightarrow \frac{z_i(t) z_o(t)}{\gamma(t)}\equiv \text{constant}.
\end{equation*}
\end{proof}

\begin{proposition}\label{second constant} Let $I_1(\mathbf{w})=\gamma(\alpha-\beta_i-\beta_o)/
[(\alpha-\beta_i)(\alpha-\beta_o)]$. 
For $t\le t(\mathbf{w}_0)$, we have $I_1(\mathbf{w}(t))\equiv I_1(\mathbf{w}(0))$.
\end{proposition}

\begin{proof}
A straightforward computation of $d(\alpha-\beta_i-\beta_o)/dt$ from \eqref{eqns: characteristics}, followed by minor simplification, yields
\begin{equation*}
\frac{1}{\alpha-\beta_i-\beta_o} \frac{d (\alpha-\beta_i-\beta_o)}{dt} = \frac{1}{e^{z_i}-1}\frac{dz_i}{dt} + \frac{1}{e^{z_o}-1} \frac{dz_o}{dt},
\end{equation*}
which can be rewritten as
\begin{equation*}
\frac{d}{dt} \left[ \ln \left( \alpha-\beta_i-\beta_o\right)-\ln\left( 1-e^{-z_i}\right)-\ln \left(1-e^{-z_o}\right)\right]=0.
\end{equation*}
So
\begin{equation*}
\ln\left[ (\alpha-\beta_{i}-\beta_{o}) \frac{e^{z_i}}{e^{z_i}-1} \frac{e^{z_o}}{e^{z_o}-1} \right] \equiv \text{constant}.
\end{equation*}
Using the definition of $z_i$ and $z_o$, we see that 
\begin{equation*}
 \left(\alpha-\beta_i-\beta_o\right) \frac{e^{z_i}}{e^{z_i}-1}\frac{e^{z_o}}{e^{z_o}-1} = \frac{\alpha-\beta_i-\beta_o}{z_i z_o}  \frac{\gamma^2}{(\alpha-\beta_{i})(\alpha-\beta_o)} = \frac{\gamma}{z_i z_o} \frac{\gamma(\alpha-\beta_{i}-\beta_{o})}{(\alpha-\beta_{i})(\alpha-\beta_{o})}.
 \end{equation*}
By the previous proposition, $\gamma/ (z_i z_o)$ is constant, so 
$\gamma(\alpha-\beta_i-\beta_o)/[(\alpha-\beta_i)(\alpha-\beta_o)]$ is constant as well.
\end{proof}
Thus we have proved that $I_1(\mathbf{w}(t))$ and $I_2(\mathbf{w}(t))$ are constant on the 
trajectory starting at $\mathbf{w}_0$.  As customary, we call these functions the integrals of the ODE \eqref{Hw}. The constancy of $I_j$ along the trajectory means that the gradient of $I_j$ is orthogonal to the ``velocity'', $H(\mathbf{w})$, along the trajectory, \eqref{eqns: characteristics},
i.e.\  
\begin{equation}\label{eqn: 10.16 constant along}
\nabla I_j(\mathbf{w}) \cdot \mathcal{E}_{\mathbf{s}}[\Delta \mathbf{s}]=0,
\end{equation}
where $\mathbf{s}/n=\mathbf{w}$.

Observe that $I_1(\mathbf{w})$, $I_2(\mathbf{w})$ are  $F_1(\mathbf{s})$,
$F_2(\mathbf{s})$ (defined in \eqref{definition of F_1}), with $\mathbf{w}=\mathbf{s}/n$. The perfect constancy of $I_j(\mathbf{w}(t))$ along the characteristics of the PDE predicts that $F_j(\mathbf{s}(t)/n)$ will be proven
to be, a.a.s.,\ almost constant along the random realization of the deletion process. That explains
why we used $F_j(\mathbf{s})$ in the definition of $\mathbf{S}_{\epsilon}$, the set of
good $\mathbf{s}$. 

Continuing with the characteristics, we will prove existence {\it at large\/} of a 
solution  $\mathbf{w}(t)$ of \eqref{Hw} for a ``good'' starting point $\mathbf{w}_0$. Analogously to good $\mathbf{s}$, i.e.\  $\mathbf{s}\in \mathbf{S}_{\epsilon}$, 
the (open) set of good $\bold w$ is defined by 
\begin{equation}\label{definition of W eps}
\mathbf{W}_{\epsilon}:= \Big\{ \mathbf{w}\geq \mathbf{0}: \mathbf{w}\text{ meets }\eqref{well};\,I_1(\mathbf{w}), I_2(\mathbf{w}) \in (c_n - \epsilon, c_n + \epsilon) \Big\};
\end{equation}
so if $\mathbf{s} \in \mathbf{S}_{\epsilon}$, then $\mathbf{s}/n \in \mathbf{W}_{\epsilon}$ .
\begin{fact}\label{propofWeps} There exists a fixed $\delta_1>0$ such that 
uniformly for $\mathbf{w}\in \mathbf{W}_{\epsilon}$
\begin{itemize}
\item {\bf (i)\/} $\alpha,\,\alpha-\beta_i,\,\alpha-\beta_o,\,\alpha-\beta_i-\beta_o\ge \delta_1$;
\item {\bf (ii)\/} $\gamma -\alpha \ge \delta_1$;
\item {\bf (iii)\/} $z_i,\,z_o\in [\delta_1,1/\delta_1]$.
\end{itemize}
\end{fact}
\noindent We omit the proof since it differs only superficially from the proof of  Fact \ref{remark: 2.24.1}.

Furthermore,
for $\mathbf{w}_0 \in  \mathbf{W}_{\epsilon}$, there exists a unique solution of
\eqref{Hw}, satisfying $\mathbf{w}(0)=\mathbf{w}_0$,  defined for $t\in [0,\mathcal{T})$
where $\mathcal{T}=\mathcal{T}(\mathbf{w_0})=\sup\{t: \mathbf{w}(t^\prime)\in\mathbf{W}
_{\epsilon}\text{ for }t^\prime\le t\}$.
Indeed, by Fact \ref{propofWeps},  on compact subsets of the interior of $\mathbf{W}_{\epsilon}$, the partial derivatives of $H(\mathbf{w})$ in \eqref{eqns: characteristics} are bounded;  it remains
to use the existence and uniqueness Theorem for systems of ODE's and a standard compactness argument.

Now, 
$\mathcal{T}=\infty$  means that the trajectory $\{\mathbf{w}(t)\}$ stays in $\mathbf{W}_{\epsilon}$ indefinitely. To determine $f_1$ and $f_2$, the solution of the PDE \eqref{the diffy q}, meeting the boundary conditions
$f_1(\alpha,0,0,\gamma)=\alpha$, $f_2(\alpha,0,0,\gamma)=\gamma$, we  need to show that the characteristic of the PDE, i.e.\ the solution of \eqref{Hw}, reaches $\partial(\mathbf{W}_{\epsilon})$,
the boundary of $\mathbf{W}_{\epsilon}$, in finite time $\mathcal{T}<\infty$ and that
$\beta_i(\mathcal{T})=\beta_o(\mathcal{T})=0$.

To motivate the proof that the trajectory ends on this desired range of the boundary, let's see what happens when we end there. Suppose that $\mathcal{T}<\infty$ and moreover that $\beta_i(\mathcal{T})=
\beta_o(\mathcal{T})=0$, and $\alpha(\mathcal{T}),\gamma(\mathcal{T})>0$.
Then 
\begin{equation*}
I_1(\mathcal{T})=\frac{\gamma(\mathcal{T})}{\alpha(\mathcal{T})-\beta_i(\mathcal{T})}=\frac{\gamma(\mathcal{T})}{\alpha(\mathcal{T})-\beta_o(\mathcal{T})}=\frac{\gamma(\mathcal{T})}{\alpha(\mathcal{T})}.
\end{equation*}
So in this case, $z_i(t)$ and $z_o(t)$ both end with the same values! In fact, $z_i(\mathcal{T})=z_o(\mathcal{T})=z(\gamma(\mathcal{T})/\alpha(\mathcal{T}))$ where, as before, $z(\eta)$ is the unique positive root of $\ell(z) = \eta$ (definition of $\ell(z)$ in \eqref{zi,zo=}).  Since $I_1(\mathcal{T})=
I_1(\mathbf{w}(0))$, we see then that
\begin{equation}\label{zi,oT=}
z_i(\mathcal{T})=z_o(\mathcal{T})=z(\eta(\mathbf{w}_0)),\quad \eta(\mathbf{w}):=\frac{\gamma(\alpha-\beta_i-\beta_o)}{(\alpha-\beta_i)(\alpha-\beta_o)}.
\end{equation}
In light of this preliminary discussion, we introduce $z^*=z^*(\mathbf{w})=z(\eta(\mathbf{w}))$.
We know that $z^*$ is constant along the trajectory, so $z^*$ depends only on $\mathbf{w}_0$.
Since
\begin{equation*}
\frac{\gamma(\alpha-\beta_i-\beta_o)}{(\alpha-\beta_i)(\alpha-\beta_o)} \leq \min \left \{ \frac{\gamma}{\alpha-\beta_i}, \frac{\gamma}{\alpha-\beta_o} \right \},
\end{equation*}
and $z(\eta)$ is increasing, we have that $z_i, z_o \geq z^*$ along the trajectory. Note that $z^*$ is bounded away from $0$ and $\infty$ because $\eta(\mathbf{w}_0)$ is bounded away from 1 and $\infty$ uniformly for $\mathbf{w}_0 \in \mathbf{W}_\epsilon$. 

\begin{proposition}\label{prop 544}
If $\mathbf{w}_0\in \mathbf{W}_{\epsilon}$, then $\beta_i(t)/\beta_o(t)$ is monotone,
varying toward $1$.
\end{proposition}

\begin{proof}  
First, from the second line of \eqref{eqns: characteristics}, 
\begin{equation*}
\frac{1}{\beta_i} \frac{d \beta_i}{dt} = \frac{\gamma(\alpha-\beta_i-\beta_o)e^{-z_i}}{(\alpha-\beta_i)(\alpha-\beta_0)} - \frac{1}{\beta_i+\beta_o} - \frac{\beta_o \gamma e^{-z_o}}{(\beta_i+\beta_o)(\alpha-\beta_i)(\alpha-\beta_o)}.
\end{equation*}
Using  $\ell(z_i)=\gamma/(\alpha-\beta_i)$ and $\ell(z_o)=\gamma/(\alpha-\beta_o)$, we find that
\begin{equation}\label{logderbetai}
\frac{1}{\beta_i} \frac{d \beta_i}{dt} = \frac{(\alpha-\beta_i-\beta_o)z_i}{(\beta_i+\beta_o)(\alpha-\beta_o)(e^{z_i}-1)} - \frac{1}{\beta_i+\beta_o} - \frac{\beta_o z_o}{(\beta_i+\beta_o)(\alpha-\beta_i)(e^{z_o}-1)},
\end{equation}
and similarly we obtain 
\begin{equation*}
\frac{1}{\beta_o} \frac{d \beta_o}{dt} = \frac{(\alpha-\beta_i-\beta_o)z_o}{(\beta_i+\beta_o)(\alpha-\beta_i)(e^{z_o}-1)} - \frac{1}{\beta_i+\beta_o} - \frac{\beta_i z_i}{(\beta_i+\beta_o)(\alpha-\beta_o)(e^{z_i}-1)}.
\end{equation*}
By subtracting these last two equations, we find that
\begin{align*}
\frac{1}{\beta_i}\frac{d \beta_i}{dt} - \frac{1}{\beta_o}\frac{d \beta_o}{dt} &= \frac{1}{(\beta_i+\beta_o)(\alpha-\beta_o)(e^{z_i}-1)} \big \{ (\alpha-\beta_i-\beta_o)z_i+\beta_i z_i \big \} \\&- \frac{1}{(\beta_i+\beta_o)(\alpha-\beta_i)(e^{z_o}-1)} \big \{ (\alpha-\beta_i-\beta_o)z_o+\beta_o z_o \big \}.
\end{align*}
Therefore
\begin{equation*}
\frac{d}{dt}\ln\left(\frac{\beta_i}{\beta_o}\right)=\frac{1}{\beta_i+\beta_o}\left(\frac{z_i}{e^{z_i}-1}
-\frac{z_o}{e^{z_o}-1}\right).
\end{equation*}
Since $\ell(z)$ is strictly increasing, and $\gamma/(\alpha-\beta)$ is increasing with $\beta$,
we see that $\beta_i< \beta_o$ iff $z_i<z_o$. Besides,  $x/(e^x-1)$ decreases for $x>0$.
Therefore if $\beta_i/\beta_o<1$, then $\tfrac{d}{dt} \ln \beta_i/\beta_o > 0$, and if
$\beta_i/\beta_o>1$, then $\tfrac{d}{dt} \ln \beta_i/\beta_o < 0$.
\end{proof}

\begin{proposition} 
If $\mathbf{w}_0 \in \mathbf{W}_{\epsilon}$, then the trajectory $\mathbf{w}(t)$ starting at $\mathbf{w}_0$ satisfies {\bf (i)} each of $\alpha, \beta_i, \beta_o, \gamma, z_i$ and $z_o$ is decreasing with $t$, and {\bf (ii)}  $\mathcal{T}<\infty$.
\end{proposition}
\begin{proof}
\textbf{(i)} That $\alpha, \gamma, z_i$ and $z_o$ decrease with $t$ is immediate from the
ODE \eqref{eqns: characteristics} as well as \eqref{eqn: 1 zi d zi} and \eqref{eqn: 1 zo d zo}.  Further, from \eqref{logderbetai} and $z_i\ge z^*$,
where $z^*$ is bounded away from zero, there exists a fixed $\delta>0$, such that
\begin{equation*}
\frac{1}{\beta_i}\frac{d \beta_i}{dt}
 \leq \frac{1}{\beta_i+\beta_o} \left(\frac{z^*}{e^{z^*}-1}-1\right)
\le -\frac{\delta}{\beta_i+\beta_o}.
\end{equation*}
Therefore $\beta_i(t)$ is decreasing, and so is $\beta_o(t)$, since likewise
\begin{equation*}
\frac{1}{\beta_i}\frac{d \beta_i}{dt}\le -\frac{\delta}{\beta_i+\beta_o}.
\end{equation*}
\textbf{(ii)} From part {\bf (i)\/} it follows that $d(\beta_i+\beta_i)/dt \le -\delta$. Therefore
$$
\mathcal{T}\le (\beta_i(0)+\beta_o(0))/\delta<\infty.
$$
\end{proof}
\noindent Thus $\mathcal{T}$ is finite and $\mathbf{w}(\mathcal{T}-):=\lim_{t \to \mathcal{T}^-} \mathbf{w}(t)$ exists. Defining $\mathbf{w}(\mathcal{T})=\mathbf{w}(\mathcal{T}-)$, we make $\mathbf{w}(t)$ continuous on the closed interval
$[0, \mathcal{T}]$. 
\begin{lemma}\label{lemma: end of trajectory}
Let $\mathbf{w}_0\in \mathbf{W}_{\epsilon}$. Then $\mathbf{w}(\mathcal{T})=(\alpha(\mathcal{T}), 0, 0, \gamma(\mathcal{T}))$, with $\alpha(\mathcal{T})>0$ and $\gamma(\mathcal{T})>0$. 
\end{lemma}
\begin{proof}
We know that each coordinate of $\mathbf{w}(t)$ is non-negative and decreasing.
So, by the definition \eqref{definition of W eps} of $\mathbf{W}_{\epsilon}$, the condition
$\mathbf{w}(\mathcal{T})\in \partial(\mathbf{W}_{\epsilon})$ means that at $t=\mathcal{T}$
\begin{itemize}
\item either at least one of $\alpha, \beta_i, \beta_o$ and $\gamma$ is zero,
\item or the condition \eqref{well} is violated.
\end{itemize}
Indeed, since $I_j(\mathbf{w}(t))\equiv \text{const}$ for $t<\mathcal{T}$, $I_j(\mathbf{w}(\mathcal{T}))$ is still in $(c_n-\epsilon, c_n+\epsilon)$, $j=1,2$.  Further,  by Fact \ref{propofWeps}, $\alpha-\beta_i-\beta_o$ is
bounded away from zero uniformly for  $\mathbf{w}\in \mathbf{W}_{\epsilon}$. Thus the case 
$\alpha-\beta_i-\beta_o=0$ at $t=\mathcal{T}$ is ruled out, and so $\alpha>\max\{\beta_i,\beta_o\}$.
Further,
as $z_i, z_o\in \bigl[z^*,\max\{z_i(0),z_o(0)\}\bigr]$,  we see that  $\ell(z_i)=\gamma/(\alpha-\beta_i)$ and $\ell(z_o)=\gamma/(\alpha-\beta_o)$ which 
are both at least $\ell(z^*) \in (c_n - \epsilon, c_n+\epsilon)$ must also be greater than 1. 
  Thus  the only reason that $\mathbf{w}(\mathcal{T})\notin\mathbf{W}_{\epsilon}$
is that at least one of $\beta_i(\mathcal{T})$, $\beta_o(\mathcal{T})$
is zero. According to Proposition \ref{prop 544},  $\beta_i(t)/\beta_o(t)$ varies monotonically in the direction towards $1$, so we conclude that both $\beta_i(\mathcal{T})$ and $\beta_o(\mathcal{T})$ are zero.
\end{proof}

\subsection{Determining the mean parameters}\label{sec: determining f1 and f2}

The study of the characteristics of the PDE  \eqref{the diffy q} in the previous section
allows us to give  explicit formulas  for the mean parameters $f_1$ and $f_2$. 
\begin{proposition}\label{prop: 5.10.1}
For $\mathbf{w}=(\alpha, \beta_i, \beta_o, \gamma)\in \mathbf{W}_{\epsilon}$, the functions
\begin{align*}
f_1(\alpha,\beta_i,\beta_o,\gamma) &= \frac{z\left(\frac{\gamma(\alpha-\beta_i-\beta_o)}{(\alpha-\beta_i)(\alpha-\beta_o)}\right)^2}{z\left(\frac{\gamma}{\alpha-\beta_i}\right)z\left(\frac{\gamma}{\alpha-\beta_o}\right)} \cdot \frac{(\alpha-\beta_i)(\alpha-\beta_o)}{\alpha-\beta_i-\beta_o},
\\ f_2(\alpha,\beta_i,\beta_o,\gamma) &=  \frac{z\left(\frac{\gamma(\alpha-\beta_i-\beta_o)}{(\alpha-\beta_i)(\alpha-\beta_o)}\right)^2}{z\left(\frac{\gamma}{\alpha-\beta_i}\right)z\left(\frac{\gamma}{\alpha-\beta_o}\right)} \cdot \gamma,
\end{align*}
solves the PDE \eqref{the diffy q}, where, we recall, $z(\eta)$ is the unique positive root of $\ell(z)=\eta$. Consequently, $f_1$ and 
$f_2$ are smooth on $\mathbf{W}_{\epsilon}$, meaning that,  on $\mathbf{W}_{\epsilon}$, the
functions $f_1$, $f_2$ have partial derivatives of every  order, and for each fixed $\ell$, the $\ell$-order
derivatives are uniformly bounded.
\end{proposition}
\begin{proof}
Let $\mathbf{w}(0) \in \mathbf{W}_{\epsilon}.$ By Lemma \ref{lemma: end of trajectory},  we have that $\mathbf{w}(\mathcal{T})=(\alpha(\mathcal{T}), 0, 0, \gamma(\mathcal{T}))$, where $\alpha(\mathcal{T}), \gamma(\mathcal{T})>0.$ Along the characteristic, the functions $f_j(\mathbf{w})$ each have a constant value since the PDE's \eqref{the diffy q} for $f_1$ and $f_2$ are homogeneous. By the boundary conditions for $f_1$ and $f_2$, we must have then  that 
\begin{equation*}
f_1( \mathbf{w}(0)) = f_1(\mathbf{w}(\mathcal{T}))  = \alpha(\mathcal{T}),\quad 
f_2( \mathbf{w}(0) ) = f_2( \mathbf{w}(\mathcal{T}))=
 \gamma(\mathcal{T}).
\end{equation*}
It remains to find $\alpha(\mathcal{T})$ and $\gamma(\mathcal{T})$ as functions
of $\mathbf{w}(0)$. 

Since $I_1(\mathbf{w}(t))=\gamma(\alpha-\beta_i-\beta_o)/[(\alpha-\beta_i)(\alpha-\beta_o)]
\equiv \text{const}$, we obtain
\begin{equation}\label{eqn: use once solve for a}
\frac{\gamma(\mathcal{T})}{\alpha(\mathcal{T})} = \frac{\gamma(\mathcal{T})(\alpha(\mathcal{T})-0-0)}{(\alpha(\mathcal{T})-0)(\alpha(\mathcal{T})-0)} = \frac{\gamma(0)(\alpha(0)-\beta_i(0)-\beta_o(0))}{(\alpha(0)-\beta_i(0))(\alpha(0)-\beta_o(0))}.
\end{equation}
Furthermore, by \eqref{zi,oT=},  $z_i(\mathcal{T})=z_o(\mathcal{T})=z^*$, where 
\begin{equation*}
z^*=z\bigl(I_1(\mathbf{w}(\mathcal{T}))\bigr)=z\bigl(I_1(\mathbf{w}(0))\bigr),
\end{equation*}
since $I_1(\mathbf{w}(t))\equiv I_1(\mathbf{w}(0))$.  The constancy of $I_2(\mathbf{w})=z_i z_o/\gamma$ along the trajectory implies that 
\begin{equation}\label{eqn: 2.25.1}
\gamma(\mathcal{T})=\gamma(0) \frac{z_i(\mathcal{T}) z_o(\mathcal{T})}{z_i(0) z_o(0)} = \gamma(0) \frac{(z^*)^2}{z_i(0)z_o(0)}.
\end{equation}
It follows from \eqref{eqn: use once solve for a} and  \eqref{eqn: 2.25.1} that  
\begin{equation*}
\alpha(\mathcal{T})= \frac{(z^*)^2}{z_i(0) z_o(0)} \frac{(\alpha(0)-\beta_i(0))(\alpha(0)-\beta_o(0))}{\alpha(0)-\beta_i(0)-\beta_o(0)}.
\end{equation*}

Since $z(\eta) \in C^{\infty}[1, \infty)$ and the arguments $\frac{\gamma(\alpha-\beta_i-\beta_o)}{(\alpha-\beta_i)(\alpha-\beta_o)}, \frac{\gamma}{\alpha-\beta_i}, \frac{\gamma}{\alpha-\beta_o}$ are bounded away from 1 for $(\alpha, \beta_i, \beta_o, \gamma)\in \mathbf{W}_\epsilon$, we have that $f_1$ and $f_2$ are smooth on $\mathbf{W}_\epsilon$. 
\end{proof}

As an illustration, and for a partial check, let us evaluate the values of $f_1$ and $f_2$ for the following initial state 
\begin{equation*}
\mathbf{w}_0 = (1-e^{-2c}, e^{-c}(1-e^{-c}), e^{-c}(1-e^{-c}), c).
\end{equation*}
We will later see that the likely initial states arising from $D(n,m=c_n n)$ are near $n \mathbf{w}_0$. We have that
\begin{equation*}
z_{i,o}(0)=
z\left(\frac{c}{1-e^{-c}}\right)=c,
\end{equation*}
and 
\begin{equation*}
z^*=z\left(c\,\frac{(1-e^{-2c}-e^{-c}(1-e^{-c})^2}{1-e^{-2c}-2e^{-c}(1-e^{-c})}\right)=z(c).
\end{equation*}
In particular, $z^*=z(c)$ is the unique root of $\ell(z)=c$, or equivalently of the equation
\begin{equation*}
1- \frac{z}{c} = e^{-z}=e^{-c (z/c)},
\end{equation*}
so that $z(c) = c \,\theta(c)$.  Therefore
\begin{align*}
f_1(\mathbf{w}_0)&=\frac{(z^*)^2}{z_i(0)z_o(0)}\cdot\frac{(1-e^{-2c}
-e^{-c}(1-e^{-c}))(1-e^{-2c}-e^{-c}(1-e^{-c})}{1-e^{-2c}-e^{-c}(1-e^{-c})-e^{-c}(1-e^{-c})}\\
&=\frac{z(c)^2}{c^2} \cdot 1=\theta(c)^2,
\end{align*}
in agreement with the results by Karp~\cite{karp} that the expected number of vertices in the strong giant component is approximately $\theta(c)^2 n$. Furthermore 
\begin{equation*}
f_2(\bar{\mathbf{w}})=\gamma(0) \frac{(z^*)^2}{z_i(0)z_o(0)}=c\,\theta^2, 
\end{equation*}
which means that the average in-degree and out-degree of the terminal digraph should be asymptotic to $c_n \sim c$, the original average in/out-degree of the initial digraph $D(n,m=c_nn)$.

\subsection{The PDE for the covariance parameters}\label{sec: pde for psi}

Turn to the covariance parameters $\{\psi_{j,k}\}$, a solution of PDE-type equation \eqref{eqn: second diff eq} meeting  boundary conditions \eqref{eqn: boundary conditions1}. The equation is 
\begin{equation*}
\hat{E}_{\mathbf{s}}\big[\Delta \mathbf{s}^T\big]\nabla \psi_{j,k}+\hat{E}_{\mathbf{s}}\big[\left(\Delta \mathbf{s}^T \nabla f_j\right)\left(\Delta \mathbf{s}^T  \nabla f_k\right)\big]=o(n^{-1/2}),
\end{equation*}
where $\Delta \mathbf{s}^T=(-a-b,r_i-a,r_o-b,-k).$ Just as for the PDE-type equation for $f_1, f_2$ \eqref{eqn: first diff eq}, it suffices to find $\psi_{j,k}$ that satisfy the corresponding PDE with $\mathcal{E}_{\mathbf{s}}
[\circ]$ instead of $\hat{E}_{\mathbf{s}}$, and the $o(n^{-1/2})$ term set equal $0$. Thus, we need to solve  
a scale-free equation
\begin{equation}\label{eqn: actual second diff eq}
\mathcal{E}_{\mathbf{s}}[\Delta \mathbf{s}^T] \nabla \psi_{j,k} + \mathcal{E}_{\mathbf{s}}[(\Delta \mathbf{s}^T \nabla f_j)(\Delta \mathbf{s}^T \nabla f_k)] = 0,
\end{equation}
subject to boundary conditions: for $\alpha, \gamma>0$, $\psi_{j,k}(\alpha, 0, 0, \gamma)=0$. 

Notice that the linear operator $\mathcal{E}_{\mathbf{s}}[\Delta \mathbf{s}^T] \nabla$ is exactly the same as in the PDE for $f_1$ and $f_2.$ In particular, we have the same characteristics. However,  unlike $f_j$, $\psi_{j,k}$ is not constant along trajectories. Instead,
\begin{equation*}
\frac{d \psi_{j,k}}{d t}(\mathbf{w}) = - \mathcal{E}_{\mathbf{s}}[(\Delta \mathbf{s}^T \nabla f_j)(\Delta \mathbf{s}^T \nabla f_k)],
\end{equation*}
with the non-zero RHS, expressed through the already known $f_1$ and $f_2$. 

For $\mathbf{w}(0)\in\mathbf{W}_{\epsilon},$ we found that $\mathbf{w}(\mathcal{T})=(\alpha(\mathcal{T}), 0, 0, \gamma(\mathcal{T}))$, where $\alpha(\mathcal{T}),\, \gamma(\mathcal{T})>0.$ So, by the boundary condition on $\psi$, we have that $\psi_{j,k}(\alpha(\mathcal{T}), 0, 0, \gamma(\mathcal{T}) ) = 0.$ Therefore
\begin{equation}\label{psi=integral}
\begin{aligned}
\psi_{j,k} (\mathbf{w}(0)) &= \int_0^{\mathcal{T}}  - \frac{d \psi_{j,k}}{d t} \, dt =\int_0^{\mathcal{T}} \Psi_{j,k}(\mathbf{w}(t))\, dt, \\&  \Psi_{j,k}(\mathbf{w}(t)):=\mathcal{E}_{\mathbf{s}}[(\Delta \mathbf{s}^T \nabla f_j)(\Delta \mathbf{s}^T \nabla f_k)]. 
\end{aligned}
\end{equation}
Observe that, for $\mathbf{u}\in\Bbb R^2$, and all $t\in [0,\mathcal{T}]$, 
\begin{equation*}
\mathbf{u}^T 
\left(
\begin{array}{cc}
\Psi_{1,1} & \Psi_{1,2} \\
\Psi_{2,1} & \Psi_{2,2} 
\end{array}
\right)
\mathbf{u}
= \mathcal{E}_{\mathbf{s}} \Big[ \Big( \sum_{j=1}^2 u_j \Delta \mathbf{s}^T \nabla f_j \Big)^2 \Big] \geq 0.
\end{equation*}
Consequently, without actually evaluating the integral in \eqref{psi=integral}, we already know
that
 \begin{equation*}
\mathbf{u}^T 
\left(
\begin{array}{cc}
\psi_{1,1}(\mathbf{w}(0)) & \psi_{1,2}(\mathbf{w}(0) \\
\psi_{2,1}(\mathbf{w}(0) & \psi_{2,2} (\mathbf{w}(0)
\end{array}
\right)
\mathbf{u} \ge 0.
\end{equation*}
Of course, this inequality should hold  since  $n\{\psi_{j,k}(\mathbf{w}(0))\}_{1\le j,k\le 2}$
is anticipated to be the limiting covariance matrix for the initial state $\mathbf{s}(0)=
n\mathbf{w}(0)$. A drawback of 
 \eqref{psi=integral} is that $\mathcal{T}$ is defined implicitly by $\beta_i(\mathcal{T})=
\beta_o(\mathcal{T})=0$. Using a change of variables, namely switching from $t$ say to $z_i$ and using our formula for $\frac{d z_i}{dt}$ (see \eqref{eqn: 1 zi d zi}),
we replace \eqref{psi=integral} by
\begin{equation}\label{eqn: 5.12.1}
\begin{aligned}
\psi_{j,k} (\mathbf{w}(0))& =\int_{z^*(0)}^{z_i(0)}\Psi_{j,k}(\mathbf{w})\,\frac{(\beta_i+\beta_o)(\alpha-\beta_o)}{\beta_i z_i}\,dz_i,\\
z^*(0)&= z\left(\frac{\gamma(0)(\alpha(0)-\beta_i(0)-\beta_o(0))}{(\alpha(0)-\beta_i(0))
(\alpha(0)-\beta_o(0))}\right).
\end{aligned}
\end{equation}
Indeed, here $\mathbf{w}=\mathbf{w}(z_i)=(\alpha,\beta_i,\beta_o,\gamma)$ is the solution of the ODE 
\eqref{Hw} where we switch from $t$ to $z_i$, i.e.\
\begin{equation*}
\frac{d\mathbf{w}}{d z_i}=\frac{(\beta_i+\beta_o)(\alpha-\beta_o)}{-\beta_i z_i}\, \mathbf{H}(\mathbf{w});\quad \mathbf{w}(z_i(0))=\mathbf{w}(0).
\end{equation*}
\\

Let's sketch the derivation of formulas for $\Psi_{j,k}(\mathbf{w})$. Since both $\Delta \mathbf{s}$ and $\nabla f_j$ are $4$-dimensional, there are $16$ summands in $\mathcal{E}_{\mathbf{s}}[(\Delta \mathbf{s}^T \nabla f_j)(\Delta \mathbf{s}^T \nabla f_k)]$ for each $(j,k)$.  Indeed, denoting $f_x=\partial f/\partial x$,
\begin{align}\label{eqn: 4.22.2} 
&(\Delta \mathbf{s}^T \nabla f_j)(\Delta \mathbf{s}^T \nabla f_k) = (a+b)^2 (f_j)_\alpha (f_k)_\alpha + (r_i-a)(-a-b) (f_j)_{\beta_{i}} (f_k)_\alpha \nonumber  \\&+ (r_o-b)(-a-b) (f_j)_{\beta_{o}} (f_k)_\alpha  + k(a+b) (f_j)_\gamma (f_k)_\alpha  +(-a-b)(r_i-a) (f_j)_\alpha (f_k)_{\beta_{i}} \nonumber \\&+ (r_i-a)^2 (f_j)_{\beta_{i}} (f_k)_{\beta_{i}} + (r_o-b)(r_i-a) (f_j)_{\beta_{o}} (f_k)_{\beta_{i}} - k(r_i-a) (f_j)_\gamma (f_k)_{\beta_{i}}  \\&+(-a-b)(r_o-b) (f_j)_\alpha (f_k)_{\beta_{o}} + (r_i-a)(r_o-b) (f_j)_{\beta_{i}} (f_k)_{\beta_{o}} + (r_o-b)^2 (f_j)_{\beta_{o}} (f_k)_{\beta_{o}} \nonumber \\&- k(r_o-b) (f_j)_\gamma (f_k)_{\beta_{o}}  +(-a-b)(-k) (f_j)_\alpha (f_k)_\gamma + (r_i-a)(-k) (f_j)_{\beta_{i}} (f_k)_\gamma \nonumber \\&+ (r_o-b)(-k) (f_j)_{\beta_{o}} (f_k)_\gamma - k(-k) (f_j)_\gamma (f_k)_\gamma. \nonumber
\end{align}
Clearly, this is a sum of quadratic polynomials of the random variables $a, b, r_i, r_o, k$ (i.e.\ $(a+b)^2, \ldots, k^2$) with deterministic coefficients being
pairwise products of partial derivatives of $f_1$ and $f_2$ taken at the current $\mathbf{w}=\mathbf{w}(t)$.  The $\mathcal{E}_{\mathbf{s}}$ expected values of these polynomials can be easily obtained from the 15 $\mathcal{E}_\mathbf{s}$ expected values (of $a^2, ab, \ldots, k^2$) given in Appendix \ref{app: app exp}. As for the gradient of $f_j$ (given in Proposition \ref{prop: 5.10.1}), they are computed by using implicit differentiation on $z(\circ)$. We refer the dedicated reader to a Mathematica notebook file, which contains all of these terms at http://www.dpoole.info/strong-giant/.

Since  $f_1$, $f_2$ are
smooth on $\mathbf{W}_{\epsilon}$, then so are the resulting $\Psi_{j,k}$. Therefore, the integrals in \eqref{eqn: 5.12.1} are well-defined, which proves existence of the sought-after solution of the PDE for $\{\psi_{j,k}\}$.  The  expressions for the integrands $\Psi_{j,k}(\mathbf{w})$ are exceedingly long and attempts to simplify them were unsuccessful. However, it follows from smoothness of $f_1$, $f_2$ and all $\Psi_{j,k}$ that  the functions $\psi_{j,k}(\mathbf{w})$ are smooth as well!\\

In summary, we have found smooth $f_j(\mathbf{w})$ and and proved existence of smooth $\psi_{j,k}(\mathbf{w})$, the
solutions of the partial differential equations  \eqref{the diffy q} and \eqref{eqn: actual second diff eq}.
 
Getting back to \eqref{eqn: first diff eq}--\eqref{eqn: second diff eq}, we have proved that, 
 uniformly over $\mathbf{s}\in\mathbf{S}_{\epsilon}$, 
\begin{equation*}
\hat{E}_{\mathbf{s}}[\Delta \mathbf{s}^T] \nabla f_j(\mathbf{s}/n) = O\left( (\ln n)^8/n \right),
\end{equation*}
and
\begin{equation*}
\hat{E}_{\mathbf{s}}[\Delta \mathbf{s}^T] \nabla \psi_{j,k} + \hat{E}_{\mathbf{s}}[\left( \Delta \mathbf{s}^T \nabla f_j\right)\left( \Delta \mathbf{s}^T \nabla f_k \right)] = O\left( (\ln n)^8/n \right).
\end{equation*}
Consequently, uniformly over $||\mathbf{u}||=O(n^{-1/2})$ and $\mathbf{s} \in \mathbf{S}_{\epsilon}
$,
\begin{equation}\label{eqn: 10.16 g recur}
G_n(\mathbf{s}/n, \mathbf{u}) -\sum_{\mathbf{s}'} G_n(\mathbf{s}'/n, \mathbf{u})\hat P(\mathbf{s}(t+1)=\mathbf{s}'|\mathbf{s}(t)=\mathbf{s}) = O( (\ln n)^8/n^{3/2} ),
\end{equation}
cf. \eqref{eqn: G approx recur}.\\

Thus we have found the Gaussian characteristic function $G_n$ very nearly satisfying
\eqref{eqn: 10.16 g recur} for $\mathbf{s}\in \mathbf{S}_{\epsilon}$, the equation satisfied 
perfectly by the actual characteristic function $\hat{\varphi}$ for all $\mathbf{s}$.
Our next step is to show that, for $\mathbf{s}(0)\in\mathbf{S}_{\epsilon^\prime}$, with $\epsilon^\prime$ close to $\epsilon$ from below, a.a.s.\  the process $\{\mathbf{s}(t)\}$  leaves $\mathbf{S}_{\epsilon}$ simply because  the number of semi-isolated vertices, $\nu_i+\nu_o$, drops
down to zero; formally, a.a.s.\  $\hat\tau=\bar\tau$. That is, establishing the asymptotic normality 
of $(\hat{\nu},\hat{\mu})$ is all we need for asymptotic normality of $(\bar\nu,\bar\mu)$.

\section{Likely reason for leaving $\mathbf{S}_{\epsilon}$}\label{sec: large deviation}

So we want to show that for $\mathbf{s}(0)\in \mathbf{S}_{\epsilon^\prime}$, with properly
chosen $\epsilon^\prime<\epsilon$, a.a.s.\ $\mathbf{s}(t)$ leaves $\mathbf{S}_{\epsilon}$
at the first moment $t$ when $\nu_i(t)+\nu_o(t)=0$. To do so, we will use constancy
of $I_j(\mathbf{w}(t))$ along the characteristic $\{\mathbf{w}(t)\}$ to show that their  
counterparts  $F_1(\mathbf{s})=\mu(\nu-\nu_i-\nu_o)/[(\nu-\nu_i)(\nu-\nu_o)]$ and $F_2
(\mathbf{s}) = z_i z_o/[\mu/n]$ a.a.s.\ are almost constant along the random process
$\{\mathbf{s}(t)\}$, for as long as $\nu_i(t)+\nu_o(t)>0$. For the proof,
we use the integrals $I_j(\mathbf{w})$ to construct a pair of exponential supermartingales,
and then apply the maximum inequality for supermartingales based on the Optional Sampling Theorem (Durrett~\cite{Durrett}). This approach
had been used in \cite{pittel spencer wormald} and in \cite{aronson} for the deletion processes for $k$-core and Karp-Sisper greedy matching problems, respectively.\\

For a given $L=L_n$, define 
\begin{equation*}
Q_j(\mathbf{s}):=\exp \Big( L \big( F_j(\mathbf{s})-F_j(\mathbf{s}(0)) \big) \Big).
\end{equation*}
Consider the process $Q^j(t):=Q_j(\mathbf{s}(t)) 1_{\{t< \hat{\tau}\}},$ for $t \geq 0$.

\begin{lemma}
If $L_n = o( (\ln n)(\ln \ln n))$, then the process
\begin{equation}\label{eqn: 3.12.1}
R^j(t) := \left( 1 + n^{-1} \right)^{-t} Q^j(t)
\end{equation}
is a non-negative supermartingale for $j=1,2$. 
\end{lemma}
\begin{proof}
There is little difference between the proofs for $j=1$ and $j=2$, so we consider  $j=1$ only. It suffices to show that
\begin{equation}\label{eqn: 10.18.1}
\hat{E}[Q^1(t+1)| \{\mathbf{s}(t')\}_{t' \leq t} ] \leq Q^1(t) (1+n^{-1}),
\end{equation}
for all $t$, since \eqref{eqn: 10.18.1} can be rewritten as
\begin{equation*}
\hat{E}[R^1(t+1) | \{\mathbf{s}(t')\}_{t'\leq t}]\leq R^1(t).
\end{equation*}
If $\mathbf{s}(t) \notin \mathbf{S}_{\epsilon}$, then  $Q^1(t+1)=Q^1(t)$ (see definition of $\hat{\tau}$ in \eqref{defhattau}), so that 
\begin{equation}\label{stnotinS}
\hat{E}[Q^1(t+1) | \{\mathbf{s}(t')\}_{t' \leq t} ] = Q^1(t),
\end{equation}
and obviously \eqref{eqn: 10.18.1} holds.
Suppose  $\mathbf{s}(t) \in \mathbf{S}_{\epsilon}.$ In this case, the (conditional) expectation 
with respect to $\hat{P}$ is the expectation with respect to $P$. Using the definition of $Q^1(\cdot)$, we obtain
\begin{equation*}
\begin{aligned}
E\bigl[Q^1(\mathbf{s}(t+1)) | \{\mathbf{s}(t')\}_{t' \leq t} \bigr] &=  Q^1(t) 
\sum_{\mathbf{s}' \in \mathbf{S}_{\epsilon}} e^{L(F_1(\mathbf{s}')-F_1(\mathbf{s}(t)))} P(\mathbf{s}(t+1)=\mathbf{s}' | \mathbf{s}(t))\\
&\le Q^1(t)(E_1+E_2);\\
E_1& := \sum_{\mathbf{s}': k \leq \ln n } e^{L(F_1(\mathbf{s}')-F_1(\mathbf{s}(t)))} P(\mathbf{s}(t+1)=\mathbf{s}' | \mathbf{s}(t)),\\
E_2&:=\sum_{\mathbf{s}' \in \mathbf{S}_{\epsilon} \atop k > \ln n} e^{L(F_1(\mathbf{s}')-F_1(\mathbf{s}(t)))} P(\mathbf{s}(t+1)=\mathbf{s}' | \mathbf{s}(t)).
\end{aligned}
\end{equation*}
Consider $E_1$.
Using Fact \ref{remark: 2.24.1}, uniformly over $\mathbf{s}(t)\in \mathbf{S}_{\epsilon}$ and $\mathbf{s}'$ such that $k \leq \ln n$, 
\begin{align*}
F_1(\mathbf{s}') = \frac{\mu' (\nu'-\nu'_i-\nu'_o)}{(\nu'-\nu'_i)(\nu'-\nu'_o)} &=  \frac{\mu(\nu-\nu_i-\nu_o)}{(\nu-\nu_i)(\nu-\nu_o)} +O\left(n^{-1} \ln n\right) \\ &= F_1(\mathbf{s}(t))+O\left(n^{-1}\ln n\right),
\end{align*}
and so
\begin{equation*}
L \left( F_1(\mathbf{s}')-F_1(\mathbf{s}(t)) \right) = O\left(\frac{(\ln n)^2 \ln \ln n}{n}\right)\to 0.
\end{equation*}
Therefore
\begin{equation*}
\exp \left( L \left( F_1(\mathbf{s}')-F_1(\mathbf{s}(t)) \right)\right)= 1+ L \left( F_1(\mathbf{s}')-F_1(\mathbf{s}(t)) \right) + O( (\ln n)^5/n^2).
\end{equation*}
Again using Fact \ref{remark: 2.24.1}, it is easy to check that $||\nabla (F_1(\mathbf{s(t)})) || = O(n^{-1})$ and that all $6$ second-order partial derivatives of $F_1(\mathbf{s})$ at  $\mathbf{s}
=(1-\lambda)\mathbf{s}(t) +\lambda \mathbf{s}^\prime$, $0\le \lambda\le 1$,  
are of order $O(n^{-2})$, uniformly for $\mathbf{s}(t)\in\mathbf{S}_{\epsilon}$, $\mathbf{s}^\prime$ and $\lambda$,  if $k\le \ln n$.   So, using
$\|\Delta\mathbf{s}\|=\|\mathbf{s^\prime}-\mathbf{s}(t)\|=O(\ln n)$,  we have 
\begin{equation*}
F_1(\mathbf{s}') - F_1(\mathbf{s}(t)) = \Delta \mathbf{s}^T( \nabla F_1(\mathbf{s}(t))) + O( (\ln n)^2/n^2)=O((\ln n)/n),
\end{equation*}
Consequently
\begin{align*}
E_1 &= \sum_{\mathbf{s}':k \leq \ln n} \left( 1 + L \Delta \mathbf{s}^T \nabla F_1(\mathbf{s}(t)) + O( (\ln n)^5/n^2 ) \right)P(\Delta \mathbf{s}| \mathbf{s}(t) ) \\ &= P(k \leq \ln n|\, \mathbf{s}(t)) + O( (\ln n)^5/n^2) + L \sum_{\mathbf{s}': k \leq \ln n} \Delta \mathbf{s}^T (\nabla F_1(\mathbf{s}(t))) P(\Delta \mathbf{s}|\, \mathbf{s}(t)).
\end{align*}
Now we found earlier, \eqref{eqn: 10.19.5},  that uniformly over $\mathbf{s}(t) \in \mathbf{S}_{\epsilon}$, 
\begin{equation*}
\sum_{\mathbf{s}': k \leq \ln n} \Delta \mathbf{s}^T P \left( \Delta \mathbf{s} |\mathbf{s}(t)\right)=\mathcal{E}_{\mathbf{s}(t)}[\Delta \mathbf{s}^T]+O( (\ln n)^8/n).
\end{equation*}
Also, by \eqref{eqn: 10.16 constant along}, since $I_j(\mathbf{w}(t))$ is constant along the trajectory $d \mathbf{w}/dt=\mathcal{E}_{\mathbf{s}(t)}[\Delta \mathbf{s}]$, the gradient $\nabla I_1(\mathbf{s}/n)$ is orthogonal to $\mathcal{E}_{\mathbf{s}(t)}[\Delta \mathbf{s}]$,  or equivalently $\nabla F_1 (\mathbf{s}) \cdot \mathcal{E}_{\mathbf{s}(t)}[\Delta \mathbf{s}]=0.$ Combining this with the 
estimate $L ||\nabla F_1(\mathbf{s}(t)||=O( (\ln n)^2/n)$ (uniformly over $\mathbf{s}(t) \in \mathbf{S}_{\epsilon}$), we have that
\begin{equation*}
L \sum_{\mathbf{s}': k \leq \ln n} \Delta \mathbf{s}^T \nabla F_1(\mathbf{s}(t)) P(\Delta \mathbf{s}| \mathbf{s}(t)) = O\left( L ||\nabla F_1|| \frac{(\ln n)^8}{n} \right) \ll \frac{(\ln n)^{10}}{n^2}.
\end{equation*}
Therefore
\begin{equation*}
E_1 = P(k \leq \ln n|\, \mathbf{s}(t)) + O( (\ln n)^{10}/n^2).
\end{equation*}
Here, by Lemma \ref{lemma: pi to qi}, {\bf (ii)\/}, we have that
\begin{equation}\label{eqn: 10.19.6}
P( k> \ln n| \mathbf{s}(t)) = \sum_{\mathbf{s}': k > \ln n} P \left( \mathbf{s}(t+1) = \mathbf{s}'| \mathbf{s}(t)\right) \leq e^{-\frac{2}{3} (\ln n)(\ln \ln n)},
\end{equation}
which implies that
\begin{equation}\label{eqn: 2.25.2}
E_1 = 1 + O((\ln n)^{10}/n^2).
\end{equation}

Consider $E_2$. Since $F_1(\mathbf{s}') \leq c_n + \epsilon$ for $\mathbf{s}' \in \mathbf{S}_{\epsilon}$, we have that
\begin{equation*}
L(F_1(\mathbf{s}')-F_1(\mathbf{s}(0))) \leq L (c_n+\epsilon) = o( (\ln n)(\ln \ln n)).
\end{equation*}
Using the bound \eqref{eqn: 10.19.6}, we have that
\begin{equation}\label{eqn: 2.25.3}
\sum_{\mathbf{s}' : k > \ln n} e^{L \left( F_1(\mathbf{s}')-F_1(\mathbf{s}(0))\right)} P \left( \mathbf{s}(t+1)=\mathbf{s}'| \mathbf{s}(t)\right) =O\left( e^{-\frac{2}{3} (\ln n)(\ln \ln n)} \right) \ll \frac{(\ln n)^{10}}{n^2}.
\end{equation}
Combining the bounds on $E_1$, \eqref{eqn: 2.25.2}, and $E_2$, \eqref{eqn: 2.25.3}, we obtain that  for $\mathbf{s}(t) \in \mathbf{S}_{\epsilon}$, 
\begin{equation}\label{eqn: 10.19 cir 2}
E\left[ 1_{\{t+1< \hat{\tau}\}} e^{L(F_1(\mathbf{s}(t+1)-F_1(\mathbf{s}(t))))} \right] \leq 1 + O( (\ln n)^{10}/n^2) < 1 + n^{-1}.
\end{equation}
Hence, by \eqref{stnotinS} and \eqref{eqn: 10.19 cir 2}, for any $\mathbf{s}(t)$,
\begin{equation*}
\hat{E}[Q^1(t+1) | \mathbf{s}(t)] \leq Q^1(t) \left( 1 + n^{-1} \right),
\end{equation*}
and $\{R^1(t)\},$ defined in \eqref{eqn: 3.12.1}, is a non-negative supermartingale.
\end{proof}

Now we are ready to prove that $F_j(\mathbf{s}(t))$ is almost constant all the way to (and especially including) the
time $\hat{\tau}$. To state the result and for future 
usage, we borrow a term from Knuth, Motwani and Pittel~\cite{knuth motwani pittel}:  we say that an event, $A$, holds {\it quite surely} ({\it q.s.} in short) if for each fixed $a>0$, 
\begin{equation*}
P(A) \geq 1 - n^{-a},
\end{equation*}
for $n$ sufficiently large.

\begin{lemma}\label{super bound}
Let $L_n =(\ln n)(\ln \ln n)/\ln \ln \ln n.$  Uniformly over $\mathbf{s}(0) \in \mathbf{S}_{\epsilon}$, q.s. 
\begin{equation}\label{eqn: 3.12.2}
\max_{t\le\hat{\tau}}|F_j(\mathbf{s}(t))-F_j(\mathbf{s}(0))| \leq \rho_n:=\frac{1}{\ln \ln \ln n}.
\end{equation}
\end{lemma}
\begin{proof}
Introduce yet another stopping time 
\begin{equation*}
\tau^\prime=\left\{\begin{aligned}
&\min\{t\le \hat\tau: \,F_j(\mathbf{s}(t))-F_j(\mathbf{s}(0)) > \rho_n\},\\
&\hat\tau+1,\quad \text{if no such }t\text{ exists}.\end{aligned}\right.
\end{equation*}
 Applying the Optional Sampling Theorem to the supermartingale $R^j(t)$ and the stopping time
$\tau^\prime$ (see Durrett~\cite{Durrett}), we have that
\begin{align*}
E[ Q^j(\tau') ] &= E[ \left(1+n^{-1}\right)^{\tau'} R^j(\tau')]  \leq \left(1+n^{-1}\right)^{n} E[R^j(0)] \\& = \left(1+n^{-1}\right)^n E[Q^j(0)] = \left(1+n^{-1}\right)^n \leq e.
\end{align*}
On the event $\{\tau^\prime \leq \hat{\tau}\}$, we have that $Q^j(\tau') \geq e^{L \rho_n}$ and so
\begin{equation*}
e^{L \rho_n} P(\tau^\prime  \leq \hat{\tau}) \leq E[Q^j(\tau')]\leq e.
\end{equation*}
Equivalently,
\begin{align*}
P\big(\max_{t \leq \hat{\tau}}[F_j(\mathbf{s}(t))-F_j(\mathbf{s}(0))] > \rho_n\big) &= P\big(\tau^\prime
 \leq \hat{\tau} \big) = O\left(e^{-L\rho_n}\right) \\& = O\left( e^{-(\ln n) \frac{\ln \ln n}{(\ln\ln \ln n)^2}}\right) \ll n^{-a}, \hspace{1 cm} \forall  \ (\text{fixed})\ a>0.
\end{align*}
The case $F_j(\mathbf{s}(t))-F_j(\mathbf{s}(0)) \leq -\rho_n$ is treated similarly.
\end{proof}

Finally,

\begin{lemma}\label{at the end}
Suppose $\epsilon'>0$ is such that $\epsilon' + \rho_n \leq \epsilon$, where $\rho_n$ is defined in \eqref{eqn: 3.12.2}. Uniformly over $\mathbf{s}(0) \in \mathbf{S}_{\epsilon'}$, q.s.
$F_j(\mathbf{s}(\hat{\tau})) \in (c_n-\epsilon,c_n+\epsilon)$, whence q.s. $\nu_i(\hat\tau)+\nu_o(\hat\tau)=0$.
Thus q.s.  $\mathbf{s}(t)$ leaves $\mathbf{S}_{\epsilon}$ because the number of semi-isolated vertices drops down to $0$, not because $F_j(\mathbf{s}(t))$ escapes $(c_n-\epsilon,c_n+\epsilon)$. In other words, q.s. $\hat{\tau}=\bar{\tau}$. 
\end{lemma} 
\begin{proof}
By the definition of $\mathbf{S}_{\epsilon'}$ \eqref{eqn: def of s eps}, we have 
$F_j(\mathbf{s}(0)) \in (c_n-\epsilon', c_n + \epsilon')$, and by Lemma \ref{super bound}, uniformly over $\mathbf{s}(0) \in \mathbf{S}_{\epsilon'}$, q.s. $|F_j(\mathbf{s}(\hat{\tau}))-F_j(\mathbf{s}(0))| \leq \rho_n$.  Therefore q. s. 
\begin{equation*}
F_j(\mathbf{s}(\hat{\tau})) \in (c_n - \epsilon'-\rho_n, c_n + \epsilon'+\rho_n) \subset (c_n - \epsilon, c_n + \epsilon).
\end{equation*}
However, if $F_{1,2}(\mathbf{s}(\hat\tau))\in (c_n - \epsilon, c_n + \epsilon)$,  then---
by the definition of $\mathbf{S}_{\epsilon}$--- the only remaining reason for the process to
leave $\mathbf{S}_{\epsilon}$ is that at this moment $\hat\tau$ the total number of semi-isolated
vertices has fallen  to zero. 
\end{proof}
Clearly, if  $(\hat\nu,\hat\mu)$ is shown to be asymptotically Gaussian, then so will be 
 $(\bar\nu,\bar\mu)$.

\section{Distribution of $(\hat{\nu}, \hat{\mu})$ starting from a generic $\mathbf{s}$}\label{sec: asymptotic distribution}

In this section, we prove  that indeed,  with $\mathbf{s}(0) \in \mathbf{S}_{\epsilon'}$, for a certain $\epsilon'$, the terminal pair $(\hat{\nu}, \hat{\mu})$, once centered and scaled, is asymptotically Gaussian. To do so we will show that the characteristic function, $\hat{\varphi}_{\mathbf{s}}(\mathbf{u}),$ of the pair $(\hat{\nu}, \hat{\mu})$ is closely approximated by the Gaussian characteristic function $G_n(\mathbf{s}(0)/n,\mathbf{u})$ (defined in \eqref{Gn=}). We begin with noticing that $\hat{\varphi}_{\mathbf{s}}(\mathbf{u})$ satisfies the one-step recurrence relation
\begin{equation*}
\hat{\varphi}_{\mathbf{s}}(\mathbf{u}) - \sum_{\mathbf{s}^{(1)}} \hat{\varphi}_{\mathbf{s}^{(1)}}(\mathbf{u}) \hat{P}(\mathbf{s}^{(1)}|\,\mathbf{s})=0.
\end{equation*}
Iterating this recurrence $n$ times, $n$ being the largest possible number of deletion steps, we find that
\begin{equation}\label{F bound}
\hat{\varphi}_{\mathbf{s}}(\mathbf{u}) - \sum_{\mathbf{s}^{(1)} \dots \mathbf{s}^{(n)}} \hat{\varphi}_{\mathbf{s}^{(n)}}(\mathbf{u}) \prod_{j=1}^n\hat{P}(\mathbf{s}^{(j)}|\, \mathbf{s}^{(j-1)})=0,
\quad \mathbf{s}^{(0)}=\mathbf{s}.
\end{equation}

Let us show that $G_n(\mathbf{s}/n,\mathbf{u})$ satisfies \eqref{F bound} with a $o(1)$ error term replacing $0$ on the right.
\begin{proposition}
Uniformly for  $||\mathbf{u}||=O(n^{-1/2})$, 
\begin{equation}\label{G bound}
\Big |G_n(\mathbf{s}/n, \mathbf{u}) - \sum_{\mathbf{s}^{(1)} \dots \mathbf{s}^{(n)}}G_n(\mathbf{s}^{(n)}/n, \mathbf{u})\hat{P}(\{\mathbf{s}(t)\}=\{\mathbf{s}^{(i)}\}) \Big | \leq n^{-0.49}.
\end{equation}
\end{proposition}
\begin{proof}
Define
\begin{equation*}
D_n(\mathbf{s}/n,\mathbf{u}):=G_n(\mathbf{s}/n, \mathbf{u}) -\sum_{\mathbf{s}^{(1)}} G_n(\mathbf{s}^{(1)}/n, \mathbf{u})P(\mathbf{s}(t+1)=\mathbf{s}^{(1)}|\,\mathbf{s}(t)=\mathbf{s}).
\end{equation*}
By \eqref{eqn: 10.16 g recur}, uniformly over $\mathbf{s} \in \mathbf{S}_{\epsilon}$ and $||\mathbf{u}|| = O(n^{-1/2}),$ 
\begin{equation*}
\bigl|D_n(\mathbf{s}/n,\mathbf{u})\bigr|= O( (\ln n)^8/n^{3/2} ).
\end{equation*}
Since $\hat{P}$ is the transition probability for the process that  freezes  at the moment $\mathbf{s}(t)$ leaves $\mathbf{S}_{\epsilon}$,  we have: for $\mathbf{s} \notin \mathbf{S}_{\epsilon}$, 
and all $t$,
\begin{equation*}
G_n(\mathbf{s}/n, \mathbf{u}) = \sum_{\mathbf{s}^{(1)}} G_n(\mathbf{s}^{(1)}/n, \mathbf{u}) \,1_{\{\mathbf{s}\}}(\mathbf{s}^{(1)}) = \sum_{\mathbf{s}^{(1)}} G_n(\mathbf{s}^{(1)}/n, \mathbf{u}) \hat{P}(\mathbf{s}(t+1)=\mathbf{s}^{(1)}|\, \mathbf{s}(t)=\mathbf{s}).
\end{equation*}
Thus, $D_n(\mathbf{s}/n,\mathbf{u})=0$ for $\mathbf{s} \notin \mathbf{S}_{\epsilon}$, whence
 uniformly for all $\mathbf{s}$ and $||\mathbf{u}||=O(n^{-1/2})$, 
\begin{equation}\label{|Dn|<}
\bigl|D_n(\mathbf{s}/n,\mathbf{u})\bigr|= O( (\ln n)^8/n^{3/2} )\le n^{-1.49},
\end{equation}
for $n$ large enough. It suffices to prove that for $1\leq k \leq n$,  we have that
\begin{equation}\label{eqn:induction}
\left |G_n(\mathbf{s}/n, \mathbf{u}) - \sum_{\mathbf{s}^{(1)}\dots\mathbf{s}^{(k)}} G_n(\mathbf{s}^{(k)}/n, \mathbf{u})\prod_{j=1}^k \hat{P}(\mathbf{s}^{(j)}|\,\mathbf{s}^{(j-1)})
 \right | \leq k \times n^{-1.49}.
 \end{equation}
Let us prove the above inequality by induction. The base case is precisely \eqref{|Dn|<}. Suppose that \eqref{eqn:induction} holds for $(k-1)$. Now, we have 
\begin{align*}
&\bigl| G_n(\mathbf{s}/n, \mathbf{u})  -\sum_{\mathbf{s}^{(1)}\dots\mathbf{s}^{(k)}} G_n(\mathbf{s}^{(k)}/n, \mathbf{u})\prod_{j=1}^k \hat{P}(\mathbf{s}^{(j)}|\,\mathbf{s}^{(j-1)})
 \bigr|\\
=& \bigl| G_n(\mathbf{s}/n, \mathbf{u})  -\sum_{\mathbf{s}^{(1)}\dots \mathbf{s}^{(k-1)}}
\prod_{j=1}^{k-1}\hat{P}(\mathbf{s}^{(j)}|\,\mathbf{s}^{(j-1)})\sum_{\mathbf{s}^{(k)}}
G_n(\mathbf{s}^{(k)}/n, \mathbf{u})\hat{P}(\mathbf{s}^{(k)}|\,\mathbf{s}^{(k-1)})\bigr|\\
=&\bigl| G_n(\mathbf{s}/n, \mathbf{u})  -\sum_{\mathbf{s}^{(1)} \dots \mathbf{s}^{(k-1)}}
\prod_{j=1}^{k-1}\hat{P}(\mathbf{s}^{(j)}|\,\mathbf{s}^{(j-1)})
\bigl[G_n(\mathbf{s}^{(k-1)},\mathbf{u})+D_n(\mathbf{s}^{(k-1)}/n,\mathbf{u})\bigr]\bigr|\\
\le&\bigl| G_n(\mathbf{s}/n, \mathbf{u})  -\sum_{\mathbf{s}^{(1)} \dots \mathbf{s}^{(k-1)}}
\prod_{j=1}^{k-1}\hat{P}(\mathbf{s}^{(j)}|\,\mathbf{s}^{(j-1)})G_n(\mathbf{s}^{(k-1)},\mathbf{u})\bigr|
+\max_{\mathbf{s}^\prime}|D_n(\mathbf{s}^\prime/n,\mathbf{u})|.
\end{align*}
The first term is at most $(k-1) \times n^{-1.49}$ by hypothesis and the second term is at most $n^{-1.49}$ by \eqref{|Dn|<}, which completes the proof of \eqref{eqn:induction}. 
\end{proof}

Next, we use this Lemma to bound the difference between  $\hat{\varphi}_{\mathbf{s}}(\mathbf{u})$ and $G_n(\mathbf{s}/n, \mathbf{u})$.  
\begin{theorem}\label{thm: fixed starting normal}
Let $0<\epsilon'<\epsilon$ with $\epsilon'=\epsilon'(n)$ possibly tending to zero and
$\epsilon$ being fixed.  For $||\mathbf{u}||=O\left(n^{-1/2}\right)$, uniformly over $\mathbf{s}(0) \in \mathbf{S}_{\epsilon'}$, 
\begin{equation}\label{G_n-hatphi}
| G_n(\mathbf{s}(0)/n, \mathbf{u})- \hat{\varphi}_{\mathbf{s}(0)}(\mathbf{u}) | \leq 2 n^{-.49}.
\end{equation}
\end{theorem}

\begin{proof} First,    
by \eqref{G bound} and \eqref{F bound},  uniformly over $||\mathbf{u}|| = O(n^{-1/2})$ and
{\it all\/} $\mathbf{s}$,
\begin{equation}\label{eqn: first triangle}
|G_n(\mathbf{s}/n, \mathbf{u}) - \hat{\varphi}_{\mathbf{s}}(\mathbf{u}) | \leq \sum_{\mathbf{s}^{(1)}, \ldots, \mathbf{s}^{(n)}} \big |G_n(\mathbf{s}^{(n)}/n, \mathbf{u}) - \hat{\varphi}_{\mathbf{s}^{(n)}}(\mathbf{u}) \big | \hat{P}\left( \{ \mathbf{s}(t) \} = \{ \mathbf{s}^{(i)} \} \right)  + n^{-.49}.
\end{equation}
Here $\mathbf{s}^{(n)}=\mathbf{s}(\hat\tau)$, of course. So, by Lemma \ref{at the end},
uniformly over $\mathbf{s}(0) \in \mathbf{S}_{\epsilon'}$, with probability at least $1-n^{-1}$,
we have $\mathbf{s}^{(n)}\in \mathbf{S}$, where
$$
\mathbf{S}=\{\mathbf{s}=(\nu,0,0,\mu):\,\nu,\mu>0,\,F_1(\mathbf{s}),F_2(\mathbf{s})
\in (c_n-\epsilon,c_n+\epsilon)\}.
$$
Accordingly, we break up the sum in \eqref{eqn: first triangle} into two parts, over  $\mathbf{s}^{(n)}\in \mathbf{S}$ and over $\mathbf{s}\notin \mathbf{S}$. Using the obvious 
$|G_n(\mathbf{s}^{(n)}/n, \mathbf{u})-\hat{\varphi}_{\mathbf{s}^{(n)}}(\mathbf{u})| \leq 2$
for the second sum, we obtain 
\begin{equation}\label{eqn: 10.15.2}
\begin{aligned}
|G_n(\mathbf{s}/n, \mathbf{u}) - \hat{\varphi}_{\mathbf{s}}(\mathbf{u}) | &\leq \sum_{\mathbf{s}^{(1)},\dots,
\mathbf{s}^{(n-1)}, \atop \mathbf{s}^{(n)} \in \mathbf{S}} \big |G_n(\mathbf{s}^{(n)}/n, \mathbf{u}) - \hat{\varphi}_{\mathbf{s}^{(n)}}(\mathbf{u}) \big | \hat{P}\left( \{ \mathbf{s}(t) \} = \{ \mathbf{s}^{(i)} \} \right)  \\&+ 2\sum_{\mathbf{s}^{(1)}, \dots,  \mathbf{s}^{(n-1)}, \atop\mathbf{s}^{(n)} \notin \mathbf{S}}  \hat{P}\left( \{\mathbf{s}(t) \} = \{\mathbf{s}^{(i)} \} \right) +n^{-0.49}.
\end{aligned}
\end{equation}
This first sum is zero. Here is why: by the definition of $\mathbf{S}$, if $\mathbf{s} \in \mathbf{S}$, then $\mathbf{s} \notin \mathbf{S}_{\epsilon}$ (defined in \eqref{eqn: def of s eps}), because $\nu_i=\nu_o=0$. So the process that starts at $\mathbf{s}
\notin  \mathbf{S}_{\epsilon}$ stays at $\mathbf{s}$. So, for any $\mathbf{s} \in \mathbf{S}$,  
\begin{equation*}
\hat{\varphi}_{\mathbf{s}}(\mathbf{u})= \hat{E}_{\mathbf{s}} \left[ \exp \left( iu_1 \hat{\nu}+i u_2 \hat{\mu}\right) \right] =  \exp \left( i u_1 \nu+i u_2 \mu \right),
\end{equation*}
and, by the boundary conditions on $f_j(\mathbf{w})$, $\psi_{j,k}(\mathbf{w})$,
\begin{equation*}
f_1(\mathbf{s}/n)=\nu/n,\quad f_2(\mathbf{s}/n)=\mu/n;\quad \psi_{j,k}(\mathbf{s}/n)=0,
\end{equation*}
whence
\begin{equation*}
G(\mathbf{s}/n, \mathbf{u})= \exp \left( i u_1 \nu + i u_2 \mu\right),
\end{equation*}
as well. The second sum in \eqref{eqn: 10.15.2} is precisely $P(\mathbf{s}(\hat\tau)\notin \mathbf{S})$, just proved to be
at most $n^{-1}$.  So \eqref{G_n-hatphi} follows. 
\end{proof}
\begin{corollary}\label{corollary: gaussian}
Let  $\epsilon'\to 0$ however slowly. For $\mathbf{s} \in \mathbf{S}_{\epsilon'}$,  
\begin{equation}\label{bar(pair)toGauss}
\left( \frac{\bar{\nu}- n f_1(\mathbf{s}/n)}{\sqrt{n}}, \frac{\bar{\mu} - n f_2(\mathbf{s}/n)}{\sqrt{n}}\right) \overset{\mathcal{D}}{\implies} \mathcal{N}\left( \textbf{0}, \boldsymbol{\psi} \right),
\end{equation}
where $\mathcal{N}\left( \textbf{0}, \boldsymbol{\psi}\right)$ is a Gaussian distribution with mean $\textbf{0}$ and covariance matrix $\boldsymbol{\psi}= (\psi_{j,k}(\mathbf{s}/n))$.
\end{corollary}

\begin{proof}
Fix $\mathbf t\in \Bbb R^2$ and set $\mathbf{u} = n^{-1/2}\mathbf{t}$.  By Theorem \ref{thm: fixed starting normal}, 
\begin{equation*}
E_{\mathbf{s}} \left[ e^{i \mathbf{t}^T\frac{(\hat{\nu}, \hat{\mu})^T}{\sqrt{n}}}\right] - \exp \left( i \sqrt{n} \mathbf{t}^T \mathbf{f}(\mathbf{s}/n)  - \frac{1}{2} \mathbf{t}^T \boldsymbol{\psi}(\mathbf{s}/n) \mathbf{t}  \right) = O(n^{-0.49}).
\end{equation*}
Equivalently, 
\begin{align*}
E_{\mathbf{s}}\left[ e^{i \mathbf{t}^T\frac{ (\hat{\nu}, \hat{\mu})^T-n \mathbf{f}(\mathbf{s}/n)}{\sqrt{n}}}\right] - \exp \left( - \frac{1}{2} \mathbf{t}^T \boldsymbol{\psi}(\mathbf{s}/n) \mathbf{t} \right) = O(n^{-0.49}).
\end{align*}
Now, since q.s.\ $\hat{\tau}=\bar{\tau}$, we may replace $(\hat{\nu},\hat{\mu})$ in the above expectation with $(\bar{\nu},\bar{\mu})$ along with an additive error of order at most $n^{-a}$ for any fixed $a>0$. Choosing $a>0.49$, we have that
\begin{align}\label{eqn: 10.14.1}
E_{\mathbf{s}}\left[ e^{i \mathbf{t}^T\frac{ (\bar{\nu}, \bar{\mu})^T-n \mathbf{f}(\mathbf{s}/n)}{\sqrt{n}}}\right] - \exp \left( - \frac{1}{2} \mathbf{t}^T \boldsymbol{\psi}(\mathbf{s}/n) \mathbf{t} \right) = O(n^{-0.49}).
\end{align}
By the Multivariate Continuity Theorem, (see Durrett~\cite{Durrett}), we have then \eqref{bar(pair)toGauss}.
\end{proof}

We are close now to to proving  a similar result for the deletion process applied to a random instance of $\mathbf{s}(0)$ induced by $D(n,m=c_n n)$. To achieve this, in the next section  we prove that $\mathbf{s}(0)$ is asymptotically Gaussian with a certain mean and a certain covariance matrix, both linear in $n$. This implies that the components of $\mathbf{s}(0)/n$, whence $f_j(\mathbf{s}(0)/n)$, $\psi(\mathbf{s}(0)/n)$, experience random fluctuations of magnitude $n^{-1/2}$. Combining this information with
Corollary \ref{corollary: gaussian}, we will see that $(\bar\nu,\bar\mu)$ for the random
$D(n,m=c_nn)$ is asymptotically Gaussian with mean computed for $\mathbf{s}^*:=
E[\mathbf{s}(0)/n]$ and the $2\times 2$ covariance matrix obtained from $n\psi_{j,k}( \mathbf{s}^*/n)$ by adding a certain non-negative $2\times 2$ matrix.

\section{Asymptotic Distribution of $\mathbf{s}(0)$ in $D(n,m=c_n n)$} \label{sec: asymp dist of s}

First of all, we write $\mathbf{s}(0)=(X, X_i, X_o, c_n n)$;
here $X, X_i, X_o$ are the number of non-isolated vertices, the number of vertices with zero in-degree and non-zero out-degree, and the number of vertices with zero out-degree and non-zero in-degree in $D(n,m=c_n n)$.

\begin{lemma}\label{Gaussian s(D(n,m))}
Let $c_n \to c \in (1, \infty)$. Then
\begin{equation*}
 \left( \frac{X-(1-e^{-2c_n})n}{\sqrt{n}}, \frac{X_i-(e^{-c_n}-e^{-2c_n})n}{\sqrt{n}}, \frac{X_o-(e^{-c_n}-e^{-2c_n})n}{\sqrt{n}} \right) \overset{\mathcal{D}}{\implies} \mathcal{N}(\mathbf{0}, \mathbf{K})
\end{equation*}
where $\mathcal{N}(\mathbf{0}, \mathbf{K})$ is a 3-dimensional Gaussian distribution with $\mathbf{0}$ and a covariance matrix $\mathbf{K}=\mathbf{K}(c)$.
\end{lemma}

\begin{comment}
The following proof actually works for $c \in (0, \infty)$, but to avoid conflicting with the convention thus far that $c_n \to c \in (1, \infty)$, we restrict the range. 
\end{comment}

\begin{proof} For simplicity of notations in the proof, we write $c$ instead of $c_n$.
It can be easily shown that
\begin{equation*}
E[X]=n(1-e^{-2c})+O(1) , \quad E[X_i]=E[X_o]=ne^{-c}(1-e^{-c})+O(1),
\end{equation*}
and that
\begin{equation*}
Var(X)=O(n), \quad Var(X_i)=Var(X_o)=O(n).
\end{equation*}
Introducing $\alpha'=1-e^{-2c}$ and $\beta'_i=\beta'_o=e^{-c}(1-e^{-c})$, we see
that 
\begin{equation*}
X^*:=\frac{X-n \alpha'}{n^{1/2}},\quad X_i^*:=\frac{X_i-n \beta_i'}{n^{1/2}},\quad 
X^*_o:=\frac{X_o-n \beta_o'}{n^{1/2}}
\end{equation*}
a.a.s.\ are small, i.e.\ $|X^*|,\,|X^*_i|,\,|X^*_o| \le \ln n$.

The joint distribution  of $X, X_i,X_0$ is given by 
\begin{equation}\label{exact formula}
P(X=\nu, X_i=\nu_i, X_0=\nu_0)=\,\frac{1}{{(n)_2 \choose m}} 
 \frac{n!}{\nu_i!\nu_o!(\nu-\nu_i-\nu_o)!(n-\nu)! } \sum_{\boldsymbol{\delta}, \boldsymbol{\Delta}}g(\boldsymbol{\delta}, \boldsymbol{\Delta}).
\end{equation}
Here $\boldsymbol{\delta}, \boldsymbol{\Delta}$ denote the generic values of in-degrees and out-degrees of a digraph with $m$ arcs and specified vertex subsets of the respective cardinalities $\nu_i,\nu_o,\nu$, and $g( \boldsymbol{\delta}, \boldsymbol{\Delta})$ is the total number of such digraphs. In particular,
\begin{align*}
&\delta_r>0, \Delta_r=0,\quad r\in \{1,\dots,\nu_i\};\\
&\Delta_s>0,\delta_s=0,\quad  s\in\{\nu_i+1,\dots,\nu_i+\nu_o\};\\
&\Delta_t>0,\delta_t>0,\quad t\in\{\nu_i+\nu_o+1,\dots,\nu\};\\
&\delta_u=\Delta_u=0,\quad u\in \{\nu+1,\dots,n\},\\
&\sum_r\delta_r+\sum_t\delta_t=\sum_s\Delta_s+\sum_t\Delta_t=m.
\end{align*}
We find a sharp asymptotic formula for the sum in \eqref{exact formula}  similarly to the asymptotic formula for $g(\mathbf{s})$, see the proof of Theorem \ref{thm g(s) asymptotics}. From Theorem \ref{thm: mckay formula}, we know that 
\begin{equation*}
g(\boldsymbol{\delta}, \boldsymbol{\Delta})=H( \boldsymbol{\delta}, \boldsymbol{\Delta})
m!\prod_r\frac{1}{\delta_r!}\ \prod_s\frac{1}{\Delta_s!}\ \prod_t\frac{1}{\delta_t!\Delta_t!},
\end{equation*}
where the ``fudge factor'' $H( \boldsymbol{\delta}, \boldsymbol{\Delta})\leq 1$. We need to obtain
a sharp asymptotic formula for the sum in \eqref{exact formula}, assuming that 
\begin{equation*}
x_i:=\frac{\nu_i-n\beta'_i}{n^{1/2}}=O(\ln n),\,\,x_o:=\frac{\nu_o-n\beta'_o}{n^{1/2}}
=O(\ln n),\,\,
 x:=\frac{\nu-n\alpha'}{n^{1/2}}=O(\ln n).
\end{equation*}
First, consider the sum without the factor $H( \boldsymbol{\delta}, \boldsymbol{\Delta})$. Then
\begin{align}
m!\sum_{ \boldsymbol{\delta}, \boldsymbol{\Delta}}\prod_r\frac{1}{\delta_r!} \ \prod_s\frac{1}{\Delta_s!} \ \prod_t\frac{1}{\delta_t!\Delta_t!}
&=m!\,[x^m y^m](e^{x}-1)^{\nu-\nu_i}(e^{y}-1)^{\nu-\nu_o}\label{eqn 5}\\
&\le m!\frac{(e^z-1)^{2\nu-\nu_i-\nu_o}}{z^{2m}},\label{Cher}
\end{align}
for any $z>0$. Introduce a truncated Poisson $Z$, where
\begin{equation*}
P(Z=j)=\frac{e^{-c} c^j/j!}{1-e^{-c}}=\frac{c^j/j!}{e^{c}-1},\quad j\ge 1,
\end{equation*}
so that
\begin{equation*}
E[z^Z]=\frac{e^{zc}-1}{e^{c}-1}.
\end{equation*}
In particular,
\begin{equation*}
E[Z]=\frac{c e^{c}}{e^{c}-1}=\frac{c}{1-e^{-c}}.
\end{equation*}
Then
\begin{align*}
[z^m](e^z-1)^a=&\,\frac{(e^c-1)^a}{c^m}\,[z^m]\,\left(\frac{e^{zc}-1}{e^c-1}\right)^a
=\,\frac{(e^c-1)^a}{c^m}\,[z^m]\,\left(E[z^Z]\right)^a\\
=&\,\frac{(e^c-1)^a}{c^m}\,P\!\!\left(\sum_{t\in [a]}Z_t=m\right),
\end{align*}
where $Z_1,\ldots,Z_a$ are independent copies of $Z$. We use this formula for
\begin{align*}
a_i=&\,\nu-\nu_i=n(1-e^{-c})+(x-x_i) n^{1/2},\\
a_o=&\,\nu-\nu_o=n(1-e^{-c})+(x-x_o)n^{1/2},
\end{align*}
so that
\begin{align*}
m-a_iE[Z]=&cn -\big[n(1-e^{-c})+(x-x_i) n^{1/2}\big]\frac{c}{1-e^{-c}}\\&=-\frac{c}{1-e^{-c}}(x-x_i)n^{1/2},\\
m-a_oE[Z]=&cn -\big[n(1-e^{-c})+(x-x_o) n^{1/2}\big]\frac{c}{1-e^{-c}}\\&=-\frac{c}{1-e^{-c}}(x-x_o)n^{1/2}.
\end{align*}
Let's use a shorthand $f\sim g$  if 
$$
f(x_i,x_o,x) =(1+o(1))g(x_i,x_o,x),\,\,\text{uniformly over } x_i, x_o, x = O(\ln n). 
$$
By the local central limit theorem (see Durrett~\cite{Durrett}),
\begin{equation}\label{local1}
\begin{aligned}
P\!\!\left(\sum_{t\in [a_i]}Z_t=m\right)=&\,P\!\!\left(\sum_{t\in [a_i]}(Z_t-E[Z])=m-a_iE[Z]\right)\\
\sim&\,\frac{1}{\sqrt{2\pi a_iVar(Z)}}\exp\left[-\frac{(m-a_iE[Z])^2}{2a_iVar(Z)}\right]\\
\sim&\,\frac{1}{\sqrt{2\pi n(1-e^{-c})Var(Z)}}\exp\left[-\frac{c^2(x-x_i)^2
}{2(1-e^{-c})^3Var(Z)}\right],
\end{aligned}
\end{equation}
and likewise
\begin{equation}\label{local2}
P\!\!\left(\sum_{t\in [a_o]}Z_t=m\right)\sim
\frac{1}{\sqrt{2\pi n(1-e^{-c})Var(Z)}}\exp\left[-\frac{c^2(x-x_o)^2}{2(1-e^{-c})^3Var(Z)}\right].
\end{equation}
By \eqref{local1}--\eqref{local2},  the equation \eqref{eqn 5} becomes
\begin{multline}\label{eqn 8}
m!\sum_{ \boldsymbol{\delta}, \boldsymbol{\Delta}}  \prod_r\frac{1}{\delta_r!} \prod_s\frac{1}{\Delta_s!} \prod_t\frac{1}{\delta_t!\Delta_t!} = \frac{m!}{c^{2m}} (e^c-1)^{2\nu-\nu_i-\nu_o} P\!\!\left(\sum_{t\in [a_i]}Z_t=m\right) P\!\!\left(\sum_{t\in [a_o]}Z_t=m\right) \\
\sim\frac{m!}{c^{2m}}\cdot\frac{(e^c-1)^{2\nu-\nu_i-\nu_o}}{2\pi n(1-e^{-c})Var(Z)}\exp\left[-\frac{c^2
\big((x-x_i)^2+(x-x_0)^2\big)}{2(1-e^{-c})^3Var(Z)}\right].
\end{multline}

Recall though that  for an asymptotic formula for the sum on the RHS in \eqref{exact formula} we
need to engage the fudge factor $H( \boldsymbol{\delta}, \boldsymbol{\Delta})$. To this end,
we notice that the dominant contribution to the RHS of \eqref{eqn 8} comes from $\boldsymbol{\delta}, \boldsymbol{\Delta}$ with 
$$
\|(\boldsymbol\delta,\boldsymbol\Delta)\|:=\max_v(\delta_v+\Delta_v)\leq 2\ln n. 
$$
Indeed if $d\geq 2\ln n$, then $d!\geq (d - [\ln n])! (\ln n)^{\ln n}$. So, similarly to \eqref{Cher} (with $z:=c$), we obtain
that the total contribution of $\boldsymbol{\delta}, \boldsymbol{\Delta}$ with 
$\|(\boldsymbol\delta,\boldsymbol\Delta)\| \ge 2\ln n$ to the sum in \eqref{exact formula} is at most  of order
\begin{equation}\label{cont,d>2logn}
n m! (\ln n)^{-\ln n}\frac{(e^c-1)^{2\nu-\nu_i-\nu_o}}{c^{2m}}.
\end{equation}
For $\|(\boldsymbol\delta,\boldsymbol\Delta)\|<2\ln n$, we have  McKay's  formula
\begin{multline*}
\ln H( \boldsymbol{\delta}, \boldsymbol{\Delta})=\,-\frac{1}{m}\sum_t \Delta_t\delta_t
-\frac{1}{2m^2}\left(\sum_r(\delta_r)_2+\sum_t(\delta_t)_2\right)\left(\sum_s (\Delta_s)_2+\sum_t(\Delta_t)_2\right)
\\+O(m^{-1}(\ln n)^4).
\end{multline*}
Defining $H(\boldsymbol{\delta}, \boldsymbol{\Delta})>0$ for  $\|(\boldsymbol\delta,\boldsymbol\Delta)\|\ge 2\ln n$,  we write
\begin{equation*}
\sum_{\boldsymbol{\delta}, \boldsymbol{\Delta}\atop \|(\boldsymbol\delta,\boldsymbol\Delta)\|
<2\ln n}
g(\boldsymbol{\delta}, \boldsymbol{\Delta})
=\frac{m!(e^c-1)^{a_i+a_o}}{c^{2m}}E\big[H(\mathbf{Z}^i,\mathbf{Z}^o) 1_{\{|\mathbf{Z}^i|=|\mathbf{Z}^o|=m)\}}
\big].
\end{equation*}
Here $\mathbf{Z}^i$ ($\mathbf{Z}^o$ resp.) consists of $a_i=\nu-\nu_i$ (resp. $a_o=\nu-\nu_0$) independent copies of
$Z$.
Now, with probability $1- O(n^{-b})$, for any $ b>0$,
\begin{equation*}
\sum_t\Delta_t\delta_t\sim (\nu-\nu_i-\nu_o)E^2[Z],
\end{equation*}
and
\begin{equation*}
\sum_r(Z^i_r)_2 +\sum_t(Z^i_t)_2\sim a_oE[(Z)_2],\quad
\sum_s(Z^o_s)_2+\sum_t(Z^o_t)_2\sim a_iE[(Z)_2].
\end{equation*}
Since 
\begin{equation*}
P(|\mathbf{Z}^i|=|\mathbf{Z}^o|=m)=\Theta(n^{-1})
\end{equation*}
is only polynomially small, we obtain that
\begin{multline*}
E\big[H(\mathbf{Z}^i,\mathbf{Z}^o)1_{\{|\mathbf{Z}^i|=|\mathbf{Z}^o|=m)\}}\big]\sim 
P(|\mathbf{Z}^i|=|\mathbf{Z}^o|=m)\\
\times\exp\left[-\frac{\nu-\nu_i-\nu_o}{m}E^2[Z]-\frac{a_ia_o}{2m^2}E^2[(Z)_2]\right].
\end{multline*}
Combining this estimate with \eqref{cont,d>2logn} we obtain
\begin{multline}\label{G del del}
\sum_{\boldsymbol{\delta}, \boldsymbol{\Delta}}g(\boldsymbol{\delta}, \boldsymbol{\Delta})
\sim \exp\left(-\frac{\nu-\nu_i-\nu_o}{m}E^2[Z]-\frac{(\nu-\nu_i)(\nu-\nu_o)}{2m^2}E^2[(Z)_2]\right)\\
\times\frac{m!}{c^{2m}}\cdot\frac{(e^c-1)^{2\nu-\nu_i-\nu_o}}{2\pi n(1-e^{-c})Var(Z)}\exp\left[-\frac{c^2
\big((x-x_i)^2+(x-x_0)^2\big)}{2(1-e^{-c})^3Var(Z)}\right].
\end{multline}
Note that
\begin{equation*}
\frac{(e^c-1)^{2\nu-\nu_i-\nu_o}}{e^{2m}}  = \frac{(e^c-1)^{(2 \alpha'- \beta'_i - \beta'_o)n}(e^c-1)^{(2x-x_i-x_o)\sqrt{n}}}{e^{2 c n}},
\end{equation*}
and
\begin{align*}
\left(e^c-1\right)^{2\alpha'-\beta'_i-\beta'_o} e^{-2c} &= \left( e^c-1 \right)^{2-2e^{-c}} e^{-2c} = e^{-2c e^{-c}} \left( 1 - e^{-c}\right)^{2-2e^{-c}}.
\end{align*}

Turn to the remaining factors in \eqref{exact formula}. Stirling formula yields
\begin{equation}\label{asymp bin}
\frac{m!}{{(n)_2 \choose m}}\sim 2\pi m\,e^{-c-c^2/2}\left(\frac{c}{e}\right)^{2m},
\end{equation}
and
\begin{align*}
\nu_i! &= \left(1+O\left(1/n\right)\right) \sqrt{2 \pi \nu_i} \left( \frac{\nu_i}{e} \right)^{\nu_i} \\&\sim \sqrt{2 \pi \beta'_i n}\frac{1}{e^{\nu_i}} \exp \left( (\beta'_i n + x_i \sqrt{n}) \left(\ln n + \ln \beta'_i + \ln \left( 1 + \frac{x_i}{\beta'_i \sqrt{n}}\right)\right)\right) \\& \sim C_i \sqrt{n} \frac{1}{e^{\nu_i}} \exp \Big( \beta'_i n \ln n+ \beta'_i n \ln \beta'_i + x_i \sqrt{n} \ln n + x_i \sqrt{n} \ln \beta'_i + x_i \sqrt{n} \\&+ x_i^2/\beta'_i - \frac{1}{2} x_i^2/\beta'_i\Big),
\end{align*}
for some constant $C_i>0$. There are similar asymptotic formulas for $\nu_o!, (\nu-\nu_i-\nu_o)!$ and $(n-\nu)!$. After simple algebra, we find that for some constant $C>0$,
\begin{align}\label{fract}
\nu_i!& \nu_o! (\nu-\nu_i-\nu_o)! (n-\nu)! \sim \frac{C n^2}{e^n} \nonumber \\ & \times \exp \Bigg\{ n \ln n + n \ln\left(e^{-2ce^{-c}} \left(1-e^{-c}\right)^{2-2e^{-c}}\right)  +\sqrt{n} (2x - x_i - x_o) \ln \left( e^c -1\right) \nonumber \\& +\frac{1}{2} \left(   \frac{x^2}{1-\alpha'} + \frac{(x-x_i-x_o)^2}{\alpha'-\beta'_i-\beta'_o} + \frac{x_i^2}{\beta'_i}+\frac{x_o^2}{\beta'_o} \right)\Bigg\}.
\end{align}

Using  \eqref{G del del}, \eqref{asymp bin} and \eqref{fract}, we find that \eqref{exact formula} becomes, after simplifying, 
\begin{equation}\label{local*}
P(X=\nu, X_i=\nu_i, X_o=\nu_o)= (1+o(1)) A (2 \pi n)^{-3/2} \exp \left( -\frac{1}{2}\mathbf{x}^T \mathbf{R} \, \mathbf{x}\right);
\end{equation}
here $\mathbf{x}=(\nu, \nu_i, \nu_o)^T$, $A=A(c)>0$  and the  {\it symmetric positive-definite} $3 \times 3$ matrix $\mathbf{R}$ is defined by 
\begin{equation*}
\begin{aligned}
\mathbf{x}^T \mathbf{R} \, \mathbf{x} &=     \frac{c^2\left(x-x_i\right)^2}{2(1-e^{-c})^3 Var[Z]} + \frac{c^2\left(x-x_o\right)^2}{2(1-e^{-c})^3 Var[Z]}   + \frac{x^2}{1-\alpha'} \\
&\quad+ \frac{(x-x_i-x_o)^2}{\alpha'-\beta'_i-\beta'_o }
+ \frac{x_i^2}{\beta'_i}+\frac{x_o^2}{\beta'_o}.
\end{aligned}
\end{equation*}
That $o(1)$ term is small uniformly over $(x,x_i,x_o)$ with $|x|+|x_i|+|x_o| \le B \ln n$, $B>0$ 
being fixed. So \eqref{local*} is actually the local limit theorem for $(X,X_i,X_o)$, and so
necessarily $A=\sqrt{\text{det} \mathbf{R}}$. Consequently 
\begin{equation*}
\left( \frac{X-\alpha' n}{\sqrt{n}}, \frac{X_i - \beta'_i n}{\sqrt{n}}, \frac{X_o-\beta'_o n}{\sqrt{n}}\right) \overset{\mathcal{D}}{\implies} \mathcal{N}(\bf 0, \bf K),
\end{equation*}
where $\mathbf{K} = \mathbf{R}^{-1}$. 
\end{proof}

Using Mathematica, we found
\begin{equation*}
\mathbf{K}(c)=\mathbf{R}^{-1} =\frac{e^c-1}{e^{4c}(2e^c-c-2)}
\left(
\begin{array}{c c c}
F & G & G \\
G & H & I \\
G & I & H \\
\end{array}
\right),
\end{equation*}
where 
\begin{align*}
F &= 2e^{2c}-ce^c-2-3c,\\
G &= (2+c)e^c-2-3c,\\
H &= 2e^{3c}+(-4-2c)e^{2c}+(4+3c)e^c-3c-2,\\
I &= -2e^c+2+3c.
\end{align*}

\section{Combining Distribution Results}\label{sec: combining dists}

\subsection{Asymptotic Distribution of the $(1,1)$-core in $D(n,m=c_n n)$}\label{sec: asymp distr of n,m}

In this Section, we combine the asymptotic distribution of $(\bar{\nu}, \bar{\mu})$ for
the process with a generic $\mathbf{s}(0)$ (Section \ref{sec: asymptotic distribution})  and
the asymptotic distribution of $\mathbf{s}(0)$ in $D(n, m=c_n n)$, just proved in
Section \ref{sec: asymp dist of s}, 
to show that, for the random 
$\mathbf{s}(0)$,  $(\bar{\nu}, \bar{\mu})$  is asymptotically  Gaussian as well (Theorem \ref{thm 3}). The argument is somewhat similar to the proof of Lemma 2 in Pittel~\cite{pittel 90}.

\begin{theorem}
Let $\bar{\nu}, \bar{\mu}$ be the number of vertices and arcs left at the end of the process  starting from $D(n,m=c_n n)$. Denote $\theta_n=\theta(c_n)$. Then we have that
\begin{equation*}
\left( \frac{\bar{\nu} - \theta_n^2 n}{\sqrt{n}}, \frac{\bar{\mu}-c_n \theta_n^2 n}{\sqrt{n}} \right) \overset{\mathcal{D}}{\implies} \mathcal{N}(\mathbf{0}, \mathbf{B}),
\end{equation*}
where $\mathcal{N}(\mathbf{0}, \mathbf{B})$ is a Gaussian distribution with mean $\mathbf{0}$ and a covariance matrix $\mathbf{B}=\mathbf{B}(c)$ which is continuous in $c$.
\end{theorem}

\begin{proof}
Let $\mathbf{t}\in \Bbb R^2$. We wish to show that 
\begin{equation*}
\bar{\varphi}(\mathbf{t}):=E\left[\exp \left( i t_1 \frac{\bar{\nu}-\theta_n^2n}{\sqrt{n}} + i t_2 \frac{ \bar{\mu}-c_n \theta_n^2n}{\sqrt{n}} \right) \right] = \exp \left( - \frac{1}{2} \mathbf{t}^T \mathbf{B} \ \mathbf{t} \right)+o(1),
\end{equation*}
where this expectation corresponds to the process starting at random $\mathbf{s}(0)$ from $D(n, m=c_n n)$. Since $D(n,m=c_n n)$ conditioned on $\{ \mathbf{s}(0)=\mathbf{s} \}$ is uniform among possible digraphs with parameter $\mathbf{s}$, we have 
\begin{equation*}
\bar{\varphi}(\mathbf{t})= \sum_{\mathbf{s}} E_{\mathbf{s}}\left[\exp \left( i t_1 \frac{\bar{\nu}-\theta_n^2 n}{\sqrt{n}}+i t_2 \frac{\bar{\mu}-c_n\theta_n^2n}{\sqrt{n}} \right) \right]P\big(\mathbf{s}(0)=\mathbf{s}\big).
\end{equation*}
Let's denote the generic expectation in the sum by $\bar{\varphi}_{\mathbf{s}}(\mathbf{t})$. 
 We break  the sum into two parts, for the likely $\mathbf{s}$ and for the unlikely $\mathbf{s}$. For
 the first part, we pick  $a>0$ and let $\mathbf{S}(a)$ be the set of $\mathbf{s}=(\nu,\nu_i,\nu_o, \mu=c_n n)$ such that 
\begin{equation}\label{<an1/2}
\sqrt{ (\nu-\alpha' n)^2+ (\nu_i-\beta'_i n)^2 + (\nu_o-\beta'_o n)^2 } \leq a \sqrt{n},
\end{equation}
where $\alpha'=1-e^{-2c_n}$ and $\beta'_i=\beta'_o=e^{-c_n}(1-e^{-c_n}).$ One can easily show that $\mathbf{S}(a) \subset \mathbf{S}_{\epsilon'}$ (defined in \eqref{eqn: def of s eps}), where $\epsilon'=n^{-1/3}.$ For the second part, we consider the remaining $\mathbf{s}$, i.e.\ $\mathbf{s} \notin \mathbf{S}(a)$. For these unlikely $\mathbf{s}$,
\begin{align}\label{eqn 591}
\Bigg | \sum_{\mathbf{s} \notin \mathbf{S}(a)} \bar{\varphi}_{\mathbf{s}}(\mathbf{t}) P\big(\mathbf{s}(0) = \mathbf{s}\big) \Bigg | \leq \sum_{\mathbf{s} \notin \mathbf{S}(a)} P\big(\mathbf{s}(0)=\mathbf{s}\big)  \to P\big( |\mathcal{N}(\mathbf{0}, \mathbf{K})| \geq a\big),
\end{align}
which is small if $a$ is large.

For $\mathbf{s}\in\mathbf{S}(a)$,  we can rewrite $\bar{\varphi}_{\mathbf{s}}(\mathbf{t})$ as 
\begin{equation}\label{eqn: 10.15.4}
\begin{aligned}
\bar{\varphi}_{\mathbf{s}}(\mathbf{t}) &= \exp \left( i \mathbf{t}^T \left( \mathbf{f}(\mathbf{s}/n) - (\theta_n^2, c_n \theta_n^2)^T \right) \sqrt{n}\, \right)\\
&\times E_{\mathbf{s}}[\exp \left( i \mathbf{t}^T \left( (\bar{\nu},\bar{\mu})^T-n \mathbf{f}(\mathbf{s}/n) \right) n^{-1/2}\right)]. 
\end{aligned}
\end{equation}
 In the proof of Corollary \ref{corollary: gaussian}, (see \eqref{eqn: 10.14.1}), we found that uniformly over $\mathbf{s} \in \mathbf{S}_{\epsilon'}$, 
\begin{equation}\label{eqn 58}
E_{\mathbf{s}}[\exp \left( i \mathbf{t}^T \left( (\bar{\nu},\bar{\mu})^T-n \mathbf{f}(\mathbf{s}/n)\right)n^{-1/2}\right)] = \exp \left( -\frac{1}{2} \mathbf{t}^T \boldsymbol{\psi} (\mathbf{s}/n) \ \mathbf{t}\right) + O(n^{-0.49}).
\end{equation}
By \eqref{<an1/2}, uniformly for $\mathbf{s} \in \mathbf{S}(a)$, we have that 
$
\|\mathbf{s}/n - (\alpha', \beta'_i, \beta'_o, c_n)^T\| = O(n^{-1/2}),
$
whence 
\begin{equation}\label{eqn: 10.15.5}
\boldsymbol{\psi} (\mathbf{s}/n) = \boldsymbol{\psi} (\alpha', \beta'_i, \beta'_o, c_n) + O(n^{-1/2}),
\end{equation}
because $\bold \psi$ has bounded partial derivatives near $(\alpha', \beta'_i, \beta'_o, c_n)$. Combining  \eqref{eqn 58} and \eqref{eqn: 10.15.5}, we rewrite \eqref{eqn: 10.15.4} as
\begin{equation}\label{eqn: 10.15.6}
\bar{\varphi}_{\mathbf{s}}(\mathbf{t}) = \exp \left( i \mathbf{t}^T \left( \mathbf{f}(\mathbf{s}/n)-(\theta_n^2, c_n \theta_n^2)^T \right) \sqrt{n}\, \right) \cdot \exp \left( - \frac{1}{2} \mathbf{t}^T \boldsymbol{\psi} \mathbf{t} \right) + O(n^{-.49}),
\end{equation}
where $\boldsymbol{\psi}$ is evaluated at $(\alpha', \beta'_i, \beta'_o, c_n)$. 

However the fluctuation of $f_j$ around $(\alpha', \beta'_i, \beta'_o, c_n)$ is on the order of $n^{-1/2}$ and we multiply $f_j$ by $\sqrt{n}$ in \eqref{eqn: 10.15.6}, so we must take this variation into account. If we denote the gradient at $f_j$ at the point $(\alpha', \beta'_i, \beta'_o, c_n)$ by $\nabla f_j$, then the Taylor expansion of $f_1$ around $(\alpha', \beta'_i, \beta'_o,c)$ gives:  uniformly over $\mathbf{s} \in \mathbf{S}(a)$,
\begin{equation*}
\sqrt{n} \left( f_1(\mathbf{s}/n) - \theta_n^2 \right)=\\ \nabla f_1 \cdot \left( \frac{\nu-\alpha' n}{\sqrt{n}},\frac{\nu_i-\beta'_i n}{\sqrt{n}},\frac{\nu_o-\beta'_o n}{\sqrt{n}}, 0\right) + O\left( \frac{1}{\sqrt{n}}\right),
\end{equation*}
and
\begin{equation*}
\sqrt{n} \left( f_2(\mathbf{s}/n) - c_n \theta_n^2 \right)=\\ \nabla f_2  \cdot \left( \frac{\nu-\alpha' n}{\sqrt{n}},\frac{\nu_i-\beta'_i n}{\sqrt{n}},\frac{\nu_o-\beta'_o n}{\sqrt{n}}, 0\right) + O\left( \frac{1}{\sqrt{n}}\right).
\end{equation*}
Hence, uniformly over $\mathbf{s} \in \mathbf{S}(a)$, we have
\begin{multline}\label{eqn 59}
\exp \left( i \mathbf{t}^T (\mathbf{f}(\mathbf{s}/n)-(\theta_n^2,c_n \theta_n^2)^T)\sqrt{n}\,\right) =\\ \exp \left( i u_1 \frac{\nu-\alpha' n}{\sqrt{n}} + i u_2 \frac{\nu_i-\beta'_i n}{\sqrt{n}}+ i u_3 \frac{\nu_o - \beta'_o n}{\sqrt{n}} \right) + O(n^{-0.49}),
\end{multline}
where $u_1 = t_1 (f_1)_{\alpha}+t_2 (f_2)_{\alpha}, u_2 = t_1 (f_1)_{\beta_i}+t_2 (f_2)_{\beta_i}, u_3 = t_1 (f_1)_{\beta_o}+t_2 (f_2)_{\beta_o}$ with these partial derivatives being evaluated at $\left(\alpha', \beta'_i, \beta'_o,c_n\right)$. 
Therefore,  by Lemma \ref{Gaussian s(D(n,m))}, in conjunction with \eqref{eqn 58} and \eqref{eqn 59},
\begin{multline}\label{eqn 595}
\lim_{n \to \infty} \sum_{\mathbf{s} \in \mathbf{S}(a)} \bar{\varphi}_{\mathbf{s}}(\mathbf{t}) P\big(\mathbf{s}(0)=\mathbf{s}\big) = \\  \exp \left( -\frac{1}{2} \mathbf{t}^T \boldsymbol{\psi} \mathbf{t} \right)  \times \int_{|\mathbf{x}|\leq a} \frac{1}{(2 \pi)^{3/2} \sqrt{\det \mathbf{K}}} \exp \left( i \mathbf{u} \cdot \mathbf{x} - \frac{1}{2} \mathbf{x}^T \mathbf{R} \mathbf{x} \right) d \mathbf{x}.
\end{multline}
Furthermore, we notice that
\begin{equation*}
\int \frac{1}{(2 \pi)^{3/2} \sqrt{\det \mathbf{K}}} \exp \left( i \mathbf{u} \cdot \mathbf{x} - \frac{1}{2} \mathbf{x}^T \mathbf{R} \mathbf{x} \right) d \mathbf{x} = \exp \left( -\frac{1}{2} \mathbf{u}^T \mathbf{K} \mathbf{u} \right).
\end{equation*}
Letting $a \to \infty$ in \eqref{eqn 591} and \eqref{eqn 595}, we obtain that
\begin{equation*}
\lim_{n\to\infty}\bar{\varphi}(\mathbf{t}) = \exp \left( -\frac{1}{2} \mathbf{t}^T \boldsymbol{\psi} \mathbf{t} - \frac{1}{2} \mathbf{u}^T \mathbf{K} \mathbf{u} \right).
\end{equation*}
Equivalently, 
\begin{equation*}
\lim_{n\to\infty}\bar{\varphi}(\mathbf{t}) = \exp \left( -\frac{1}{2} \mathbf{t}^T \mathbf{B}  \mathbf{t}\right),
\end{equation*}
where $\mathbf{B}(c)= (B_{j,k})$ is the $2\times 2$ matrix defined by
\begin{equation}\label{eqn: 5.10.1}
\mathbf{B} = \left(
\begin{array}{c c}
\psi_{1,1} & \psi_{1,2} \\
\psi_{2,1} & \psi_{2,2} \\
\end{array} 
\right)  + \left(
\begin{array}{c c c}
(f_1)_\alpha & (f_1)_{\beta_i} & (f_1)_{\beta_o} \\
(f_2)_\alpha & (f_2)_{\beta_i} & (f_2)_{\beta_o} \\
\end{array}
\right) \mathbf{K} \left(
\begin{array}{c c}
(f_1)_\alpha & (f_2)_{\alpha} \\
(f_1)_{\beta_i} & (f_2)_{\beta_i} \\
(f_1)_{\beta_o} & (f_2)_{\beta_o} \\
\end{array}
\right),
\end{equation}
and all functions involved are evaluated at 
\begin{equation}\label{eqn: 4.22.1}
\mathbf{w}':=(\alpha', \beta'_i, \beta'_o, c)=\left(1-e^{-2c},e^{-c}-e^{-2c},e^{-c}-e^{-2c},c\right).
\end{equation}
It is straightforward to compute the partial derivatives of $f_1$ and $f_2$ from their definition in Proposition \ref{prop: 5.10.1} and we omit the explicit formulas. Clearly $\mathbf{B}(c)$ is continuous since each of its components are. 
\end{proof}

\subsection{Description of the mean and covariance parameters}\label{sec: special case}

The limiting covariance matrix $\mathbf{B}(c),$ defined in \eqref{eqn: 5.10.1}, of the scaled random vector $(\bar{\nu}, \bar{\mu}),$ or in other words $(|V_{1,1}|, |A_{1,1}|),$ is written as a sum of $\boldsymbol{\psi},$ whose entries were given as integral expressions, and a matrix whose entries are explicit functions of $c$ (in the next section, we will prove that $(|V_1|, |A_1|)$ has the same asymptotic distribution).  We defined the entries $\psi_{j,k}$ using the method of characteristics and found that 
\begin{equation}\label{eqn: 5.10.2}
\psi_{j,k}(\mathbf{w}_0) = \int_{z^*(0)}^{z_i(0)} \mathcal{E}_{\mathbf{s}} [ (\Delta^T \nabla f_j)(\Delta^T \nabla f_k)] \frac{(\beta_i+\beta_o)(\alpha-\beta_o)}{\beta_i z_i} dz_i.
\end{equation}
For general $\mathbf{w}_0$, since each of $\alpha, \beta_i, \beta_o, \gamma, z_i, z_o$ are strictly decreasing along the trajectory, each variable can be written as a function of $z_i$. However, determining these exact expressions is difficult to do for general $\mathbf{w}_0$. Fortunately, for $\mathbf{w}'$, defined in \eqref{eqn: 4.22.1}, we can find  these expressions. 

Since the trajectory starting at $\mathbf{w}'$ begins with $\beta_i(0)=\beta_o(0)$, by Proposition \ref{prop 544}, we have that $\beta_i(t)=\beta_o(t)=:\beta(t)$ for $t \geq 0;$ whence $z_i(t)=z_o(t)=:z(t)$. Further, we found that the functions $I_1$ and $I_2$ are constants along this trajectory; in particular,
\begin{equation*}
I_1(\mathbf{w}') = c = \frac{\gamma(\alpha-2\beta)}{(\alpha-\beta)^2}, \quad I_2(\mathbf{w}') = c = \frac{z^2}{\gamma}.
\end{equation*}
These two equations along with the definition of $z$ (i.e.\ $\tfrac{z e^z}{e^z-1}=\frac{\gamma}{\alpha-\beta}$) allow us to solve for $\alpha, \beta,$ and $\gamma$ in terms of $z$. Thus the integrand in \eqref{eqn: 5.10.2} can be written as an explicit function of $z$ over the interval $[z^*(\mathbf{w}'), z_i(\mathbf{w}')]=[c \theta ,c]$.  However, even these simplified versions of the integrals are rather long and are omitted.  Nevertheless, one can approximate the entries of $\boldsymbol{\psi}$ by using numerical integration. We refer the interested reader to a Mathematica notebook file, which contains these integrands at http://www.dpoole.info/strong-giant/.

We can say much more about $\psi_{j,k}(\mathbf{w}')$ for $c=1+\epsilon,$ $\epsilon \downarrow 0$. Using Mathematica, we expand the integrand around $z=0$ and $\epsilon=0$. This allows us to integrate this power series for $z \in [c \theta, c]$ and determine the leading order terms. We find that $\psi_{j,k}=40\epsilon+O(\epsilon^2)$, for $j,k\in{1,2}$. So while the 
individual variances of  $|V_1|$ and $|A_1|$ are both of order $n\epsilon$, the variance
of the excess $\text{Exc}_1:=|A_1|-|V_1|$ is of order $n\epsilon^2$, at most. So, instead of $(|V_1|,
|A_1|)$, we consider the random vector $(|V_1|, \text{Exc}_1)$. This random vector is asymptotically Gaussian with mean $n (\theta^2, (c-1)\theta^2)$ and covariance matrix $n \tilde{\mathbf{B}}$, where
\begin{equation}\label{B'}
\tilde{\mathbf{B}} := \left( \begin{array}{cc} B_{1,1} & B_{1,2}-B_{1,1} \\ B_{1,2}-B_{1,1} & B_{1,1}+B_{2,2}-2B_{1,2} \end{array} \right).
\end{equation}
As $\epsilon \downarrow 0,$ the leading order terms for the mean vector are $n(4\epsilon^2, 4 \epsilon^3).$ To obtain the leading order terms for the covariance matrix, we needed to determine $\psi_{j,k}$ up to order $\epsilon^3$. In particular, we find that 
\begin{equation}\label{B'series}
\tilde{\mathbf{B}}=\left( \begin{array}{cc} 40 \epsilon + O(\epsilon^2) & 60 \epsilon^2 + O(\epsilon^3) \\ 60 \epsilon^2+O(\epsilon^3) & \frac{272}{3}\epsilon^3+O(\epsilon^4) \end{array} \right).
\end{equation}
Thus the variance of $\text{Exc}_1$ is actually of order $n\epsilon^3$, rather than naively
expected $n\epsilon^2$. The leading order terms for this covariance matrix $\tilde{\mathbf{B}}$
come from the contributions of $\boldsymbol{\psi}$ rather than those of the second term of \eqref{eqn: 5.10.1}. The latter ones are felt only in the next higher order terms. 

Qualitatively, the parameters of the largest strong component in $D(n, m=cn)$ are closer to 
the parameters of the giant $2$-core in $G(n, m=cn/2)$, than to those of  $G(n, m=cn/2)$'s
giant component. Indeed, from the results  in \cite{pittel wormald} it follows that, for $c=1+\epsilon,\, \epsilon \downarrow 0$, the scaled mean vector and covariance matrix of 
the (asymptotically Gaussian) pair $(\text{size of 2-core}, \text{excess of 2-core})$ for $G(n,m=cn/2)$ are of exactly the same orders as their counterparts  for $(|V_1|, \text{Exc}_1)$ in $D(n, m=cn)$.

\subsection{Asymptotic Distribution of the $(1,1)$-core in $D(n,p=c_n/n)$}\label{core in Dnp}

As essentially an afterthought, here is a similar claim for $D(n,p=c_n/n)$.
\begin{theorem}\label{coreinDnp}
Let $V_{1,1}^p ,A_{1,1}^p$ denote the vertex set and arc set of the $(1,1)$-core in  the random digraph $D(n,p)$. Suppose $p=c_n/n$ and $\lim c_n = c\in (1, \infty)$. Denote $\boldsymbol{\mu}(c)=(\theta^2(c), c \theta^2(c) )$. Note that $\mu(x)$ is differentiable and we denote its derivative as $\mu'(x)$. Then 
\begin{equation*}
\left( \frac{|V_{1,1}^p| - \theta_n^2 n}{n^{1/2}}, \frac{|A_{1,1}^p| - c_n \theta_n^2 n}{n^{1/2}} \right) \overset{d}{\implies} \mathcal{N}(\mathbf{0}, \mathcal{B}),
\end{equation*}
where $\mathcal{N}(\mathbf{0}, \mathcal{B})$ is the $2$-dimensional Gaussian vector
with mean $\mathbf{0}$ and the $2\times 2$ covariance matrix $\mathcal{B}$, given by
\begin{equation*}
\mathcal{B}(c)=\mathbf{B}(c)+ c (\boldsymbol{\mu}'(c))^T (\boldsymbol{\mu}'(c)).
\end{equation*}
\end{theorem}
The proof is basically a copy of the argument in \cite{pittel 90} (Lemma 2),
that allows transfer of the asymptotic normality results for the Erd\H os-R\'enyi graph $G(n,m)$
to the Bernoulli model $G(n,p)$, with $p=m/\binom{n}{2}$. For the sake of completeness, we include the proof in the Appendix (see Lemma \ref{distributional conversion from dnm to dnp}).

Consequently, the pair $(|V_1|, \text{Exc}_1)$ is asymptotically Gaussian with mean
$n(\theta^2),(c-1)\theta^2)$ and the covariance matrix $\mathcal{B}'(c)$ obtained from $\mathcal{B}(c)$ just like $\mathbf{B}'$ is obtained from $\mathbf{B}$, see \eqref{B'}. For $\epsilon=c-1
\downarrow 0$, $\mathcal{B}^\prime$ is given by the asymptotic expression on the RHS of
\eqref{B'series}.\\

Thus we have proved that for both $D(n,m=c_nn)$ and $D(n,p=c_n/n)$, the pair
$(|V_{1,1}|,|A_{1,1}|)$ is asymptotically Gaussian, with the covariance matrix dependent  on the
model. What's left is to show that in each model a.a.s.\ the parameters $(|V_{1,1}|,|A_{1,1}|)$ of the 
$(1,1)$-core are at a polylog distance from the parameters $(|V_1|, |A_1|)$ of the strong giant, the distance
negligible compared to the likely random fluctuations of the $(1,1)$-core's parameters. 
Our argument will utilize a depth-first-search process, rather than the deletion process
we have used for the $(1,1)$-core.

\section{Structure of $(1,1)$-core in $D(n,m)$ and $D(n,p)$}\label{sec: rest}

The main goal of this section is to prove that a.a.s.\ the $(1,1)$-core of $D(n, m=c_n n)$ and of $D(n, p=c_n/n)$, contains at most a polylog number of vertices and arcs outside the largest strong component (Theorem \ref{thm 2}). Breaking with the logic sequence so far, we first prove this theorem for $D(n, p)$ and then transfer this result to $D(n, m)$. 

Let's define $\mathcal{D}(V)$ ($\mathcal{A}(V)$ resp.) to be the descendant (ancestor resp.) of a set of vertices $V$. As a reminder, we define a vertex to be a descendant (ancestor) of itself. So by definition, for any vertex set $V$, we have that $V \subset \mathcal{D}(V) \cap \mathcal{A}(V);$ further if $V$ is the vertex set of a strong component, then $V=\mathcal{D}(V) \cap \mathcal{A}(V)$. 

To prove Theorem \ref{thm 2}, it is sufficient to prove the following lemma. Let us say that a event $A$ holds {\it asymptotically more surely} ({\it a.m.s.}) if $P(A) \geq 1-o(n^{-1/2}).$ We recall that an event $A$ holds q.s., if $P(A) \geq 1 - O(n^{-a})$ for any fixed $a>0$.

\begin{lemma}\label{lemma: 3.17.2}
A.m.s.,\  $D(n, p=c_n/n)$ is such that 
\begin{equation}\label{eqn: 3.19.1}
|V_{1,1} \setminus  \mathcal{D}(V_1)| \leq (\ln n)^8,\quad |V_{1,1} \setminus  \mathcal{A}(V_1)| \leq (\ln n)^8.
\end{equation}
\end{lemma}

Before proving this lemma, let us show how it implies Theorem \ref{thm 2}.

\begin{proof}[Proof of Theorem \ref{thm 2}]
Since $V_1 = \mathcal{D}(V_1) \cap \mathcal{A}(V_1)$, we have that
\begin{equation*}
V_{1,1} \setminus V_1 = \left( V_{1,1} \setminus \mathcal{D}(V_1) \right) \cup \left( V_{1,1} \setminus  \mathcal{A}(V_1)\right),
\end{equation*}
and so, a.m.s.,
\begin{equation}\label{V11-V1}
|V_{1,1}|-|V_1|=|V_{1,1} \setminus V_1| \leq 2 (\ln n)^8.
\end{equation}
Furthermore, one can easily show that q.s. max in-degree and max out-degree is at most $(\ln n)^2$; on this event, each vertex in $V_{1,1} \setminus V_1$ is an endpoint of at most $(\ln n)^2+(\ln n)^2$ arcs. Hence, a.m.s.,\
\begin{equation}\label{A11-A1}
0 \leq |A_{1,1}| - |A_1| \leq 2(\ln n)^8 \cdot 2 (\ln n)^2=4 (\ln n)^{10}.
\end{equation}
A standard inequality 
\begin{equation*}
P(D(n,m) \in A) = O\left[\sqrt{\text{Var} (e(D(n,p)))} \,P(D(n, p) \in A)\right],
\end{equation*}
for any event $A$, where $e(D)$ is the number of arcs of $D$ and $p=m/n^2$, 
and $\text{Var} (e(D(n,p=c_n/n)))=O(n^{1/2})$,  allow us to convert the a.m.s.\ bounds in 
\eqref{V11-V1}, \eqref{A11-A1} for $D(n, p=c_n/ n)$ into a.a.s.\ bounds in $D(n, m=c_n n)$.  
\end{proof}

Let us outline the proof of Lemma \ref{lemma: 3.17.2}. By symmetry, it suffices to prove only the first bound in \eqref{eqn: 3.19.1}. Following T. \L uczak and Seierstad~\cite{luczak seierstad}, we partition the vertex set of $D(n,p)$ into 2 sets, $V$ and $[n] \setminus V$, where $V$ is the set of all vertices that are descendants of ``long" cycles, i.e.\ those of length at least $(\ln n)^2$. All other cycles are called short. 

Vertices of $[n] \setminus V$ can not be reached from $V,$ because otherwise these vertices would be descendants of long cycles. Furthermore, the induced digraph on $[n] \setminus V$ can not contain a long cycle. 

Recall that a vertex is in the $(1,1)$-core if and only if it is in a cycle or in a directed path joining two cycles. Since there are no arcs from $V$ to $[n]\setminus V$, any vertex in $V_{1,1} \setminus V$ must be a descendant of a necessarily short cycle, with all its vertices in $[n] \setminus V$. So to prove that a.m.s.\ $|V_{1,1} \setminus V|$ is small, it suffices to show that a.m.s.\ the descendant set of cycles in the induced digraph on $[n] \setminus V$ is small. 

Let $\mathcal{L}$ be the induced digraph on $V$. The random digraph $D(n, p)$, conditioned on $\mathcal{L}$, is described as follows:
\begin{itemize}
\item there are no arcs from $V$ to $[n]\setminus V,$
\item each of the potential arcs from $[n] \setminus V$ to $V$ is present with probability $p$ independently of all other such arcs, as there is no new information on these arcs gained from knowing the set $V$,
\item the arc set within $[n] \setminus V$ is distributed as $D(n-|V|, p)$, conditioned on the event that there are no long cycles.
\end{itemize}

Here is our plan for action. First, we show that q.s. $|V|=\theta_n n + O(n^{1/2}\ln n).$ Second,
we use this result to prove that q.s. $D(n-|V|, p)$ does not have a long cycle; this allows us to drop conditioning on the event ``no long cycles" in the \L uczak-Seierstad decomposition. Third, we show that a.m.s.\ the descendant set of cycles within $D(n-|V|, p)$ has cardinality at most $(\ln n)^8.$ Finally,  we prove that a.m.s.\ $V=\mathcal{D}(V_1)$. 
\\

\noindent {\bf Step 1}.
We will use the following fundamental result due to Karp~\cite{karp}. Let $d(v)$ and $c(v)$
denote, respectively,  the number of descendants of vertex $v$ in $D(n,p)$ and the size  of the component containing $v$ in $G(n,p)$. Then $d(v)$ and $c(v)$ are equidistributed.
As for $\{c(v)\}_{v \in [n]}$, the following theorem was proved in~\cite{pittel 90}: 
\begin{theorem}[Gap Theorem in $G(n,p)$]\label{gap theorem in gnp1}
Suppose $c'_n \in (1, \infty)$ is bounded away from 1 and $\infty$. Then, q.s. exactly one component of $G(n, p=c'_n/n)$ has size in 
\begin{equation*}
[\theta'_n n - n^{1/2}\ln n, \theta'_n n + n^{1/2} \ln n],
\end{equation*}
such that $\theta'_n=\theta(c'_n),$ where $\theta(x)$ is the unique root of $1-\theta = e^{-x \theta}$, and each other component has size less than $(\ln n)^2$.
\end{theorem}

This theorem and Karp's identity allows us to obtain a similar claim for the sequence of descendant sets delivered by a ``full" directed version of the classical depth-first search (DFS) algorithm on $D(n,p)$. 

Here is how this algorithm works. Given a directed graph $D$ on $[n]$,  we start with vertex $w_1=1$ and run the directed depth-first search, finding the chronologically ordered sequence $w_2, \ldots, w_{k_1},$ of all descendants of $w_1$ together with the resulting search tree $T_1$. Moreover, we also find $W_1, W_2, \ldots, W_{k_1}$, their respective descendant sets within the search tree $T_1$. Next, pick $w_{k_1+1},$ the lowest-index vertex from $[n]\setminus V(T_1)$, and repeat the above step on $n \setminus V(T_1)$ obtaining $T_2$ and the descendant sets of vertices within $T_2$. And so on. We end up with a sequence of disjoint search trees $T_1, T_2, \ldots$, and the $n$-long sequence of vertices $w_1, \ldots, w_n$, together with their {\it partial} descendant sets $W_1, \ldots, W_n$. Note that the {\it full} descendant set of the tree $T_i$ in the {\it digraph} is contained within $T_1 \cup  \ldots \cup T_i$. Since every two vertices of a strong component are descendants of one another, the vertices of a strong component must reside in the same search tree. Applied to an undirected graph $G$ on $[n]$, the DFS delivers as many search
trees $T_1,T_2,\ldots$ as there are components of $G$, each $T_i$ spanning its own component.

Karp's observation can be extended to this full search. Namely, the two sequences of search trees $(T_1, T_2, \ldots)$, one for $D(n, p)$ and another for $G(n,p)$, are equidistributed. So,
letting $U_i=V(T_i)$ in $D(n,p)$, without any additional effort, we have
\begin{lemma}\label{lemma: 3.20.1}
For $D(n, p=c_n/n)$, q.s. exactly one of $U_1, U_2, \ldots,$ is giant, i.e.\ has size in 
\begin{equation}\label{eqn: 4.24.2}
[\theta_n n - n^{1/2}\ln n, \theta_n n + n^{1/2}\ln n],
\end{equation}
and all other $U_j$ have size less than $(\ln n)^2.$ 
\end{lemma} 

Since vertices in a strong component must reside in the same $U_i$, this gap theorem implies that q.s. any strong component with size at least $(\ln n)^2$ must be contained in the unique giant $U_i$, of  size in the interval \eqref{eqn: 4.24.2}. By the definition of long cycles, each long cycle must reside within this giant $U_i$. On the q.s. event ``there is an unique giant $U_i$", let $i_0$ denote its index. Since no descendants of long cycles can be found in $\cup_{i>i_0} U_i$, by the definition of $V$, we have that $V \subset \cup_{i \leq i_0} U_i,$ and so
\begin{equation*}
|V| \leq \sum_{i=1}^{i_0} |U_i| \leq \theta_n n + n^{1/2}\ln n + i_0 (\ln n)^2.
\end{equation*}

\begin{lemma}\label{lemma: 3.19.6}
Q.s. $i_0 \leq (\ln n)^2$. Consequently, q.s. $|V| \leq \theta_n n + 2n^{1/2}\ln n.$
\end{lemma}
\begin{proof}
First of all, for each $k \in [n]$, we have $\{1, 2, \ldots, k\} \subset U_1 \cup \ldots \cup U_k.$ By Karp's extended correspondence, it suffices to show that q.s. one of the first $(\ln n)^2$ vertices of $[n]$ is in the giant component of $G(n,p)$. 
To this end, let $\mathcal{V}$ denote the vertex set of largest component of $G(n,p).$ By symmetry, the set $\mathcal{V}$ conditioned on the value of $|\mathcal{V}|$ is uniformly distributed on all subsets of $[n]$ of size $|\mathcal{V}|$. Therefore, 
\begin{align*}
P(\mathcal{V} \cap \{1, \ldots, (\ln n)^2\}=\emptyset \big | |\mathcal{V}|=\ell) &= P(\mathcal{V} \subset \{(\ln n)^2+1, \ldots, n\} \big| |\mathcal{V}|=\ell) \\ &= {n-(\ln n)^2 \choose \ell} \Big / {n \choose \ell} \leq \exp \left( - \ell (\ln n)^2/n\right).
\end{align*}
By the gap theorem in $G(n,p)$, q.s., $|\mathcal{V}| = \theta_n n + O(n^{1/2} \ln n)$, and
for $\ell= \theta_n n + O(n^{1/2} \ln n)$,  the above conditional probability is $O(n^{-a})$ for any fixed $a>0$. 
\end{proof}

It remains to prove a matching lower bound for $|V|$. At this moment, we don't even know whether $V$, the descendant set of long cycles, is empty or not. We will prove likely existence of a special long cycle such that its descendant set has cardinality close to $\theta_n n.$ As a first step, we will prove likely existence of a special long path, which will later be shown to be contained a special long cycle; this is where our choice of using a {\it depth}-first search will play a vital role. In particular, we will prove that q.s. the giant descendant tree $T_{i_0}$ has the following ``broom-like'' structure, an example of which is shown in Figure \ref{figure: the broom}: (1) there is a long path in $T_{i_0}$, whose ``splinters" have at most polylog total size; (2) the descendant set of the endpoint of this path (the ``bristles") is giant; formal statement of this property is in Lemma \ref{lemma: 3.19.5}. Once this is proved, we will show that,  in $D(n,p)$, there is likely an arc from at least one of the bristles to the top end of the handle, resulting in a long cycle whose descendant set includes all the bristles.

\begin{figure}[!h]
\begin{center}
\begin{minipage}[t]{.8in}
\begin{center}$T_1$\\\vspace{1em}\includegraphics[scale=.5]{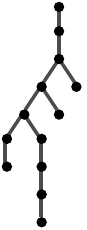}\end{center}
\end{minipage}
\begin{minipage}[t]{.8in}
\begin{center} $T_2$\\\vspace{1em}\includegraphics[scale=.5]{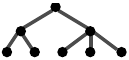}\end{center}
\end{minipage}
\begin{minipage}[t]{.8in}
\begin{center} $T_3$\\\vspace{1em}\hspace{-3em}\includegraphics[scale=.5]{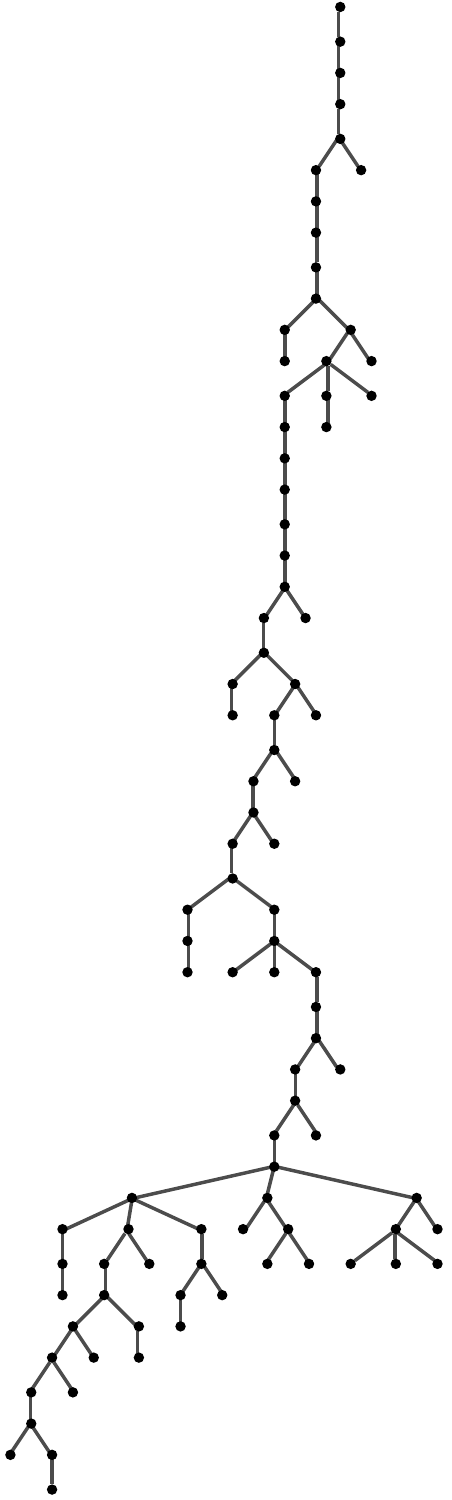}\end{center}
\end{minipage}
\begin{minipage}[t]{.8in}
\begin{center} $T_4$\\\vspace{1em}\includegraphics[scale=.5]{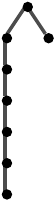}\end{center}
\end{minipage}
\begin{minipage}[t]{.8in}
\begin{center} $T_5$\\\vspace{1em}\includegraphics[scale=.5]{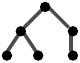}\end{center}
\end{minipage}
\begin{minipage}[t]{.3in}
\begin{center} $...$\\\vspace{1em}...\end{center}
\end{minipage}
\end{center}
  \caption{\label{figure: the broom} The broom within $T_3$ in the depth-first search.}
\end{figure}

First,
\begin{lemma}\label{lemma: 3.20.2}
Q.s., for each $i \leq n^{1/2},$ either $|W_i| <(\ln n)^2$ or 
\begin{equation}\label{eqn: 3.19.4}
|W_i| \in [\theta_n n -2n^{1/2}\ln n, \theta_n n + 2n^{1/2}\ln n].
\end{equation}
\end{lemma}
\begin{proof}
By Karp's correspondence, conditioned on the prehistory up to defining $w_j$,  $|W_j|$ is equal in distribution to $c(v)$, the size of the component of a generic vertex $v$ in $G(n-j+1, p)$. By the Gap Theorem, for each $j \leq n^{1/2}$, q.s. either $|W_j| < (\ln n-j+1)^2\leq (\ln n)^2$ or 
\begin{equation}\label{eqn: 3.19.3}
|W_j| \in [\theta^{(j)}_n (n-j+1) - n^{1/2}\ln n, \theta^{(j)}_n(n-j+1)+n^{1/2}\ln n],
\end{equation} 
where $\theta^{(j)}_n= \theta \left(c_n (1-\frac{j-1}{n})\right).$ For $j \leq n^{1/2}$, we have that $\theta^{(j)}_n = \theta(c_n)+O(n^{-1/2})$ (because the derivative of $\theta(x)$ is bounded near $c$), which implies that the interval in \eqref{eqn: 3.19.3} is contained in the interval in \eqref{eqn: 3.19.4}.
\end{proof}

Next,

\begin{lemma}(``broom property'')\label{lemma: 3.19.5} Let $W(v)$ denote the descendant
set of a vertex $v$ in the search forest $\{T_i\}$; so $W(w_j)=W_j$. 
Q.s. in the giant tree $T_{i_0}$ there is a unique path, $y_1 \to y_2 \to \ldots$ such that  

{\bf (i)} the path starts at the progenitor of $T_{i_0}$, i.e.\ $y_1 = \text{root}(T_{i_0})$; 

{\bf (ii)} the path has length $\lceil (\ln n)^3 \rceil$; 

{\bf (iii)} the descendant set of $y_{\lceil(\ln n)^3\rceil}$ is giant, i.e.
\begin{equation*}
|W(y_{\lceil(\ln n)^3\rceil})| \in [\theta_n n - 2n^{1/2}\ln n, \theta_n n + 2n^{1/2}\ln n].
\end{equation*}
\end{lemma}

\begin{proof}
Call a vertex $v$  {\it giant} ({\it small} resp.) if 
$|W(v)| \in [\theta_n n - 2n^{1/2}\ln n, \theta_n n + 2n^{1/2}\ln n]$,
(if  $|W(v)| <(\ln n)^2$, resp.). By the definition of $T_{i_0}$, $y_1$ is certainly giant. A giant vertex $v$ can have at most one giant child, because 
\begin{equation*}
|W(v)| = 1 + \sum_{u: \text{ children of }v} |W(u)|.
\end{equation*}
Recursively, if $y_j$ is giant, then let $y_{j+1}$ be the unique giant child if it exists. Let's show that q.s. $y_j$ exists for all $j \leq (\ln n)^3$, and in particular, $y_j = w_k$ for some $k=k(j) \leq j (\ln n)^4.$ 

Let's consider the event $A$ such that: $(a)$ $i_0 \leq (\ln n)^2$, $(b)$ for each $j \leq n^{1/2},$ $w_j$ is either small or giant, $(c)$ maximum out-degree is less than $(\ln n)^2$. By Lemmas \ref{lemma: 3.19.6}, \ref{lemma: 3.20.2}, as well as earlier comments about the event $(c)$, $A$ holds q.s. since every one of the events $(a), (b), (c)$ is a q.s. event. 

Suppose the event $A$ holds. Then, by $(a)$, $i_0$ is defined and $y_1 = w_k$ for  $k=|U_1|+\ldots |U_{i_0-1}|+1$. Furthermore, each of $U_1, \ldots, U_{i_0-1}$ have size less than $(\ln n)^2$, so 
\begin{equation*}
k \leq (i_0-1)\cdot (\ln n)^2+1 \leq (\ln n)^4.
\end{equation*}
Inductively, suppose that for any $j \in [1, (\ln n)^3]$ we have $y_j = w_k$ with $k \leq j(\ln n)^4.$ By $(c)$, $y_j$ has at most $(\ln n)^2$ children. Since $y_j$ is giant, at least one of $y_j$'s children is not small. Let $w_\ell$ be the chronologically first child of $y_j$ that is not small, i.e.\  $|W_{\ell}| \geq (\ln n)^2+1$. Since we are performing a {\it depth}-first search, for each earlier,
necessary small, child of $y_j$, we uncover their descendant sets within the search tree before coming to $w_{\ell}$. In particular, 
\begin{align*}
\ell =&\, k + 1+ \sum_{v\in \{y_j's\text{ children before }w_{\ell}\}} |W(v)|\\
 <&\, k + 1+ ((\ln n)^2-1)(\ln n)^2 \leq (j+1)(\ln n)^4\le (\ln n)^8\le n^{1/2}.
\end{align*}
Then, by $(b)$,  a non-small $w_{\ell}$ is either small or giant, hence $w_{\ell}$ must be giant. 
\end{proof}
Next,

\begin{lemma}\label{lemma: 4.24.1}
In $D(n,p=c_n/n),$ q.s. there is an arc from $W(y_{(\ln n)^3})$ to $\{y_1, y_2, \ldots, y_{(\ln n)^2}\}$, and $|V| \geq \theta_n n - 2 n^{1/2} \ln n$.
\end{lemma}

\begin{proof}
Conditioned on the search forest $\vec{T}:=\{T_i\}$, each possible backward arc, that is an arc
$(w_j, w_i)$ for $i<j$, is present in $D(n,p)$ independently with probability $p$. Let's condition on the event that $\vec{T}$ has the unique path described in Lemma \ref{lemma: 3.19.5}. By property {\bf (iii)} of the long path, $|W(y_{(\ln n)^3})| \geq \theta_n n - 2n^{1/2}\ln n$, and so
\begin{align*}
P(\text{no arcs from }W(y_{2(\ln n)^3})\text{ to }\{y_1, \ldots, y_{(\ln n)^2}\}| \vec{T}) &= (1-p)^{|W(y_{(\ln n)^3})| (\ln n)^2} \ll n^{-a},
\end{align*}
for any fixed $a>0$. If at least one of these potential arcs is present, then there is a cycle of length at least $(\ln n)^3-(\ln n)^2\gg(\ln n)^2$, passing through $y_{(\ln n)^3}$; whence, $V$ is nonempty and $W(y_{(\ln n)^3}) \subset V$.
\end{proof}

Combining Lemmas \ref{lemma: 3.19.6}, \ref{lemma: 4.24.1}, we establish the concentration 
property of $|V|$.

\begin{corollary}\label{typical A}
Quite surely, $D(n,p=c_n/n)$ is such that 
\begin{equation*}
|V| \in [ \theta_n n - 2 n^{1/2} \ln n, \theta_n n + 2 n^{1/2} \ln n],
\end{equation*}
where, we recall, $\theta_n$ is the unique positive root of $1-\theta=e^{-c_n \theta}.$
\end{corollary}

\noindent {\bf Step 2. } 

\begin{lemma}\label{lemma: 3.19.2}
Quite surely, there are no long cycles in $D(n-|V|, p = c_n/n)$. Consequently, the probability of any event in $D(n-|V|,p)$, conditioned on having no long cycles, is within distance $O(n^{-a})$ from the probability of this event in the unconditioned $D(n-|V|, p)$. 
\end{lemma}

\begin{proof}
Note that 
\begin{equation*}
P(\exists \text{ long cycle in }D(n-|V|, p)) = \sum_{\nu=0}^{n} P(\exists \text{ long cycle in }D(n-\nu,p))P(|V|=\nu).
\end{equation*}
By Corollary \ref{typical A}, q.s. $|V| \in [\theta_n n -2 n^{1/2}\ln n, \theta_n n + 2n^{1/2}\ln n]$. Further, the probability that $D(n-\nu,p)$ contains a long cycle is decreasing in $\nu$, and
for $\nu=\nu' :=\theta_n n-2n^{1/2}\ln n$ we have
\begin{align*}
P(\exists \text{ long cycle in }D(n-\nu', p)) &\leq \sum_{\ell=(\ln n)^2}^{n-\nu'} {n-\nu' \choose \ell}(\ell-1)! p^l \leq \sum_{\ell=(\ln n)^2}^{\infty} \frac{1}{\ell} \left( (n-\nu')p \right)^\ell.
\end{align*}
Now $(n-\nu')p = (1-\theta(c))c + o(1)$ since $c_n \to c$ and $\theta(x)$ is continous (see definition of $\theta(x)$ in Theorem \ref{gap theorem in gnp1}). It is easy to show that $(1-\theta(c))c <1$; hence there is some $\delta \in (0,1)$ such that $(n-\nu')p <  \delta$ for all sufficiently large $n$. Hence
\begin{equation*}
P(\exists \text{ long cycle in }D(n-\nu', p)) \leq \sum_{\ell=(\ln n)^2}^{\infty} \frac{\delta^{\ell}}{\ell} \ll n^{-a}, \quad \forall \, \text{(fixed)} \, a>0.  
\end{equation*}
\end{proof}

\noindent {\bf Step 3.} So we need to bound the size of the descendant set of cycles within $[n] \setminus V$. Let $C_n$ be the number of vertices in short cycles in $D(n-|V|, p).$ 

\begin{lemma}\label{lemma: 3.20.3}
A.m.s.\ in $D(n-|V|, p=c_n/n)$, cycles are vertex-disjoint  and $C_n \leq (\ln n)^6$. 
\end{lemma}

\begin{proof}
As in the proof of Lemma \ref{lemma: 3.19.2}, it suffices to consider this event in $D(n', p'=\delta/n')$, where $n'=n-\nu',$ $\nu'= \theta_n n - 2 n^{1/2}\ln n$ and $n' p \leq \delta<1.$  

Using the first moment argument, it is straightforward to show that in $D(n', p')$, a.m.s.\ the number of vertices in {\it disjoint} short cycles is at most $(\ln n)^6$. Therefore, all that remains is to show that a.m.s.\ any two cycles are vertex-disjoint. As noticed by \L uczak and Seierstad \cite{luczak seierstad}, if two cycles intersect, then there is a directed cycle with a directed chord of possibly zero length. On $k$ vertices, there are at most $k! k^2$ pairs of such cycles and chords. That is, if $k_1$ denotes the length of the cycle and $k_2:=k-k_1$ denotes the number of vertices in the chord, then the number of such pairs is at most ${k \choose k_1}(k_1-1)! \, k_1^2 \, k_2!.$ Summing over $k_1$, we find that
\begin{equation*}
\sum_{k_1=3}^k {k \choose k_1} (k_1-1)! \ k_1^2 \ k_2! \leq \sum_{k_1=3}^k k!\  k_1 \leq k!\ k^2.
\end{equation*}
Therefore, the probability that two cycles intersect in $D(n', p')$ is at most
\begin{equation*}
\sum_{k=3}^{n'} {n' \choose k} k! \, k^2 \, p^{k+1} \leq p \sum_{k=3}^{n'} k^2 \, \delta^k = O(1/n).
\end{equation*}
\end{proof}

Having bound the likely number of vertices in cycles in $D(n-|V|, p)$, we turn to  the descendant sets of those vertices and complete Step 3.

\begin{lemma}\label{ancestor 1}
In $D(n-|V|, p)$, q.s. for any set of vertices $W$ we have $d(W) \leq |W| (\ln n)^2$.
Consequently, by applying Lemmas \ref{lemma: 3.19.2}, \ref{lemma: 3.20.3}, in $D(n-|V|, p=c_n/n)$, a.m.s.\ the size of the descendant set of cycles is at most $(\ln n)^8$. Therefore, in $D(n, p=c_n/n)$, a.m.s.\
\begin{equation*}
|V_{1,1} \setminus V| \leq (\ln n)^8.
\end{equation*}
\end{lemma}

\begin{proof}
Again it suffices to consider only $\nu':=\theta_n n - 2n^{1/2}\ln n.$ Using Karp's correspondence, and the fact that increasing the edge probability can only increase, stochastically,  $c(v)$, 
\begin{equation*}
P(|d(v)| \geq (\ln n)^2) \leq P(c(v)\geq (\ln n)^2 \text{ in }G(n', p'=\delta/n')),
\end{equation*}
where $n'=n-\nu'$ and $n' p \leq \delta <1$. Using a first moment argument on the counts of trees of given size, one can show that this latter probability is at most $e^{-b (\ln n')^2}( \ll n^{-a})$ for some $b>0.$ The union bound over all vertices completes the proof.
\end{proof} 

\noindent {\bf Step 4.} Recall that in this Step, we wish to prove that a.m.s.,\ $\mathcal{D}(V_1)=V$. Once proven, it delivers, in combination with  Lemma \ref{ancestor 1},  the proof 
of Lemma \ref{lemma: 3.17.2}.

\begin{lemma}
In $D(n,p=c_n/n)$, a.m.s.\ $\mathcal{D}(V_1)=V$.
\end{lemma}

\begin{proof}
Let $V'$ be the ancestor set of long cycles. The results in Steps (1)-(3) have their obvious analogues to $V'$. Since there are no arcs from $V$ to $[n] \setminus V$ or from $[n] \setminus V'$ to $V'$, any strong component outside $V \cap V'$ must completely reside in either $[n]\setminus V$ or $[n] \setminus V'$. Now by Lemma \ref{lemma: 3.20.3}, a.m.s.\ each strong component within $[n] \setminus V$ (or $[n] \setminus V'$) is either a single vertex or a short induced cycle, which by the definition of short cycle has less than $(\ln n)^2$ vertices.  Therefore, a.m.s.\ each strong component outside $V \cap V'$ has size less than $(\ln n)^2$. Further by Corollary \ref{typical A}, q. s. long cycles do exist, they are necessarily  contained in a strong component of size at least $(\ln n)^2$. Hence a.m.s.\ the largest strong component is within $V \cap V'$. To complete the proof, it suffices to show that a.m.s.\ $V \cap V'$ is in fact the vertex set of a strong component. To this end, we want to show that a.m.s.\ between any two long cycles $C_1$ and $C_2$, there are directed paths from $C_1$ to $C_2$ and from $C_2$ to $C_1$. 

Using Karp's correspondence, by the gap theorem, as well as the union bound over all vertices in $D(n, p)$, we have that q.s., for every vertex $v$, the descendant set $d(v)$ and the
ascendant set $a(v)$ have their cardinalities
\begin{equation}\label{eqn: 4.21.1}
|d(v)|, |a(v)| \in [1, (\ln n)^2] \cup [\theta_n n - n^{1/2}\ln n, \theta_n n + n^{1/2}\ln n].
\end{equation}
Given two vertices $v\neq w$,  $A(v,w)$ be the event that $v$ and $w$ are both in long cycles and  satisfy \eqref{eqn: 4.21.1}, but there is no path from $w$ to $v$. $A(v,w)$ is contained in the event that $d(v), a(w) = \theta_n n + O(n^{1/2}\ln n)$ but there is no path from $v$ to $w$. Note that in $D(n, p)$, conditioned on $\mathcal{D}(v)$, the induced subgraph on $[n] \setminus \mathcal{D}(v)$ is distributed as $D(n-|\mathcal{D}(v)|, p)$, without any conditioning. Therefore
\begin{equation*}
P(A | \mathcal{D}(v)) \leq P(a(w) = \theta_n n + O(n^{1/2}\ln n)\text{ in } D(n- |\mathcal{D}(v)|, p) ).
\end{equation*}
Just as we proved Lemma \ref{ancestor 1}, one can easily prove that for $d(v) = \theta_n n + O(n^{1/2}\ln n)$, this latter probability is less than $n^{-a}$ for any fixed $a>0$. The union bound over all pairs of vertices completes this proof. 
\end{proof}
\bigskip\noindent
{\bf Acknowledgement.\/} We express our heartfelt gratitude to Huseyin Acan, Nick Peterson
and Chris Ross for their invaluable feedback during the last years, and to Neil Falkner and
to Matt Kahle, the members of the second author dissertation committee, for their incredible
patience and unflagging support. It is a pleasure to thank Nick for his generous help with Mathematica
and graphics. We thank the referees for their time and effort required to read this paper with such a painstaking care and for providing us with many thoughtful critical remarks and helpful suggestions.

\appendix

\section{Appendix}

\subsection{Approximate Expectations}\label{app: app exp}
Here are the approximate expected values of the 15 pairwise products $a^2, ab, \ldots, k^2$. Note that each of these expected values is zero-degree homogeneous. 

\begin{align*}
\mathcal{E}_{\mathbf{s}}[a^2] & :=  \frac{\nu_i}{\nu_i+\nu_o} + \frac{\nu_o}{\nu_i+\nu_o} \frac{e^{z_i}}{e^{z_i}-1} \frac{\nu_i z_i z_o}{(e^{z_o}-1) \mu} \left(1+\frac{\nu_i z_i z_o}{(e^{z_o}-1)\mu}\right)  \\
\mathcal{E}_{\mathbf{s}}[b^2] & :=  \frac{\nu_o}{\nu_i+\nu_o}+ \frac{\nu_i}{\nu_i+\nu_o} \frac{e^{z_o}}{e^{z_o}-1} \frac{\nu_o z_i z_o}{(e^{z_i}-1) \mu} \left(1+\frac{\nu_o z_i z_o}{(e^{z_i}-1)\mu}\right) \\
\mathcal{E}_{\mathbf{s}}[a b ] & :=  \frac{\nu_i}{\nu_i+\nu_o} \frac{e^{z_o}}{e^{z_o}-1} \frac{\nu_o z_i z_o}{(e^{z_i}-1)\mu} +  \frac{\nu_o}{\nu_i+\nu_o} \frac{e^{z_i}}{e^{z_i}-1} \frac{\nu_i z_i z_o}{(e^{z_o}-1)\mu}  \\
\mathcal{E}_{\mathbf{s}}[a r_i]  &:= \frac{\nu_i}{\nu_i+\nu_o} \frac{e^{z_o}}{e^{z_o}-1} \frac{\nu-\nu_i-\nu_o}{e^{z_i}-1} \frac{z_i z_o}{\mu}   \\
\mathcal{E}_{\mathbf{s}}[ a r_o]  &:=  \frac{\nu_o}{\nu_i+\nu_o} \frac{e^{z_i}}{e^{z_i}-1} \frac{\nu_i z_i z_o}{(e^{z_o}-1)\mu} \frac{(\nu-\nu_i-\nu_o)z_i z_o}{(e^{z_o}-1)\mu}  \\
\mathcal{E}_{\mathbf{s}}[a k]   &:=   \frac{\nu_i}{\nu_i+\nu_o} \frac{z_o e^{z_o}}{e^{z_o}-1}+ \frac{\nu_o}{\nu_i+\nu_o} \frac{e^{z_i}}{e^{z_i}-1} \frac{\nu_i z_i z_o}{(e^{z_o}-1)\mu} (1+z_i) \\
\mathcal{E}_{\mathbf{s}}[b r_i] &:=   \frac{\nu_i}{\nu_i+\nu_o} \frac{e^{z_o}}{e^{z_o}-1} \frac{\nu_o z_i z_o}{(e^{z_i}-1)\mu} \frac{(\nu-\nu_i-\nu_o)z_i z_o}{(e^{z_i}-1)\mu} \\
\mathcal{E}_{\mathbf{s}}[b r_o] &:=  \frac{\nu_o}{\nu_i+\nu_o} \frac{e^{z_i}}{e^{z_i}-1} \frac{\nu-\nu_i-\nu_o}{e^{z_o}-1} \frac{z_i z_o}{\mu}   \\
\mathcal{E}_{\mathbf{s}}[b k]  &:= \frac{\nu_i}{\nu_i+\nu_o} \frac{e^{z_o}}{e^{z_o}-1} \frac{\nu_o z_i z_o}{(e^{z_i}-1)\mu} (1+z_o)+ \frac{\nu_o}{\nu_i+\nu_o} \frac{z_i e^{z_i}}{e^{z_i}-1}  \\
\mathcal{E}_{\mathbf{s}}[r_i^2] &:= \frac{\nu_i}{\nu_i+\nu_o} \frac{e^{z_o}}{e^{z_o}-1} \frac{\nu-\nu_i-\nu_o}{e^{z_i}-1}\frac{z_i z_o}{\mu} \left(1+ \frac{\nu-\nu_i-\nu_o}{e^{z_i}-1} \frac{z_i z_o}{\mu} \right) \\
\mathcal{E}_{\mathbf{s}}[r_i r_o ] &:= 0  \\
\mathcal{E}_{\mathbf{s}}[r_i k ] &:=  \frac{\nu_i}{\nu_i+\nu_o} \frac{e^{z_o}}{e^{z_o}-1} \frac{\nu-\nu_i-\nu_o}{e^{z_i}-1}\frac{z_i z_o}{\mu} \left(1+z_o\right)  \\
\mathcal{E}_{\mathbf{s}}[r_o^2] &:=   \frac{\nu_o}{\nu_i+\nu_o} \frac{e^{z_i}}{e^{z_i}-1} \frac{\nu-\nu_i-\nu_o}{e^{z_o}-1}\frac{z_i z_o}{\mu} \left(1+ \frac{\nu-\nu_i-\nu_o}{e^{z_o}-1} \frac{z_i z_o}{\mu} \right)   \\
\mathcal{E}_{\mathbf{s}}[r_o k]  &:=   \frac{\nu_o}{\nu_i+\nu_o} \frac{e^{z_i}}{e^{z_i}-1} \frac{\nu-\nu_i-\nu_o}{e^{z_o}-1}\frac{z_i z_o}{\mu} \left(1+z_i\right)    \\
\mathcal{E}_{\mathbf{s}}[k^2]  &:=  \frac{\nu_i}{\nu_i+\nu_o} \frac{z_o e^{z_o}}{e^{z_o}-1} \left(1+z_o\right)+ \frac{\nu_o}{\nu_i+\nu_o} \frac{z_i e^{z_i}}{e^{z_i}-1} \left(1+z_i\right) \\
\end{align*}

\subsection{Converting Gaussian Limits from $D(n,m)$ to $D(n,p)$}

The proof of the following lemma is essentially a copy of Lemma 2 in Pittel~\cite{pittel 90}, which allowed transfer of asymptotic normality results in $G(n,m)$ to $G(n,p)$.

\begin{lemma}\label{distributional conversion from dnm to dnp}
Let $k \geq 1$ be fixed, and let $\mathbf{Y}=\mathbf{Y}(D) \in \mathbb{R}^k$. Suppose that there exists a $k$-dimensional vector function $\boldsymbol\mu=\boldsymbol\mu(x) (x>0)$ and a $k \times k$ symmetric matrix function $\mathbf{A}=\mathbf{A}(x)$ such that for $D(n, m=c_n n)$ and any $c_n \to c \in (0, \infty)$, $[\mathbf{Y}-n \boldsymbol\mu(c_n)] n^{-1/2}$ is asymptotically Gaussian with the zero vector of means and the covariance matrix $\mathbf{A}=\mathbf{A}(c)$. Suppose also that $\boldsymbol \mu(x)$ is continuously differentiable, and $\mathbf{A}(x)$ is continuous. Then, for $D(n,p=c_n/n)$,  $(\mathbf{Y}- n \boldsymbol \mu(c_n))n^{-1/2}$ is also asymptotically Gaussian, with zero mean and covariance matrix 
\begin{equation*}
\mathcal{A} = \mathcal{A}(c) = \mathbf{A}(c) + c \boldsymbol \mu'(c) \boldsymbol \mu'(c)^T.
\end{equation*}
\end{lemma}

\begin{proof}
Let $\mathbf{u} \in \mathbb{R}^k$. We need to show that for $\boldsymbol \mu_n=\boldsymbol \mu(c_n)$, $\mathbf{A}=\mathbf{A}(c)$, $\mathcal{A}=\mathcal{A}(c)$,
\begin{equation}\label{eqn: 5.1.3}
\varphi_p(\mathbf{u}) := E_p[ \exp \left( i \mathbf{u}^T (\mathbf{Y}- n \boldsymbol \mu_n) n^{-1/2}\right)] \to \exp \left( - \frac{1}{2} \mathbf{u}^T \mathcal{A} \mathbf{u} \right),
\end{equation}
where $E_p[\circ]$ is the expectation over $D(n,p=c_n/n)$. Since $D(n,p)$ conditioned on the value of present edges $e(D(n,p))=m$ is equal in distribution to $D(n, m)$, we have that
\begin{equation*}
\varphi_p(\mathbf{u}) = \sum_{m=0}^{n(n-1)} {n(n-1) \choose m}p^m (1-p)^{n(n-1)-m} E_m[ \exp \left( i \mathbf{u}^T (\mathbf{Y}- n \boldsymbol \mu)^{-1/2}\right)],
\end{equation*}
where $E_m[\circ]$ is the expectation over $D(n, m)$. 

Fix $a>0$, and we split the sum above into $\Sigma_1$ with those $m$ such that $|m - n(n-1)p| \leq a \sqrt{ n(n-1) p (1-p)}$, and $\Sigma_2$ with those remaining $m$. By the central limit theorem,
\begin{align}\label{eqn: 5.1.2}
\big | \Sigma_2 \big | &\leq \sum_{|m-n(n-1)p| > a \sqrt{n(n-1) p (1-p)}} {n(n-1) \choose m} p^m (1-p)^{n(n-1)-m} \nonumber \\ & \to (2/\pi)^{1/2} \int_a^\infty e^{-z^2/2} dz \hspace{1cm} (n \to \infty).
\end{align}
Further, for all $m$ in the sum $\Sigma_1$, 
\begin{equation*}
|m/n - c_n| = O(n^{-1/2}).
\end{equation*}
Note that
\begin{align*}
 E_m[ \exp \left( i \mathbf{u}^T (\mathbf{Y}- n \boldsymbol \mu)^{-1/2}\right)] = \exp & \left( i \mathbf{u}^T [ \boldsymbol \mu(m/n)-\boldsymbol \mu ] n^{1/2}\right) \\ & \times E_m[ \exp \left( i \mathbf{u}^T [ \mathbf{Y}- n \boldsymbol \mu(m/n)]n^{-1/2} \right) ].
\end{align*}
This second factor approaches $\exp \left( - \frac{1}{2} \mathbf{u}^T \mathbf{A} \mathbf{u}\right)$, while the first term is
\begin{equation*}
\exp \left( i (c_n)^{1/2} [\mathbf{u}^T \boldsymbol \mu'(c_n)] \frac{m-n(n-1)p}{\sqrt{ n(n-1) p (1-p)}} \right) + o(1),
\end{equation*}
uniformly for $m$ in $\Sigma_1$. Therefore, 
\begin{align}\label{eqn: 5.1.1}
\Sigma_1 \to & (2\pi)^{-1/2} \int_{|z| \leq a} \exp \left( i (c)^{1/2} [\mathbf{u}^T \boldsymbol \mu'(c)]z - z^2/2\right) dz \nonumber \\ & \times \exp \left( -\frac{1}{2} \mathbf{u}^T \mathbf{A} \mathbf{u} \right).
\end{align}
Letting $a \to \infty$ in \eqref{eqn: 5.1.2} and \eqref{eqn: 5.1.1}, we get that
\begin{equation*}
\varphi_p(\mathbf{u}) \to \exp \left( - \frac{1}{2} c [ \mathbf{u}^T \boldsymbol \mu'(c)]^2 - \frac{1}{2} \mathbf{u}^T \mathbf{A} \mathbf{u}\right),
\end{equation*}
which is equivalent to \eqref{eqn: 5.1.3}.
\end{proof}

\end{document}